\documentclass[a4paper,12pt]{article}
\usepackage{color}
\usepackage{hyperref}
\usepackage{setspace} 
\usepackage{amsmath,amssymb,amsthm,amscd,amsfonts}
\usepackage{bm}
\usepackage{graphicx}
\usepackage{ascmac}
\usepackage{url}
\usepackage{enumerate}
\usepackage{type1cm}
\usepackage{latexsym}
\usepackage{mathtools}
\usepackage[all]{xy}
\usepackage{titletoc}
\usepackage[utf8]{inputenc}
\usepackage[english]{babel}

\DeclareMathOperator{\spec}{sp}

\DeclareMathOperator{\Cliff}{Cliff}

\DeclareMathOperator{\exte}{ext}
\DeclareMathOperator{\inte}{int}
\DeclareMathOperator{\id}{id}
\DeclareMathOperator{\Ad}{Ad}
\DeclareMathOperator{\op}{op}

\newcommand{\hill}{\mathcal{H}}
\newcommand{\E}{\mathcal{E}}
\newcommand{\cC}{\mathcal{C}}
\newcommand{\cS}{\mathcal{S}}
\newcommand{\cL}{\mathcal{L}}
\newcommand{\cB}{\mathcal{B}}
\newcommand{\cH}{\mathcal{H}}
\newcommand{\BH}{B(\mathcal{H})}
\newcommand{\bbC}{\mathbb{C}}
\newcommand{\bbZ}{\mathbb{Z}}

\newcommand{\bbR}{\mathbb{R}}

\newcommand{\K}{\mathcal{K}}
\newcommand{\A}{\mathfrak{A}}
\newcommand{\T}{\mathfrak{T}}

\newcommand{\Ahat}{\widehat{A}}
\newcommand{\C}{C^\ast}
\newcommand{\Calg}{C^\ast\text{-algebra}}
\newcommand{\Calgs}{C^\ast\text{-algebras}}

\newcommand{\MnC}{M_n(\mathbb{C})}
\newcommand{\s}[1]{\langle #1 \rangle}
\newcommand{\red}{\C_{\text{red}}(G)}

\theoremstyle{definition}
\newtheorem{theorem}[subsection]{Theorem}
\newtheorem*{theorem*}{Theorem}  
\newtheorem{dfn}[subsection]{Definition}
\newtheorem*{dfn*}{Definition}
\newtheorem{prop}[subsection]{Proposition}

\newtheorem{lemma}[subsection]{Lemma}
\newtheorem{cor}[subsection]{Corollary}
\newtheorem{example}[subsection]{Example}

\newtheorem{conj}[subsection]{Conjecture}

\begin{document}

\title{On the Lifting of the Dirac Elements\\ in the Higson-Kasparov Theorem}
\author{Shintaro Nishikawa\\Keio University}
\date{September 2016}

\maketitle
\thispagestyle{empty}
\begin{abstract} In this thesis, we investigate the proof of the Baum-Connes Conjecture with Coefficients for a-$T$-menable groups. We will mostly and essentially follow the argument employed by N. Higson and G. Kasparov in the paper \cite{HigKas2}. The crucial feature is as follows. One of the most important point of their proof is how to get the Dirac elements (the inverse of the Bott elements) in Equivariant $KK$-Theory. We prove that the group homomorphism used for the lifting of the Dirac elements is an isomorphism in the case of our interests. Hence, we get a clear and simple understanding of the lifting of the Dirac elements in the Higson-Kasparov Theorem. In the course of our investigation, on the other hand, we point out a problem and give a fixed precise definition for the non-commutative functional calculus which is defined in the paper \cite{HigKas2}. In the final part, we mention that the $\Calg$ of (real) Hilbert space becomes a $G$-$\Calg$ naturally even when a group $G$ acts on the Hilbert space by an affine action whose linear part is of the form an isometry times a scalar and prove the infinite dimensional Bott-Periodicity in this case by using Fell's absorption technique.
\end{abstract}

\newpage
\setcounter{page}{2}
\tableofcontents

\newpage
\section*{Acknowledgement}\markboth{Acknowledgement}{Acknowledgement} I would like to thank, first and foremost, my advisor, Takeshi Katsura who always gives me great insights all of which help me to overcome difficulties I have.  At the time when this thesis was written, two and a half years passed since he introduced me a beautiful realm of Functional Analysis, Operator Algebras and Noncommutative Geometry where the infinite dimensionality makes such a wonderful interplay between topology, analysis and algebra. Since then, his way of seeing and doing mathematics has influenced mine a lot in a good way. It is his constant support which made it possible for me to get such wonderful understandings of these subjects some of which are explored in this thesis. Also, here, I would like to thank Nigel Higson who kindly answered to my questions and gave me many helpful comments on my research. I am deeply indebted to his great hospitality during my visit to the Pennsylvania State University in December 2015. Finally, I would like to thank Narutaka Ozawa and Klaus Thomsen who kindly responded to my questions in their busy schedule. 
\addcontentsline{toc}{section}{\hspace{40pt}Acknowledgement}
\newpage
\section*{Introduction}\markboth{Introduction}{Introduction} The Baum-Connes Conjecture (Conjecture \ref{thm:BC}) is a long-standing conjecture in non-commutative geometry. It does have deep relations with other fields of mathematics; the Novikov conjecture in topology and the idempotent conjecture in algebra are famous examples of conjectures which the Baum-Connes Conjecture implies. Since it was formulated in 1982 by Baum and Connes, there has been outstandingly great progress in understanding and verification of this conjecture. For a second countable, locally compact topological group $G$, the reduced group $\Calg$ $\C_{\text{red}}(G)$ of $G$ is defined as the completion of the convolution algebra $L^1(G)$ acting on the Hilbert space $L^2(G)$ of square integrable functions on $G$. The set of unitary equivalence classes of irreducible representations of the $\Calg$ $\red$ correspond bijectively to that of irreducible  unitary representations of the group $G$ which are weakly contained in the (left) regular representation of $G$; this set is the reduced unitary dual $\hat{G}_r$. When $G$ is a compact or an abelian group, the natural topology defined on $\hat{G}_r$ is locally compact and Hausdorff. However, for a general group $G$, the topology on $\hat{G}_r$ may not be Hausdorff. The $K$-theory $K_\ast(\red)$ of the $\C$-algebra $\red$ can be considered as one of the tools for properly describing the geometric nature of the ``space'' $\hat{G}_r$. On the other hand, Kasparov (\cite{Kas2}) generalized the index theory of elliptic operators on smooth manifolds to develop the bivariant theory of $\Calgs$: the equivariant $KK$-theory. This beautiful generalization of the index theory  achieved to define not only the notion of abstract elliptic operators which induce the group homomorphisms on $K$-theory groups of $\Calgs$ but also the well-defined product of two elliptic operators so that it is compatible with the composition of group homomorphisms they induce; this is the Kasparov product. Kasparov and others managed to define the (higher) indices of elliptic operators taking values in the groups $K_\ast(\red)$. The Baum-Connes Conjecture states that all elements of the $K$-theory groups $K_\ast(\red)$ should be indices of some elliptic operators and that any two elliptic operators having same indices should be linked by certain geometric relations (i.e.\ homotopies). 

N. Higson and G. Kasparov (\cite{HigKas2}) showed that the Baum-Connes Conjecture holds for all a-$T$-menable groups (Definition \ref{dfn:a-t}), in particular for all amenable groups. Actually, what they proved is that they satisfy the Baum-Connes Conjecture with Coefficients (Conjecture \ref{thm:BCC}) which is a much stronger conjecture than the Baum-Connes Conjecture. This is the Higson-Kasparov Theorem (Theorem \ref{thm:BCCa-T}). They proved this result following the Dual-Dirac method (Theorem \ref{DD}), the standard method used for proving the Baum-Connes Conjecture with Coefficients which says the Baum-Connes Conjecture with Coefficients holds for a group $G$ if one finds an isomorphism between the $\Calg$ $\bbC$ and some proper $G$-$\Calg$ in Equivariant Kasparov's category $KK^G$. For an a-$T$-menable group $G$, there is a natural candidate of this isomorphism which is called the Bott element. However, as is described in the paper \cite{HigKas2}, there is a certain analytic technicality in finding the inverse of the Bott element, the Dirac element. N. Higson and G. Kasparov defined for separable $G$-$\Calgs$ $A,B$, an abelian group $\{\Sigma A, B\}_G$ and a group homomorphism $\eta$ from the odd Kasparov group $KK^G_1(A, B)$ to $\{\Sigma A, B\}_G$ ($\Sigma A=C_0(0, 1)\otimes A$). They defined a Dirac element $\alpha$ in the group $\{\Sigma A(\hill), S\Sigma\}_G$  where $A(\hill)$ is a certain proper $G$-$\Calg$ and $S=C_0(\bbR)$. They managed to find the ``honest'' Dirac element $d$ by showing that we can lift $\alpha$ to $d$ by $\eta$. Their proof of this lifting (\cite{HigKas2} Theorem 8.1.) contains a very technical argument concerning the extension of $G$-$\Calgs$ having a not necessarily equivariant completely positive cross section. In this thesis, among the other things, we prove the following result:

\begin{theorem}\label{Theo}(See Theorem \ref{Result}) Let $A,B$ be separable $G$-$\Calgs$. Suppose that $A$ is a nuclear, proper $G$-$\Calg$ and that $B$ is isomorphic to $\Sigma B'$ for some separable $G$-$\Calg$ $B'$. Then, the homomorphism $\eta\colon KK^G_1(A, B)\to \{\Sigma A, B\}_G$ is an isomorphism of abelian groups.
\end{theorem}

Thanks to this result, we can avoid the technical theorem (\cite{HigKas2} Theorem 8.1.) in defining the Dirac element in Equivariant Kasparov's category $KK^G$.

The brief description of this thesis is as follows. Chapter 1 serves as a very quick introduction of $\Calgs$ for readers who might not be familiar with these notions. Chapter 2 contains further preliminary materials which are used in later chapters such as graded $\Calgs$, Hilbert-modules and unbounded multipliers. The functional calculus for unbounded multipliers is explained using the Bott-Dirac operator which plays the important role in the proof of the Higson-Kasparov Theorem. In Chapter 3, we give a basic introduction of $K$-Theory and $K$-Homology of $\Calgs$ and go on to introducing Kasparov's Equivariant KK-Theory in Chapter 4 and Equivariant E-Theory in Chapter 5. We confine ourselves to see only,  necessary facts for our investigation of the proof of the Higson-Kasparov Theorem. In Chapter 6, we quickly see the standard formalization of the Baum-Connes Conjecture and the Baum-Connes Conjecture with Coefficients. In this chapter, we also introduce the Higson-Kasparov Theorem and give a brief review of the proof given by N. Higson and G. Kasparov. In Chapters 7 and 8, we give a proof of the Higson-Kasparov Theorem following the argument employed by N. Higson and G. Kasparov. In Chapter 7, we point out a certain problem of the non-commutative functional calculus defined by N. Higson and G. Kasparov and give a fixed precise definition. In Chapter 8, among the other things, we show our main result (Theorem \ref{Theo}) which says that the group homomorphism used for the lifting of the Dirac elements is in fact, an isomorphism in the case of our interests. This gives us a clear and simple understanding of the technical part of the Higson-Kasparov Theorem. In the final Chapter, we mention that  the $\Calg$ of Hilbert space becomes a $G$-$\Calg$ naturally even when a group $G$ acts on the Hilbert space by an affine action whose linear part is not necessarily isometric but of the form an isometry times a scalar, and prove the infinite dimensional Bott-Periodicity in this case by using the Fell's absorption technique.

\addcontentsline{toc}{section}{\hspace{40pt}Introduction}
\newpage
\pagenumbering{arabic}
\section{$\C$-Algebras}

In this first chapter, we give a basic introduction of $\Calgs$. Materials given here can be found in many textbooks of this subject such as \cite{BrOza}, \cite{DixC}, \cite{HigRoe} and \cite{analysisnow}.

\begin{dfn} A complex algebra $A$ is a {Banach algebra} (resp.\ {normed algebra}) if its underlying vector space is a Banach space (resp.\ normed vector space) with a norm which is submultiplicative (i.e.\ $\|xy\|\leq\|x\|\|y\|$ for all $x,y \in A$).

An {involution} on a normed algebra $A$ is a conjugate-linear antimultiplicative isometry of order two, denoted $x \mapsto x^\ast$ for $x\in A.$

A {Banach {$\ast$}-algebra} (resp. {normed {$\ast$}-algebra}) is a Banach algebra (resp. normed algebra) with an involution. 

A {{$\Calg$}} is a Banach $\ast$-algebra $A$ satisfying {{$\C$-identity}}:
\begin{align}
\|x^\ast x\| &= \|x\|^2 \quad \text{for all} \quad x\in A
\label{eq:identity}
\end{align}

A normed algebra is called {separable} if it is separable in a topological sense, i.e. if it has a countable dense subset. 

A normed algebra $A$ is called {unital} if it has a unit (a multiplicative identity) usually denoted $1$ or $1_A$. A {unital} subalgebra of $A$ is a subalgebra of $A$ containing the unit $1_A.$  

A {{$\C$-subalgebra}} is a norm-closed selfadjoint (closed under the involution) subalgebra of a $\C$-algebra. It is a $\C$-algebra in an obvious way.

Let $A$ and $B$ be normed $\ast$-algebras. A {{$\ast$}-homomorphism} from $A$ to $B$ is an algebraic homomorphism from $A$ to $B$ which intertwines the involutions. An isomorphism of normed $\ast$-algebras is a surjective isometric $\ast$-homomorphism.

For any Banach $\ast$-algebra (resp.\ $\Calg$) $A$, there is a unital Banach ${\ast}$-algebra (resp.\ $\Calg$) $\tilde A$ containing $A$ as a subalgebra of codimension one. Its algebraic structure is unique, but it may be defined several norms on $\tilde A$. If $A$ is a $\C$-algebra, it will soon be clear that a $\Calg$ $\tilde A$ is unique up to isomorphisms. For a non-unital algebra $A$, $\tilde A$ is called a {unitization} of $A.$ 

Let $A$ be a unital Banach algebra. For $a \in A$, the {spectrum} $\spec_A(a)$ of $a$ in $A$ is a subset of $\mathbb{C}$ defined by $\spec_A(a)=\{ \, \lambda \in \mathbb{C} \mid \text{$\lambda - a$ is not invertible in $A$}\, \}$. The spectrum $\spec_A(a)$ is a nonempty compact subset of $\mathbb{C}$. If $A$ is a unital subalgebra of a unital Banach algebra $B$, $\spec_A(a)$ and $\spec_B(a)$ may not coincide in general. Fortunately, if $B$ is a $\C$-algebra and $A$ is its unital $\C$-subalgebra, they can be shown to be the same. Henceforth, we can speak about the spectrum of $a$ in a $\C$-algebra $A$ without any confusion; we will denote it by $\spec(a)$ (for $a$ in a non-unital $\Calg$ $A$, $\spec(a)$ is defined to be $\spec_{\tilde A}(a)$). 
\end{dfn}

Any $\ast$-homomorphism from a Banach $\ast$-algebra to a $\C$-algebra is bounded (continuous). In fact, it is always norm decreasing. Therefore, any bijective $\ast$-homomorphism between $\C$-algebras is  automatically an isomorphism. This explains the uniqueness of a unitization of a $\C$-algebra.

\begin{dfn} Let $A$ be a $\C$-algebra.
\begin{itemize}
\item[(i)] $a \in A$ is {normal} if $a^\ast a=aa^\ast$;
\item[(ii)] $a \in A$ is {selfadjoint} if $a=a^\ast$; 
\item[(iii)] $a \in A$ is {positive} if $a=b^\ast b$ for some $b \in A$; 
\item[(iv)] $p \in A$ is a {projection} if $p=p^\ast=p^2$;
\item[(v)] $w \in A$ is a {partial isometry} if $w^\ast w$ and $ww^\ast$ are projections.

Assume $A$ is unital.
\item[(vi)] $v \in A$ is an {isometry} if $v^\ast v=1$;
\item[(vii)] $u \in A$ is a {unitary} if $u^\ast u=uu^\ast=1.$
\end{itemize}
\end{dfn}

The set of positive elements in a $\C$-algebra $A$ forms a cone. We define an order for selfadjoint elements in $A$ in the following way. For  selfadjoint elements $a,b \in A$, $a\leq b$ if $b-a$ is positive.

\begin{example} Let $\hill$ be a complex Hilbert space, i.e.\ a complex Banach space whose norm is coming from an inner product $\langle \cdot,\cdot \rangle \colon \hill \times \hill \to \mathbb{C}$ (we will always take it to be linear in the second variable). A linear operator $T$ on a normed vector space is continuous if and only if it is uniformly bounded on the unit ball; we call such $T$ a bounded operator. The algebra $\BH$ of bounded operators on $\hill$ is a $\C$-algebra in the following way. We consider the {operator norm} $\|T\|=\displaystyle \sup_{\|\xi\|=1}\|T\xi\|.$ There is an involution $T\mapsto T^\ast$, where $T^\ast$ for a bounded operator $T$ is the unique bounded operator on $\hill$ satisfying $\s{T\xi,\eta}=\s{\xi,T^\ast\eta}$ for all $\xi,\eta \in \hill.$ Equipped with them, $\BH$ becomes a $\C$-algebra. If $\dim(\hill)=n<\infty$, $\BH$ can be identified with a matrix algebra $\MnC$ uniquely up to inner automorphisms. In this paper, $\MnC$ will be almost all cases treated as $\C$-algebras endowed with the canonical operator norm.
\end{example}

$\C$-subalgebras of $\BH$ are sometimes called concrete $\C$-algebras. It turns out any abstract $\C$-algebra is isomorphic to a concrete one (See Proposition \ref{prop:cstarrep}.) In other words, $\C$-identity \eqref{eq:identity} encodes all the necessary and sufficient informations for Banach $\ast$-algebras to be realized as ``operator algebras'' on Hilbert spaces. All the definitions above about particular elements of $\C$-algebras reflect the corresponding notions defined for operators on Hilbert spaces.

\begin{example} (commutative $\C$-algebras) Let $X$ be a locally compact Hausdorff space. The algebra $C_b(X)$ of bounded continuous $\mathbb{C}$-valued functions on X becomes a $\C$-algebra in the following way. The norm is supremum norm $\displaystyle \|f\|=\sup_{x\in X}|f(x)|$; and the involution is a pointwise complex conjugation $f\mapsto \overline{f}.$ Inside $C_b(X)$, there is a normed $\ast$-algebra $C_c(X)$ of continuous functions on $X$ with compact supports; its completion $C_0(X)$ in $C_b(X)$ is a $\C$-subalgebra of $C_b(X)$; it is identified with the algebra of continuous functions on $X$ vanishing at infinity. The $\Calg$  $C_0(X)$ is unital if and only if $X$ is compact, and in this case we usually denote it by $C(X).$ The algebra  $C_0(X)$ is separable if and only if $X$ is second countable. 
\end{example}
Any commutative $\C$-algebra is canonically isomorphic to $C_0(X)$ for some $X$.

\begin{prop} (cf. \cite{HigRoe} THEOREM 1.3.12) Let $A$ be a commutative $\C$-algebra. Denote by $\Ahat$, the space of characters of $A$ (nonzero $\ast$-homomorphisms from $A$ to $\mathbb{C}$) equipped with the weak-$\ast$ topology (pointwise convergence topology). Then $\Ahat$ is a locally compact Hausdorff space; and is compact if and only if $A$ is unital. The $\C$-algebra $A$ is isomorphic to $C_0(\Ahat)$; the isomorphism sends $a$ in $A$ to a function $\hat{a}\colon \psi \mapsto \psi(a).$ The character space $\Ahat$ is called Gelfand spectrum of $A.$   
\end{prop} 

For a normal element $a$ in a unital $\C$-algebra $A$, denote by $C^\ast(a,1)$ the minimal unital $\C$-subalgebra of A containing $a.$ This is a unital commutative $\C$-algebra isomorphic to $C(\spec(a))$; the canonical (unital) isomorphism takes $a \in C^\ast(a,1)$ to the coordinate function $z \in C(\spec(a)).$ For any continuous function $f \in C(\spec(a)),$ we denote by $f(a)$ the corresponding element in $C^\ast(a,1).$ This correspondence is called {functional calculus.} More generally, for a normal element $a$ in any $\C$-algebra, the minimal $\C$-subalgebra  $C^\ast(a)$ containing $a$ is canonically isomorphic to $C_{0}(\spec(a)\backslash \{0\})$. There is analogous functional calculus for this possibly non-unital situation.

\begin{dfn} For a normed $\ast$-algebra $A$, a {representation} of $A$ on a Hilbert space $\hill$ is a bounded $\ast$-homomorphism from $A$ to $\BH.$ Two representations $\rho_1$ on $\hill_1$ and $\rho_2$ on $\hill_2$ are {unitary equivalent} if there exists a unitary (an isomorphism of Hilbert spaces) $U$ from $\hill_1$ to $\hill_2$ which intertwines the two representations:
\begin{align*}
U\rho_1(a) U^\ast &= \rho_2(a) \quad \text{for} \,\, a \in A
\end{align*}
A representation $\rho$ of $A$ on $\hill$ is called {nondegenerate} if $\rho(A)\hill=$ span$\{\,\rho(a)v \mid a\in A, v\in \hill \,\}$ is dense in $\hill.$ 
\end{dfn}

\begin{prop}
\label{prop:cstarrep} (cf. \cite{HigRoe} THEOREM 1.6.2) The following are equivalent for a normed $\ast$-algebra $A.$
\begin{itemize}
\item[(i)] $A$ is a $\C$-algebra;
\item[(ii)] $A$ is isomorphic to a $\C$-subalgebra of $\BH$ for some Hilbert space $\hill.$ 
\end{itemize}
\end{prop}
\begin{proof} The proof comes down to constructing for each selfadjoint element $a$ in $A$, a representation of $A$ which sends $a$ to a nonzero element. This is done by an elaboration of the Hahn-Banach Theorem and the GNS-construction. We omit the detail; see \cite{HigRoe} for example.
\end{proof}

For any $\Calg$ $A$, the matrix algebra $M_n(A)$ over $A$ becomes a $\Calg$ in the following way. One first identifies $A$ as a $\C$-subalgebra of $\BH$ by faithfully representing $A$ on some Hilbert space $\hill$. Then, $M_n(A)$ is naturally identified with a $\C$-subalgebra of $B(\hill^n)$. The $\C$-norm defined on $M_n(A)$ in this way is independent of representations of $A$. 

\begin{example} Let $A$ be a Banach $\ast$-algebra. There is a canonical pre-$\C$-norm on $A$ defined by:
\begin{align}
\|a\| &= \displaystyle \sup_{\rho}\|\rho(a)\|
\label{eq:envelop norm}
\end{align}
where the supremum is taken over all representations of $A.$ Since any representation of $A$ is norm decreasing as we remarked earlier, this norm is well-defined. It satisfies $\C$-identity \eqref{eq:identity} because it comes from the operator norm on Hilbert spaces. The completion of $A$ with this new norm (after taking a quotient by ``zero'' elements) is the {enveloping {$\C$-algebra}} of $A.$ By its construction, it has a universal property that representations of $A$ correspond bijectively to representations of the enveloping $\C$-algebra of $A$.  
\end{example}

Let's see some examples of this construction. The last one is the most important; the first two are described just for seeing what kinds of properties of Banach $\ast$-algebras make them far away from being $\C$-algebras.

\begin{itemize}
\item[(1)] Consider a subalgebra $A$ (as a Banach algebra) of $M_2(\mathbb{C})$ consisting of upper triangular matrices. Define an involution on $A$ by the following formula.
\begin{equation*}
\begin{pmatrix}
a & b \\
0 & c
\end{pmatrix}
^\ast=
\begin{pmatrix}
\bar{c} & \bar{b} \\
0 & \bar{a}
\end{pmatrix}
\quad \text{for} \,\, a,b,c \in \mathbb{C} 
\end{equation*}
One can check that this makes $A$ a Banach $\ast$-algebra and that its enveloping $\C$-algebra is $0.$  The reason of vanishing of all elements is clear: all the three basic coordinate vectors satisfy $a^\ast a = 0$ which implies $a=0$ in $\C$-algebras.  
\item[(2)]Denote the closed unit disk of the complex plane $\mathbb{C}$ by $\mathbb{D}$. Consider a subalgebra $A$ (as a Banach algebra) of $C(\mathbb{D})$ consisting of bounded holomorphic functions on $\mathbb{D}$. Define a new involution on $A$ by $f^\ast(z)=\overline{f(\overline{z})}.$ This makes it a commutative Banach $\ast$-algebra; and its enveloping $\C$-algebra is $C([-1,1])$ of continuous functions on the interval $[-1,1].$ To check this, one can first see the image of the coordinate function $z$ and the identity (the constant function $1$) generate the enveloping algebra and note that since $z$ is selfadjoint in $A$, the spectrum of its image must be contained in $\mathbb{D}\cap\mathbb{R}=[-1,1].$  

\item[(3)](full group $\C$-algebras) Let $G$ be a locally compact topological group. We denote by $\mu$ its left invariant Haar measure which is unique up to scalar multiplication. Let $\Delta$ be the associated modular function. Consider a Banach space $L^1(G,\mu)$ of integrable functions. We define a product and an involution to make it a Banach $\ast$-algebra: for $f,g \in L^1(G,\mu)$
\begin{align*}
(fg)(t)&=\int f(s)g(s^{-1}t)d\mu(s) \\
(f^\ast)(t)&=\Delta(t)^{-1}\overline{f(t^{-1})}
\end{align*}
The enveloping $\C$-algebra of $L^1(G,\mu)$ is the (full) {group {$\C$}-algebra} $\C(G)$ of a locally compact topological group $G.$ This $\C$-algebra has an important universal property that associated to any unitary representation of $G$ on a Hilbet space, the canonical representation of $C_c(G)$ extends continuously (hence uniquely) to $\C(G);$ here one may identify $C_c(G)$ as a subalgebra of $\C(G)$ not just of $L^1(G,\mu).$ Conversely, any nondegenerate representation of $\C(G)$ arises in this way and uniquely determines the underlying unitary representation of $G.$ The $\Calg$ $\C(G)$ is commutative if and only if $G$ is abelian; and in this case, it is isomorphic to $C_0(\widehat{G})$ where $\widehat{G}$ is the character space of $G$ which is locally compact in its own right. The group $\Calg$ $\C(G)$ is separable if $G$ is second countable. 
\end{itemize}

\begin{dfn} Let $G$ be a locally compact topological group. Associated to the left regular representation of $G$ on $L^2(G,\mu)$, we have the canonical representation of $\C(G)$. The image of this representation is the {reduced group {$\Calg$}} of $G$; it is denoted by $\C_{\text{red}}(G)$.
\end{dfn}

\begin{dfn} Let $A$ be a $\Calg$ and $G$ be a locally compact topological group. A $G$-action on $A$ is a group homomorphism from $G$ to the automorphism group $\mathrm{Aut}(A)$ of $A$.  An element $a$ in $A$ with a $G$-action is {{$G$}-continuous} if the map $g \mapsto g\cdot a$ is a continuous map from $G$ to $A$. A  $\Calg$ $A$ with a $G$-action is called a {{$G$}-{$\Calg$}} if all elements in $A$ are $G$-continuous. A $\ast$-homomorphism between $\Calgs$ with $G$-action is called $G$-equivariant or simply, equivariant if it intertwines the two $G$-actions. Note, an equivariant $\ast$-homomorphism necessarily sends $G$-continuous elements to $G$-continuous elements.
\end{dfn}

A $G$-Hilbert space is a Hilbert space $\hill$ with a unitary representation of $G$.  A representation of $G$-$\Calg$ $A$ on a $G$-Hilbert space $\hill$ is a $\C$-algebraic representation $\rho$ of $A$ on $\hill$ which satisfies the following additional condition: for any $a$ in $A$ and for any $g$ in $G$, \,$\rho(g\cdot a)=u_g\rho(a)u_g^\ast$. Here, $u_g$ denotes the unitary on $\hill$ corresponding to $g$ in $G$.

\begin{dfn} Let $G$ be a locally compact group and $A$ be a $G$-$\Calg$. Consider a Banach space $L^1(G,A)$ of integrable functions from $G$ to $A$. We define a product and an involution to make it a Banach $\ast$-algebra: for $f,g \in L^1(G,A)$
\begin{align*}
(fg)(t)&=\int f(s)s(g(s^{-1}t))d\mu(s) \\
(f^\ast)(t)&=\Delta(t)^{-1}t(f(t^{-1}))^\ast
\end{align*}
The enveloping $\C$-algebra of $L^1(G,A)$ is the {full crossed product} $\C_{\text{max}}(G,A)$ of $A$ by $G$. It has a universal property that associated to any representation of a $G$-$\Calg$ $A$ on a $G$-Hilbert space $\hill$, the canonical representation of $C_c(G,A)$ extends continuously (hence uniquely) to a representation of $\Calg$ $\C_{\text{max}}(G,A)$. Conversely, any nondegenerate representation of $\C_{\text{max}}(G,A)$ arises in this way. 
\end{dfn}

\begin{dfn} Let $G$ and $A$ be the same as above. Represent $A$ faithfully and nondegenerately on a Hilbert space $\hill$. Then, the Hilbert space $L^2(G,\hill)$ becomes a $G$-Hilbert space by means of the left regular representation. There is a canonical representation of $G$-$\Calg$ $A$ on $L^2(G,\hill)$. The image of the associated representation of $\C_{\text{max}}(G,A)$ is the {reduced crossed product} $\C_{\text{red}}(G,A)$ of $A$ by $G$.
\end{dfn}

If $A$ is a commutative $G$-$\Calg$ $C_0(X)$ of continuous functions which vanish at infinity on a locally compact space $X$ equipped with a continuous $G$-action, we usually denote the full (resp.\ reduced) crossed product algebra by $\C_{\text{max}}(G,X)$ (resp. $\C_{\text{red}}(G,X)$).

\begin{dfn} Let $A$ be a $\C$-algebra. A (countable) {approximate unit} for $A$ is an increasing sequence $(u_n)_{n\geq 1}$ of positive contractible elements (contractive means having the norm no more than $1$) in $A$ such that for all $a \in A$, $\|a-u_na\|\rightarrow 0$ as $n \rightarrow \infty.$ A continuous approximate unit for $A$ is a family $(u_t)_{t\geq1}$ of (not necessarily increasing) positive contractive elements in $A$ such that for all $a \in A$, $\|a-u_ta\|\rightarrow 0$ as $t \rightarrow \infty.$
\end{dfn}
There is a net version of approximate units; and any $\C$-algebra has an approximate unit in this sense. A $\C$-algebra having a countable approximate unit is called {{$\sigma$}-unital.} Separable $\C$-algebras are $\sigma$-unital. A $\Calg$ has a continuous approximate unit if and only if it is $\sigma$-unital. 

\begin{dfn} Let $J$ be a closed selfadjoint ideal of a $\Calg \,\,A$ (selfadjointness actually follows from the other conditions). Then, the quotient algebra $A/J$ naturally becomes a $\Calg$. In this paper, by an ideal of a  $\Calg$ $A$, we mean a closed selfadjoint ideal of $A$.
\end{dfn}
Associated to an ideal $J$ of $A$, we have a short exact sequence:
\begin{align}
\xymatrix{
0 \ar[r] & J \ar[r] & A \ar[r] & A/J \ar[r] & 0
}
\label{def:extension}
\end{align}
We usually call a short exact sequence \eqref{def:extension} an extension of $A/J$ by $J$. When all $\Calgs$ which appear in \eqref{def:extension} are $G$-$\Calg$ and all connecting $\ast$-homomorphisms are equivariant, then we call it an $G$-extension.

An ideal $J$ of $A$ is called essential if the annihilator ideal \[J^\perp=\{\, a\in A\mid \text{$aj=ja=0$ for all $j\in J$} \,\} \]of $J$ in $A$ is $0$. For a $\Calg$ $A$, the {multiplier algebra} $M(A)$ of $A$ is a $\Calg$ containing $A$ as an essential ideal and maximal among such in the following sense. For any $\Calg$ $B$ containing $A$ as an ideal, there is a unique $\ast$-homomorphism from $B$ to $M(A)$ which is identity on $A$ and has kernel $A^\perp$. The multiplier algebra $M(A)$ can be defined for example, after faithfully and nondegenerately representing $A$ on a Hilbert space $\hill$, as an idealizer $\{\,T\in B(\hill)\mid \text{for any $a\in A$, $Ta, \,aT\in A$ } \,\}$ of $A$ in $B(\hill)$. When $A$ is a $G$-$\Calg$, the $G$-action extends to the natural $G$-action on the multiplier algebra $M(A)$. The quotient algebra $M(A)/A$ is called the outer multiplier algebra of $A$.

\begin{example} An operator $T$ on a Hilbert space is {compact} if it is a norm-limit of finite rank operators. The set of compact operators on a separable infinite dimensional Hilbert space $\hill$ forms an ideal of $\BH$. We denote it by $\K(\hill)$ or simply by $\K$ when there is no confusions. Calkin algebra $Q(\hill)$ or simply $Q$ is the quotient of $\BH$ by $\K$.
\end{example}

\begin{dfn} Suppose we have a system $(A_\lambda)_{\lambda\in\Lambda}$ of $\Calgs$ indexed by an upward filtering set $\Lambda$ with connecting $\ast$-homomorphisms $\phi_{\lambda_1\lambda_2}\colon A_{\lambda_1}\to A_{\lambda_2}$ for $\lambda_1\leq\lambda_2$ satisfying $\phi_{\lambda_2\lambda_3}\circ\phi_{\lambda_1\lambda_2}=\phi_{\lambda_1\lambda_3}$ for $\lambda_1\leq\lambda_2\leq\lambda_3$. We assume all connecting maps are injective and $\phi_{\lambda\lambda}=\id_{A_\lambda}$. An {inductive limit} $\displaystyle \lim_{\lambda\in\Lambda}A_\lambda$ of $(A_\lambda)_{\lambda\in\Lambda}$ is defined as the completion of an algebraic inductive limit of $(A_\lambda)_{\lambda\in\Lambda}$ which can be viewed as the union $\displaystyle\cup_{\lambda\in\Lambda}A_\lambda$ with the obvious pre-$\C$-norm.
\end{dfn}

A tensor product of $\Calgs$ is a very complicated notion. It is defined as a completion of an algebraic tensor product of $\Calgs$ by a $\C$-norm. Surprisingly, such a completion is not unique in general. The following two $\C$-tensor products are standard and very important. The detail of these definitions, for example, the definition of a tensor product of Hilbert spaces can be found in \cite{BrOza}. 

\begin{dfn} Let $A$ and $B$ be $\Calgs$. The {maximal tensor product} $A\otimes_{\text{max}}B$ of $A$ and $B$ is the completion of an algebraic tensor product $A\odot B$ by the following $\C$-norm:
\begin{align*}
\|x\| &= \displaystyle \sup_{\rho}\|\rho(x)\| \,\, \text{for $x$ in $A\odot B$}
\end{align*}
Here, the supremum is taken for all (algebraic) representation of the $\ast$-algebra $A\odot B$. The maximal tensor product $A\otimes_{\text{max}}B$ has a universal property that for any pair of commuting $\ast$-homomorphisms from $A$ and from $B$ to a $\Calg$ $C$, it ``extends'' uniquely to a $\ast$-homomorphism from $A\otimes_{\text{max}}B$ to $C$.

The {minimal tensor product} $A\otimes_{\text{min}}B$ or simply $A\otimes B$ of $A$ and $B$ is defined by the following way. We first faithfully represent $A$ and $B$ on Hilbert spaces $\hill_1$ and $\hill_2$. Then, an algebraic tensor product $A\odot B$ is realized as a $\ast$-subalgebra of $B(\hill_1\otimes\hill_2)$, where $\hill_1\otimes\hill_2$ is a tensor product of Hilbert spaces. We take $A\otimes_{\text{min}}B$ to be the completion of $A\odot B$ inside $B(\hill_1\otimes\hill_2)$. It is independent of choices of representations.
\end{dfn}

\begin{example} Let $A$ be a $\Calg$ and $X$ be a locally compact Hausdorff space. The algebra $C_0(X,A)$ (also denoted by $A(X)$) of $A$-valued continuous functions on $X$ which vanish at infinity naturally becomes a $\Calg$. There is a canonical isomorphism from the tensor product $A\otimes C_0(X)$ to $A(X)$ which sends an elementary tensor $a\otimes f$ to a function $x\mapsto f(x)a$. When $X$ is an interval, say $(0,1)$, then we further simply write $A(X)$ by $A(0,1)$.
\end{example}

The {nuclearity} of $\Calgs$ is a very fundamental notion in $\Calg$ theory. We refer to \cite{BrOza} for a very detailed account for this class of $\Calgs$. Here, we note the following important facts. For a nuclear $\Calg$ $A$, the maximal tensor product $A\otimes_{\text{max}}B$ and the minimal tensor product $A\otimes B$ coincide for any $\Calg$ $B$. All commutative $\Calgs$ are nuclear. A direct sum and an inductive limit of nuclear $\Calgs$ are nuclear. The minimal (maximal) tensor product of nuclear $\Calgs$ is nuclear.

\newpage
\section{Further Preliminaries}
In this chapter, we give a further preparation needed for the discussions in the following chapters. The contents of this chapter include proper $\Calgs$, graded Hilbert spaces, graded $\Calgs$, Hilbert modules, continuous fields of Hilbert spaces, continuous fields of $\Calgs$, unbounded operators on a Hilbert space and unbounded multipliers on a Hilbert module. In this chapter, $G$ always denote a second countable, locally compact topological group.

\begin{dfn} A second \,countable,\ locally \ compact, Hausdorff \ topological \ space equipped with a $G$-action $G\times X\to X$ is called a $G$-space. A $G$-space $X$ is called a proper $G$-space if the map $G\times X \ni (g, x) \to (gx, x) \in X\times X$ is proper (i.e. the inverse image of any compact set is compact). A separable $G$-$\Calg$ $A$ is a {proper} $G$-$\Calg$ if, for some second countable, locally compact proper $G$-space $X$, there exists an equivariant $\ast$-homomorphism from the $G$-$\Calg$ $C_0(X)$ to the center of the multiplier algebra $Z(M(A))$ of $A$ such that $C_0(X)A$ is dense in $A$. We denote by $A_c(X)$ the (frequently non-complete) subalgebra $C_c(X)A$ of $A$.
\end{dfn}

\begin{dfn} Let $X$ be a proper $G$-space. A cut-off function $c$ on $X$ is a bounded, non-negative continuous function on $X$ satisfying the following conditions. First, for any compact subset $K$ of $X$, there exists a compact subset $L$ of $G$ such that $(gc)(x)=c(g^{-1}x)=0$ for any $x$ in $K$ and for any $g$ outside $L$: in other words, a map $g\mapsto (gc)f$ from $G$ to $C_c(X)\subset C_0(X)$ has compact support for any $f$ in $C_c(X)$. Secondly $\int_G(gc)(x)^2d\mu=1$ for all $x$ in $X$. 
\end{dfn}

A cut-off function exists for any proper $G$-space $X$; this may be constructed as follows. Our assumption on $X$ ensures the orbit space $X/G$ is second countable, locally compact, Hausdorff and in particular paracompact. We can take a family of compact sets and relatively compact open sets $K_\lambda\subset U_\lambda$ of $X/G$ such that the compact sets $K_\lambda$ cover $X/G$ and that the family of relatively compact open sets $U_\lambda$ are locally finite. Now, we can further take a family of compact sets and relatively compact open sets $F_\lambda\subset W_\lambda$ such that each $W_\lambda/G$ is contained in $U_\lambda$ and that each $F_\lambda$ contains $K_\lambda$. Now, for each $\lambda$, get a nonnegative function $\theta_\lambda$ such that $\theta_\lambda(x)=1$ for all $x$ in $F_\lambda$ with support contained in $W_\lambda$. Then, a well-defined expression $\theta=\Sigma\theta_\lambda$ defines a continuous, nonnegative function $\theta$ such that first, for any $G$-compact subset $F$ of $X$, there exists a $G$-invariant open subset $W$ of $X$ on which the sum $\Sigma\theta_\lambda$ becomes a finite sum (thus $\theta$ has compact support inside $W$) and secondly, for any $x$ in $X$ there exsits $g\in G$ with $\theta(gx)>0$. The desired cut-off function $c$ on $X$ can be defined by $c(x)={\left(\frac{\theta(x)}{\int_G(g\theta)(x)d\mu} \right)}^{1/2}$. We also remark here that the set of cut-off functions on a proper $G$-space $X$ is connected in $C_b(X)$.

\begin{dfn}(cf. \cite{HigRoe} APPENDIX A) A graded $G$-Hilbert space is a $G$-Hilbert space $\hill$ with a fixed grading automorphism $\epsilon$ which is involutive (a selfadjoint unitary) and commutes with the action of  $G$. A grading automorphism, or simply a grading $\epsilon$ defines a decomposition of $\hill$ into two orthogonal closed $G$-invariant subspaces $\hill^{(0)}$ and $\hill^{(1)}$, where $\hill^{(0)}$ (resp.\ $\hill^{(1)}$) is the $+1$ (resp.\ $-1$) eigenspace of $\epsilon$. In this way, a graded $G$-Hilbert space is understood as nothing but as a pair of $G$-Hilbert spaces $\hill^{(0)}$ and $\hill^{(1)}$. An operator on a graded $G$-Hilbert space is called {even} (resp. {odd}) if it commutes (resp.\ anti-commutes) with the grading $\epsilon$. A graded tensor product $\hill_1\hat\otimes\hill_2$ of graded $G$-Hilbert spaces $\hill_1$ and $\hill_2$ is defined as a Hilbert space $\hill_1\otimes\hill_2$ with a grading $\epsilon_1\otimes\epsilon_2$ where $\epsilon_i$ are the gradings on $\hill_i$ for $i=1,2$.
\end{dfn}

\begin{dfn}(cf. \cite{HigRoe} APPENDIX A) A graded $G$-$\Calg$ is a $G\times\mathbb{Z}/2\mathbb{Z}$-$\Calg$ $A$. This is nothing but a $G$-$\Calg$ with a fixed grading automorphism on $A$ of degree two which commutes with the $G$-action. A graded $G$-$\Calg$ $A$ decomposes into two $G$-invairant closed selfadjoint subspaces $A^{(0)}$ and $A^{(1)}$, where $A^{(0)}$ (resp.\ $A^{(1)}$) is the $+1$ (resp.\ $-1$) eigenspace of the grading automorphism on $A$. They satisfy $A^{(i)}\cdot A^{(j)} = A^{(i+j)}$ for $i,j \in \mathbb{Z}/2\mathbb{Z}$. An element $a$ in $A^{(i)}$ is called homogeneous of degree $i$; and we express it by $\partial a = i$. We also call a element $a$ in $A^{(0)}$ (resp.\ $A^{(1)}$) as even (resp.\ odd).  A grading commutator $[\,,\,]$ is defined by $[a,b]=ab-(-1)^{\partial a\partial b}ba $ for homogeneous elements $a,b \in A$ and by extending it linearly.
\end{dfn}

Let $\hill$ be a graded $G$-Hilbert space with a grading $\epsilon$. The conjugation by $\epsilon$ defines a grading on the $\Calg$ $\BH$.

The algebra $C_0(\mathbb{R})$ of continuous functions on the real line which vanish at infinity becomes a graded $\Calg$ by a grading automorphism which is identity on even functions and $-1$ on odd functions. We denote this graded $\Calg$ by $\mathcal{S}$.

Let $A$ and $B$ be graded $G$-$\Calgs$. There is a notion of graded tensor products of $A$ and $B$. The crucial feature is that we first define a product and an involution on an algebraic tensor product $A\hat\odot B$ (a tensor product of vector spaces) by $(a\hat\otimes b)(c\hat\otimes d)=(-1)^{\partial b\partial c}(ac)\hat\otimes(bd), (a\hat\otimes b)^\ast=(-1)^{\partial a\partial b}a^\ast\hat\otimes b^\ast$ for homogeneous $a,c \in A, \, b,d \in B$. There are the maximal graded tensor product $A\hat\otimes_{\text{max}}B$ and the minimal graded tensor product $A\hat\otimes_{\text{min}}B$, or simply $A\hat\otimes B$. They coincide when $A$ or $B$ is nuclear. We refer to the book \cite{Bla} for further details.

\begin{example}\label{Clif} (Clifford algebras) (cf. \cite{HigRoe} APPENDIX A) Let $V$ be a finite dimensional vector space over $\mathbb{R}$. The complexified exterior algebra $\Lambda^\ast(V)\otimes\mathbb{C}$ naturally becomes a graded Hilbert space. The Clifford algebra $\Cliff(V)$ of $V$ is a (graded) $\C$-subalgebra of $B(\Lambda^\ast(V)\otimes\mathbb{C})$ generated by the Clifford multiplication operators $c(v)=\text{ext}(v)+\text{int}(v)$ for $v \in V$. Here, $\text{ext}(v)$ is the exterior multiplication by $v$ and $\text{int}(v)$ is its adjoint. One can also define the Clifford algebra $\overline{\Cliff}(V)$ as a (graded) $\C$-subalgebra of $B(\Lambda^\ast(V)\otimes\mathbb{C})$ generated by the Clifford multiplication operators $\overline c(v)=\text{ext}(v)-\text{int}(v)$. The graded $\Calgs$ $\Cliff(V)$ and $\overline{\Cliff}(V)$ are both isomorphic (as graded $\Calgs$) to the (abstract) Clifford algebra $\mathbb{C}_n$ where $n=\text{dim}(V)$ which is a graded $\Calg$ generated by $n$ anticommuting odd selfadjoint unitaries. A graded tensor product $\Cliff(V)\hat\otimes\overline{\Cliff}(V)$ can be naturally identified with $B(\Lambda^\ast(V)\otimes\mathbb{C})$. Also, we have a natural isomorphism $\Cliff(V)\hat\otimes \Cliff(W) \cong \Cliff(V\oplus W)$ for finite real vector spaces $V$ and $W$. When a group $G$ acts on $V$ by linear isometries $g\colon v\mapsto g(v)$ for $g$ in $G$ and $v$ in $V$, the $\Calg$ $\Cliff(V)$ or $\overline{\Cliff}(V)$ naturally becomes a graded $G$-$\Calg$ by defining $g(c(v))=c(g(v))$ for $g$ in $G$ and $v$ in $V$.
\end{example}

\begin{dfn} (cf. \cite{Bla} Chapter 13.) Let $B$ be a $G$-$\Calg$. A pre-Hilbert $B$-module is a right $B$-module $\mathcal{E}$ with a $B$-valued inner product $\langle \cdot,\cdot \rangle \colon \mathcal{E} \times \mathcal{E} \to B$. The norm on $\E$ is defined by $\|e\|=\|\s{e,e}\|^{\frac12}$ for $e \in \E$. If $\E$ is complete with this norm, we call it a Hilbert $B$-module. It is called full if $\s{\E,\E}=B$. A Hilbert $G$-$B$-module is a Hilbert $B$-module with a continuous $G$-action which is compatible with the $B$-module structure:\ $\s{ge_1,ge_2}=g\s{e_1,e_2}$, $g(eb)=g(e)g(b)$ for $g\in G, e,e_1,e_2\in\E, b\in B$. For a graded $G$-$\Calg$ $B$ (i.e.\ $G\times\bbZ/2\bbZ$-$\Calg$), a graded Hilbert $G$-$B$-module is nothing but a Hilbert $G\times\bbZ/2\bbZ$-$B$-module.   
\end{dfn}

For Hilbert $B$-modules $\E_1,\E_2$, a $B$-linear map $T\colon\E_1\to\E_2$ is called adjointable if there exists a $B$-linear map $T^\ast\colon\E_2\to\E_1$ such that $\s{Te_1,e_2}=\s{e_1,T^\ast e_2}$ for $e_1\in \E_1, e_2\in \E_2$. An adjointable $B$-linear map is automatically continuous. We denote the set of adjointable $B$-linear maps on $\E_1$ (resp.\ from $\E_1$ to $\E_2$) by $B(\E_1)$ (resp.\ $B(\E_1,\E_2)$).
The set $B(\E_1)$ becomes a $\Calg$ with the operator norm. If $\E_1$ is a graded $G$-$B$-Hilbert module, $B(\E_1)$ naturally becomes a graded $\Calg$ with a $G$-action. 

An adjointable $B$-linear map $T$ from $\E_1$ to $\E_2$ is called compact if it is in the closed linear span of $B$-rank-one operators $\theta_{e_2,e_1}\colon e \mapsto e_2\s{e_1,e}$ $e_1\in \E_1, e_2\in\E_2$. The set of compact operators on $\E_1$, denoted by $\K(\E_1)$, is an ideal of $B(\E_1)$ (we also denote the set of compact operators from $\E_1$ to $\E_2$ by $\K(\E_1,\E_2)$).  Moreover, $B(\E_1)$ can be identified with the multiplier algebra $M(\K(\E_1))$. The quotient $\Calg$ of $B(\E_1)$ by the ideal $\K(\E_1)$ is sometimes called as the Calkin algebra of $\E_1$; we denote it by $Q(\E_1)$. If $\E_1$ is a graded Hilbert $G$-$B$-module, $\K(\E_1)$ is a graded $G$-$\Calg$.

(Graded) exterior and interior tensor products of Hilbert modules are defined in \cite{Bla} for example. For ungraded (resp. graded) Hilbert $G$-$B_i$-modules $\E_i$ for $i=1, 2$, we denote their ungraded (resp. graded) exterior tensor product by $\E_1\otimes\E_2$ (resp. $\E_1\hat\otimes\E_2$). It is an ungraded (resp. graded) Hilbert $G$-$B_1\hat\otimes B_2$ module. We may sometimes omit to write $\otimes$ or $\hat\otimes$ for convenience when there could be no confusion. 

\begin{example} Let $B$ be a graded $G$-$\Calg$. Then, $B$ itself may be viewed as a graded Hilbert $G$-$B$-module with $\s{b_1,b_2}=b_1^\ast b_2$. The $\Calg$ $\K(B)$ of compact operators on $B$ is naturally isomorphic to $B$ by means of left multiplication. Hence the $\Calg$ $B(B)$ of adjointable operators on $B$ is isomorphic to the multiplier algebra $M(B)$.
\end{example}

The standard $G$-Hilbert space $\hill_G$ is defined as $L^2(G)\otimes l^2$ equipped with the left-regular representation $G$ on $L^2(G)$ and the trivial one on $l^2$. For a proper $\Calg$ $A$, any countably generated Hilbert $G$-$A$-module can be equivariantly embedded into a standard one $A\otimes\hill_G$.

\begin{prop} \label{prop:stab}(cf. \cite{bolic} PROPOSITION 5.5.) Let $A$ be a proper $\Calg$ with the base space $X$ and $\E$ be a countably generated Hilbert $G$-$A$-module. Then, there exists a $G$-equivariant adjointable isometry $V$ from $\E$ to $A\otimes\hill_G$. 
\end{prop}  
\begin{proof} Take any cut-off function $c$ on $X$. We have an adjointable isometry $V$ from $\E$ to $L^2(G, \E)$ which maps an element $e$ in $\E$ to a function $f\colon g\mapsto c(g)e$. The adjoint of $V$ is an operator from $L^2(G, \E)$ to $\E$ sending a function $f$ to $\int_Gc(g)f(g)d\mu$. This defines an equivariant embedding of $\E$ into $L^2(G, \E)$. Now, by using any non-equivariant embedding $W$ of $\E$ into $A\otimes l^2$ (we refer the book \cite{Bla} for the existence of such embeddings), we have an equivariant embedding $\tilde W$ of $L^2(G, \E)$ into $L^2(G, A\otimes l^2)$ which sends a function $f$ to a function $\tilde W(f)\colon g\mapsto g(V(g^{-1}(f(g))))$. Hence, we have an equivariant embedding of $\E$ into $L^2(G, A\otimes l^2)\cong A\otimes\hill_G$.
\end{proof}

In the above proposition, considering in particular the Hilbert $G$-$A$-module $A$, we have an $G$-equivariant embedding $V$ of $A$ into $A\otimes\hill_G$. We have an injective $G$-equivariant $\ast$-homomorphism $\Ad_V$ from $\K(A)\cong A$ into $\K(A\otimes\hill_G)\cong A\otimes\K(\hill_G)$. This is called the Stabilization of a proper $G$-$\Calg$. Here, we define for any adjointable isometry $V$ of Hilbert $G$-$B$-modules $\E_1$ to $\E_2$, the $\ast$-homomorphism $\Ad_V\colon T\mapsto VTV^\ast$ from $B(\E_1)$ to $B(\E_2)$ which is $G$-equivariant if $V$ is; and this $\ast$-homomorphism restricts to the $\ast$-homomorphism $\Ad_V$ from $\K(\E_1)$ to $\K(\E_2)$.

\begin{dfn}(Continuous field of Hilbert spaces) (cf. \cite{DixC} CHAPTER 10.) A continuous field of (complex) $G$-Hilbert spaces over a locally compact, Hausdorff topological space $X$ is a pair $((\hill_x)_{x\in X}, \Gamma)$ of a family $(\hill_x)_{x\in X}$ of $G$-Hilbert spaces over $X$ and a prescribed set $\Gamma$ of sections of $(\hill_x)_{x\in X}$ (we call them basic sections) satisfying the following conditions: $\Gamma$ is a vector space (over $\bbC$) with respect to its vector space structure coming from those of $\hill_x$; $\Gamma_x=\{\,v(x)\in\hill_x \,\mid\, v\in\Gamma\,\}$ is dense in $\hill_x$ for each $x$ in $X$; for any sections $v,w$ in $\Gamma$, a function $x \to \langle v(x),w(x)\rangle_x$ is continuous on $X$ ($\langle\cdot,\cdot\rangle_x$ denotes the inner product on $\hill_x$); the set $\Gamma$ is $G$-invariant with respect to the evident pointwise action of $G$ on  $(\hill_x)_{x\in X}$; and for any section $v$ in $\Gamma$ and for any sequence $(g_n)$ converging to $1$ in $G$, $(g_n(v(x)))$ converges to $v(x)$ uniformly on compact subsets of $X$. An arbitrary section of $(\hill_x)_{x\in X}$ is said to be continuous if it is a uniform limit over compact subsets of $X$ of basic sections. The set $\E$ of continuous sections which vanish at infinity becomes a Hilbert $G$-$C_0(X)$-module. The inner product on $\E$ is defined by $\langle v\,w\rangle\colon x\mapsto \langle v(x),w(x)\rangle_x$ for $v,w$ in $\E$. The continuous field of graded $G$-Hilbert spaces can be defined analogously. In this case, the set of continuous sections which vanish at infinity becomes a graded Hilbert $G$-$C_0(X)$-module. We note here that one can further generalize this construction to define a continuous field of graded Hilbert $G$-$B$-modules. 
\end{dfn}

\begin{dfn}(Continuous field of $\Calgs$) (cf. \cite{DixC} CHAPTER 10.) A continuous field of $G$-$\Calgs$ over a locally compact, Hausdorff topological space $X$ can be defined similarly to a continuous field of Hilbert spaces. It is a pair $((A_x)_{x\in X}, \Gamma)$ of a family $(A_x)_{x\in X}$ of $\Calgs$ over $X$ and a prescribed set $\Gamma$ of (basic) sections of $(A_x)_{x\in X}$ satisfying the following conditions: $\Gamma$ is a $\ast$-algebra with respect to its structure of a $\ast$-algebra coming from those of $A_x$; $\Gamma_x=\{\,v(x)\in\ A_x \,\mid\, v\in\Gamma\,\}$ is dense in $A_x$ for each $x$ in $X$; for any section $v$ in $\Gamma$, a function $x \to ||v(x)||_x$ is continuous on $X$ ($||\cdot||_x$ denotes the norm on $A_x$); the set $\Gamma$ is $G$-invariant; and for any section $v$ in $\Gamma$ and for any sequence $(g_n)$ converging to $1$ in $G$, $(g_n(v(x)))$ converges to $v(x)$ uniformly on compact subsets of $X$. An arbitrary section of $(A_x)_{x\in X}$ is said to be continuous if it is a uniform limit over compact subsets of $X$ of basic sections. The set of continuous sections $A$ which vanish at infinity becomes a $\Calg$. The continuous field of graded $G$-$\Calgs$ can be defined analogously. In this case, the set of continuous sections which vanish at infinity becomes a graded $G$-$\Calg$.
\end{dfn}

Given a continuous field $((A_x)_{x\in X}, \Gamma)$ of graded $G$-$\Calgs$ and a nuclear graded $G$-$\Calg$ $B$, one can perform a (graded) tensor product  on each fiber to get a new continuous field $((A_x\hat\otimes B)_{x\in X}, \Gamma')$. The set of basic sections $\Gamma'$ can be defined as the span of algebraic tensors $v\hat\otimes b$ with $v\in\Gamma$ and $b\in B$. On the other hand, when $G$ is an abelian group or more generally, an amenable group (a group $G$ is said to be amenable if its reduced group $\Calg$ $C^{\ast}_{\text{red}}(G)$ is nuclear), one can perform a reduced (maximal) crossed product on each fiber to get a continuous field $((\C_{\text{red}}(G,A_x))_{x\in X}, \Gamma'')$. The set of basic sections $\Gamma''$ can be defined as the span of algebraic tensors $f\otimes c$ with $v\in\Gamma$ and a continuous function $f\in C_c(G)$ with compact support. That these indeed define continuous fields of $\Calgs$ is explained in \cite{KW} for example. 

\begin{example} Let $((\hill_x)_{x\in X}, \Gamma)$ be a continuous field of graded $G$-Hilbert spaces over a locally compact space $X$. It naturally defines a continuous field $(\K(\hill_x)_{x\in X}, \Gamma')$ of graded $G$-$\Calgs$ over $X$. The set of basic sections $\Gamma'$ is defined to be the span of rank-one sections $x\mapsto \theta_{v(x),w(x)}$ for sections $v,w$ in $\Gamma$.
\end{example}

\begin{dfn}(Unbounded operators on a Hilbert space) (cf. \cite{analysisnow} CHAPTER 5.) Let $\hill$ be a Hilbert space over $\bbC$ or $\bbR$ and denote the inner product on $\hill$ by $\langle\cdot,\cdot\rangle$. The linear map $T$ from a dense subspace $D(T)$ of $\hill$ to $\hill$ is usually called an unbounded operator on $\hill$. (One can of course consider an unbounded operator taking another Hilbert space for its range). The adjoint $T^\ast$ of $T$ is defined on $D(T^\ast)=\{\,v\in\hill\mid w\mapsto \langle v,Tw\rangle \,\,\text{is a bounded linear functional on $D(T)$}\,\}$; for any $v$ in $D(T^\ast)$, $T^\ast v$ is defined to be the unique vector in $\hill$ satisfying $\langle T^\ast v, w\rangle= \langle v,Tw\rangle$ for all $w$ in $D(T)$. When $D(T^\ast)$ is a dense subspace of $\hill$, in other words, if $T^\ast$ is densely defined on $\hill$, then we say $T$ is adjointable, and call $T^\ast$ as the adjoint of $T$. Evidently, in this case, $T^\ast$ is an adjointable unbounded operator on $\hill$ and its adjoint $T^{\ast\ast}$ is an extension of $T$. (By an extension of $T$, we mean an unbounded operator $S$ on $\hill$ defined on the domain containing $D(T)$ such that $S=T$ on $D(T)$ as linear maps.) For an adjointable operator $T$, $D(T^{\ast\ast})$ coincides with a subspace \begin{small}\[\{\,v\in\hill \mid \text{there exists a sequence $(v_n)\subset D(T)$ s.t. $v_n\to v$ and $(Tv_n)$ is convergent as $n\to \infty$}\,\}.\] \end{small}An adjointable unbounded operator $T$ on $\hill$ is called symmetric if $T^\ast$ is an extension of $T$, selfadjoint if $T$ is symmetric and $D(T)=D(T^\ast)$ and essentially selfadjoint if $T^{\ast\ast}$ is selfadjoint. For a symmetric operator $T$, being essentially selfadjoint is equivalent to that $T^2+1$ has dense range; and in complex case, this is equivalent to that $T\pm i$ have dense ranges. In this case, we have well-defined bounded operators $(T^2+1)^{-1}$ and $(T\pm i)^{-1}$ in complex case. We say an essentially selfadjoint operator $T$ has compact resolvent if these operators are compact. We are mostly interested in selfadjoint operators, and so, when $T$ is essentially selfadjoint, we usually treat $T$ as a selfadjoint operator by implicitly using its extension $T^{\ast\ast}=T^{\ast\ast\ast}=T^\ast$ whenever it makes no confusion. Any diagonalizable operator with diagonal entries in $\bbR$ is essentially selfadjoint; and it has compact resolvent if and only if the number of its eigenvalues lying in any compact set of $\bbR$ is finite taking into account the multiplicities. For an essentially selfadjoint operator $T$ on a complex Hilbert space $\hill$, $T\pm i$ are unbounded operators defined on $D(T)$ which have dense images in $\hill$ and are bounded away from $0$. One has the unique $\ast$-homomorphism from $C_0(\bbR)$ to $B(\hill)$ sending functions $(x\pm i)^{-1}$ to $(T\pm i)^{-1}$. When $T$ is a diagonalizable operator, then this $\ast$-homomorphism becomes the evident one. If in addition, $T$ has compact resolvent, then it is also clear that this $\ast$-homomorphism takes $\K(\hill)$ for its range. We remark here that in the case when the Hilbert space $\hill$ is graded and the operator $T$ is odd, then the above defined $\ast$-homomorphism becomes a graded $\ast$-homomorphism from the graded $\Calg$ $\cS$. 
\end{dfn}

\begin{example}\label{harm}(Harmonic oscillator) (cf.\cite{HigKas2} Definition 2.6.) For a positive real number $\alpha$, let $H=\alpha^2\Delta+x^2$ be an unbounded operator on the real (or complex) Hilbert space $L^2(\bbR)$ which is defined on the subspace $C^\infty_c(\bbR)$ of test functions. Here, $\Delta$ is the laplacian $-\frac{d^2}{dx^2}$. This operator is a diagonalizable operator with diagonal entries in $\bbR$ having compact resolvent, and so in particular, is essentially selfadjoint. This can be seen as follows. First, note that $H=KL-\alpha=LK+\alpha$ where $K=\alpha\frac{d}{dx}+x$ and $L=-\alpha\frac{d}{dx}+x$. One finds that $HK^n=K^nH+2n\alpha K^n$ and that a function $f_0(x)=e^{-\frac{x^2}{2\alpha}}$ is in the kernel of $K$ and hence an eigenvector for $H$ with an eigenvalue $\alpha$. By $HK^n=K^nH+2n\alpha K^n$, we see a function $f_n(x)=K^nf_0(x)=p_n(x)e^{-\frac{x^2}{2\alpha}}$ is an eigenvector for $H$ with an eigenvalue $(2n+1)\alpha$ for any $n\geq0$ where $p_n(x)$ is a certain polynomial of degree $n$. We note that any two eigenvectors for a symmetric operator corresponding to different eigenvalues are orthogonal.  After being normalized, these functions become an orthonormal basis of the Hilbert space $L^2(\bbR)$ and hence $H$ is diagonalizable with eigenvalues $(2n+1)\alpha$ for $n\geq0$ each of which has single multiplicity. The functional calculus sends $f$ in $C_0(\bbR)$ to a bounded diagonal operator $\left(
\begin{array}{ccccc}
f(\alpha)&0&0&0&\cdots\\
0&f(3\alpha)&0&0&\cdots\\
0&0&f(5\alpha)&0&\cdots\\
0&0&0&\ddots&\\
\vdots&\vdots&\vdots&&\ddots \\
\end{array}
\right)$ where we used the mentioned orthonormal basis for $L^2(\bbR)$ to write an operator as an infinite matrix.
\end{example}

\begin{lemma}\label{lemselfad} Let $T$ be a symmetric unbounded operator on a complex Hilbert space $\hill$ defined on $D(T)$ with $D(T)=D(T^2)$ (i.e. the image $TD(T)$ is contained in $D(T)$).  Suppose, $T^2$ is essentially selfadjoint. Then $T$ is essentially selfadjoint. In addition, if $T^2$ has compact resolvent, then so does $T$.
\end{lemma}
\begin{proof} Denote the inner product by $\s{\cdot,\cdot}$. First, we see $T^2+1$ has the dense range. In fact, assume for $y\in\hill$, $\s{(T^2+1)x, y}=0$ for any $x\in D(T)$. Then, it follows $y\in D({T^{2}}^{\ast})$. Take $(y_n)\subset D(T)$ with $y_n\to y$ and $T^2y_n \to {T^2}^\ast y$. We have $\s{({T^2}^\ast+1)y, y}=0$ showing $y=0$. Now, it follows $T\pm i$ has the dense ranges. Hence we have bounded operators $(T\pm i)^{-1}$. Now, we may assume $T\pm i$ is onto (just replace $T$ by $T^{\ast\ast}$). Take any $y=(T+i)x\in D(T^\ast)$ with $x\in D(T)$: so, $z\mapsto \s{(T+i)x, Tz}$ is bounded on $D(T)$. Since $T$ is symmetric, it follows $x$ in $D({T^2}^\ast)$; so take $(x_n)\subset D(T)$ with $x_n\to x$ and $((T^2+1)x_n)$ convergent. Applying a bounded operator $(T-i)^{-1}$ to a convergent sequence ($(T^2+1)x_n$), we see ($(T+i)x_n)=(y_n$) converges to $(T+i)x=y$. Also, ($Ty_n$) is convergent. Thus, $y$ is in $D(T^{\ast\ast})$ showing $T$ is essentially selfadjoint. When $T^2$ has compact resolvent, it follows $(T^2+1)^{-1}$ is a compact operator on $\hill$. Thus, $(T\pm i)^{-1}$ must be compact.
\end{proof}

\begin{example}(Bott-Dirac operator)\label{BottDirac}  (cf. \cite{HigKas2} Definition 2.6.)  For a positive real number $\alpha$, let $B=\alpha\overline{c}(w)\frac{d}{dx}+c(w)x=\begin{pmatrix}
0& -\alpha \frac{d}{dx}+x \\
\alpha \frac{d}{dx}+x & 0
\end{pmatrix}$ be an odd symmetric unbounded operator on a graded complex Hilbert space $\hill=L^2(\bbR, \Lambda^\ast(\bbR)\otimes\bbC)$ which is defined on a subspace $C^\infty_c(\bbR, \Lambda^\ast(\bbR)\otimes\bbC)$. Here, we used the Clifford multiplication $c(w)$ and $\overline{c}(w)$ as explained in Example \ref{Clif} where $w$ denotes the standard basis vector on a real Hilbert space $\bbR$. The matrix representation respects the even subspace $L^2(\bbR)$ and the odd subspace $L^2(\bbR)w$ of $\hill$. Note, in the same notation in Example \ref{harm}, $B=\begin{pmatrix}
0& L \\
K & 0
\end{pmatrix}$. Hence, $B^2=\begin{pmatrix}
LK & 0 \\
0 & KL
\end{pmatrix}=\begin{pmatrix}
H-\alpha & 0 \\
0 & H+\alpha
\end{pmatrix}$. It is now easy to see that $B^2$ is a diagonalizable operator having compact resolvent with eigenvalues $2n\alpha$ $(n\geq0)$ on the even subspace and $2(n+1)\alpha$ $(n\geq0)$ on the odd subspace. It follows by Lemma \ref{lemselfad}, $B$ is an odd essentially selfadjoint operator having compact resolvent, hence diagonalizable. Note, eigenvalues of $B$ is necessarily $\pm\sqrt{2n\alpha}$ $(n\geq0)$ (each of which has single multiplicity): if we have a nonzero eigenvalue $a$ of $B$, its eigenvector $v$ can be written $v^{(0)}+v^{(1)}$ with homogeneous $v^{(0)},v^{(1)}$; and it follows the odd operator $B$ must send $v^{(0)}$ to $av^{(1)}$ and $v^{(1)}$ to $av^{(0)}$. We see $v^{(0)}-v^{(1)}$ is an eigenvector with eigenvalue $-a$. One may want to write $B$ as a diagonal operator, but since any eigenvector for $B$ corresponding to a nonzero eigenvalue is necessarily not homogeneous, it is not so enlightening to do so. However, it is easy to see we have a following way of writing $B$ as an infinite matrix which respects the grading of the Hilbert space: \begin{equation*}B=\left(
\begin{array}{ccccc}
0 &0&0&0&\cdots\\
0& \left(\begin{array}{cc}
0&\sqrt{2\alpha}\\
\sqrt{2\alpha}&0\\
\end{array} \right)
&0&0&\cdots\\
0&0&\left(\begin{array}{cc}
0&\sqrt{4\alpha}\\
\sqrt{4\alpha}&0\\
\end{array} \right)&0&\cdots\\
0&0&0&\left(\begin{array}{cc}
0&\sqrt{6\alpha}\\
\sqrt{6\alpha}&0\\
\end{array} \right)&\\
\vdots&\vdots&\vdots&&\ddots \\
\end{array}
\right) \end{equation*} where we are using here a basis of Hilbert space $\hill$ consisting of (homogeneous) eigenvectors for $B^2$. We remark that whenever we have an odd (symmetric) diagonalizable operator $T$ on a graded Hilbert space with $T^2$ which is diagonalizable by using homogeneous eigenvectors, we can represent $T$ in similar way using eigenvectors for $T^2$ even if $T^2$ have an infinite dimensional eigenspace for some eigenvalues. Now, the functional calculus for $B$ send an odd function $f$ to:\begin{small}
\begin{equation*}f(B)=\left(
\begin{array}{ccccc}
0 &0&0&0&\cdots\\
0& \left(\begin{array}{cc}
0&f(\sqrt{2\alpha})\\
f(\sqrt{2\alpha})&0\\
\end{array} \right)
&0&0&\cdots\\
0&0&\left(\begin{array}{cc}
0&f(\sqrt{4\alpha})\\
f(\sqrt{4\alpha})&0\\
\end{array} \right)&0&\cdots\\
0&0&0&\left(\begin{array}{cc}
0&f(\sqrt{6\alpha})\\
f(\sqrt{6\alpha})&0\\
\end{array} \right)&\\
\vdots&\vdots&\vdots&&\ddots \\
\end{array}
\right), \end{equation*} \end{small} and an even function $f$ to: \begin{small}
\begin{equation*}f(B)=\left(
\begin{array}{ccccc}
f(0) &0&0&0&\cdots\\
0& \left(\begin{array}{cc}
f(\sqrt{2\alpha})&0\\
0&f(\sqrt{2\alpha})\\
\end{array} \right)
&0&0&\cdots\\
0&0&\left(\begin{array}{cc}
f(\sqrt{4\alpha}) & 0\\
0 & f(\sqrt{4\alpha}) \\
\end{array} \right)&0&\cdots\\
0&0&0&\left(\begin{array}{cc}
f(\sqrt{6\alpha})&0\\
0& f(\sqrt{6\alpha})\\
\end{array} \right)&\\
\vdots&\vdots&\vdots&&\ddots \\
\end{array}
\right). \end{equation*}  \end{small}
\end{example}

\begin{lemma}(Mehler's formula)  (cf. \cite{HKT} APPENDIX B, \cite{Shrodinger}) For $\alpha>0$, we have the following equation of bounded operators on the Hilbert space $L^2(\bbR)$: \begin{equation*}
e^{-(\alpha^2\Delta+x^2)}=e^{-r(\alpha)\alpha^{-1}x^2}e^{-s(\alpha)\alpha\Delta}e^{-r(\alpha)\alpha^{-1}x^2} \\
\end{equation*}
with, 
\begin{equation*}
 r(\alpha)=\frac{1}{\sqrt{2}}\frac{\sinh(\frac{\alpha}{\sqrt{2}})}{\cosh(\frac{\alpha}{\sqrt{2}})},\,\,\,\,\, s(\alpha)=\frac{1}{\sqrt{2}}\sinh(\sqrt{2}\alpha) \\
\end{equation*}
\end{lemma}
\begin{proof}
It suffices to show the following holds for any $t>0$:
\begin{equation*}
e^{-t(\Delta+x^2)}=e^{-r(t)x^2}e^{-s(t)\Delta}e^{-r(t)x^2} \\
\end{equation*} First, notice by letting $A=x^2$, $B=\Delta$, $C=[x^2, \Delta]=2x\frac{d}{dx}+1$, we have $[A, B]=C$, $[C, A]=4A$, $[C,B]=-4B$. On the other hand, the same algebraic relations hold when we set $A$ as $\begin{pmatrix} 
0 & \sqrt{2} \\
0 & 0 \\
\end{pmatrix} $, $B$ as $\begin{pmatrix} 
0 & 0 \\
 \sqrt{2} & 0 \\
\end{pmatrix} $ and $C$ as $\begin{pmatrix} 
2 & 0 \\
0 & -2 \\
\end{pmatrix} $. Hence, it suffices to check that the following holds for any $t>0$:
\begin{equation*}
e^{-t(B+A)}=e^{-r(t)A}e^{-s(t)B}e^{-r(t)A} \\
\end{equation*}
This can be checked easily, so we omit the rest of our calculations. 
\end{proof}

\begin{dfn}(Unbounded multiplier on a Hilbert module) (cf. \cite{HKT} APPENDIX A) Let $B$ be a $\Calg$, $\E$ be a Hilbert $B$-module and denote the inner product by $\langle\cdot,\cdot\rangle$. The notion of unbounded operators easily translates into this situation. However, since we don't have an analogue of Riesz Representation Theorem on a Hilbert space, there is some difference from the Hilbert space case. An essentially selfadjoint unbounded multiplier is a linear map $T$  defined on a dense $B$-submodule $D(T)$ of $\E$ which is symmetric (i.e. $\langle Tv, w\rangle=\langle v, Tw\rangle$ for $v,w$ in $D(T)$) and such that $(T\pm i)^{-1}$ have dense images in $\E$. Similarly to the Hilbert space case, one has the unique $\ast$-homomorphism from $C_0(\bbR)$ to $B(\E)$ sending functions $(x\pm i)^{-1}$ to $(T\pm i)^{-1}$ which we call functional calculus associated to $T$. We say $T$ has compact resolvent if this $\ast$-homomorphism takes $\K(\E)$ for its range. (i.e. if $(T\pm i)^{-1}$ is in $\K(\E)$). When the Hilbert module $\E$ is graded and the operator $T$ is odd, then the above defined $\ast$-homomorphism becomes a graded one. Note the functional calculus is necessarily nondegenerate ($(T\pm i)^{-1}$ have dense range). Conversely, given any nondegenerate $\ast$-homomorphism from $C_0(\bbR)$ to $B(\E)$, a symmetric unbounded multiplier $x$ on $\E$ can be defined on a subspace $C_c(\bbR)\E$ in an obvious way. The image of $(x\pm i)^{-1}$ contains $C_c(\bbR)\E$, and thus, is dense. 
\end{dfn}

\begin{example} Consider a graded Hilbert $\cS$-module $\E=\cS\hat\otimes\cS$. A symmetric unbounded multiplier $T=x\hat\otimes1+1\hat\otimes x$ on $\E$ is defined on a subspace $C_c(\bbR)\hat\otimes C_c(\bbR)$ (the tensor product here is the algebraic one). That the multipliers $(T\pm i)^{-1}$ have the dense ranges or that $T$ is essentially selfadjoint may not be easy to be seen. However, it is easy to check that we have a graded $\ast$-homomorphism from $\cS$ to $\cS\hat\otimes\cS=\K(\E)$ sending $e^{-x^2}$ to $e^{-x^2}\hat\otimes e^{-x^2}$ and $xe^{-x^2}$ to $xe^{-x^2}\hat\otimes e^{-x^2}+e^{-x^2}\hat\otimes xe^{-x^2}$. (Continuity can be checked by representing $\cS\hat\otimes\cS$ on $(L^2(\bbR)\oplus L^2(\bbR)^{\op})\hat\otimes(L^2(\bbR)\oplus L^2(\bbR)^{\op})$ for example: $L^2(\bbR)$ and $L^2(\bbR)^{\op}$ are an even space and an odd space respectively.) This representation is nondegenerate, which easily implies that $(T\pm i)^{-1}$ have the dense ranges. The functional calculus associated to $T$ is evidently, our already defined graded $\ast$-homomorphism.
\end{example}

\begin{example} The observation given in the above example can be pushed further. Let $B$ be a graded $\Calg$, and $\E_1$, $\E_2$ be graded Hilbert $B$-modules. Given any odd essentially selfadjoint unbounded multipliers $T_1$ and $T_2$ on $\E_1$ and on $\E_2$ respectively with domains $D(T_1)$ and $D(T_2)$, we can define an odd symmetric unbounded multiplier $T=T_1\hat\otimes1+1\hat\otimes T_2$ on a Hilbert $B$-module $\E_1\hat\otimes\E_2$ defined on $D(T_1)\hat\otimes D(T_2)$. Again, it may not be easy to see this multiplier $T$ is essentially selfadjoint at first look. However, we have a graded $\ast$-homomorpshim $\cS$ to $B(\E_1)\hat\otimes B(\E_2)\subset B(\E_1\hat\otimes\E_2)$ defined as the composition of the graded $\ast$-homomorpshism from $\cS$ to $\cS\hat\otimes\cS$ which appeared in the above example with a graded $\ast$-homomorphism from $\cS\hat\otimes\cS$ to $B(\E_1)\hat\otimes B(\E_2)$ which is the graded tensor product of two functional calculus associated to $T_1$ and $T_2$. This $\ast$-homomorpshim is nondegenerate; and thus, it follows that $T$ is essentially selfadjoint and that the functional calculus for $T$ is the graded $\ast$-homomorphism defined above.  
\end{example}

\begin{example} (cf. \cite{HigKas2}) We observed that the Bott-Dirac operator $B$ (depending on $\alpha>0$) on a graded Hilbert space (Hilbert $\bbC$-module) $\hill=L^2(\bbR, \Lambda^\ast(\bbR)\otimes\bbC)$ defines an odd essentially selfadjoint operator having compact resolvent, and so the functional calculus $\cS\to \K(\hill)$. Combining this with an odd multiplier $x$ on a graded $\cS$-module $\cS$, we now want to consider the functional calculus  $\cS\to\cS\hat\otimes\K(\hill)$ associated to the essentially selfadjoint odd unbounded multiplier $T=x\hat\otimes1+1\hat\otimes B$ on a graded Hilbert $\cS$-module $\cS\hat\otimes\hill$. To describe this functional calculus, we first decompose $\cS\hat\otimes\hill$ as a direct sum $\displaystyle\bigoplus_{n\geq0}\cS\hat\otimes\hill_n$ where $\hill_n$ is the eigenspace for $B^2$ corresponding to its eigenvalue $2n\alpha$ for $n\geq0$. We just need to see the functional calculus associated to the multipliers on each summands.  On the summand $\cS\hat\otimes\hill_0\cong\cS$, $T$ acts as $x\hat\otimes1$; thus the functional calculus is just $\id_{\cS}\colon\cS\to\cS=\K(\cS\hat\otimes\hill_0)$. For other ``two-dimensional'' summands $\cS\hat\otimes\hill_n$, $T$ acts as $x\hat\otimes1+1\hat\otimes\begin{pmatrix}
0 & \sqrt{2n\alpha}\\
\sqrt{2n\alpha} & 0
\end{pmatrix}$. For $n\geq1$, it may not be simple to describe the functional calculus associated to this odd unbounded multiplier on $\cS\hat\otimes\hill_n$. However, there are some other way of observing this functional calculus by regarding this graded $\ast$-homomorphism $\cS\to\cS\hat\otimes\K(\hill_n)$ as a $\ast$-homomorphism from $C_0(\bbR)$ to $C_0(\bbR)\otimes\K(\hill_n)\cong M_2(C_0(\bbR))$ i.e. by neglecting the grading information. Here, we use an isomorphism of (ungraded) $\Calgs$ $\cS\hat\otimes\K(\hill_n)$ and $C_0(\bbR)\otimes\K(\hill_n)$ sending a homogeneous element $f\hat\otimes T$ to $f\otimes \epsilon^{\partial f} T$ where $\epsilon$ is the grading operator on $\hill_n$. Then, we may view the functional calculus for $T$ on $\cS\hat\otimes\hill_n$ as the functional calculus associated to an unbounded multiplier $\begin{pmatrix}
x & \sqrt{2n\alpha} \\
\sqrt{2n\alpha} & -x \\
\end{pmatrix}$ on an ``ungraded'' Hilbert $C_0(\bbR)$-module $C_0(\bbR)\oplus C_0(\bbR)$. Indeed, one may use this observation on whole space $\cS\hat\otimes\hill$. Namely, using an isomorphism of ungraded $\Calgs$ $\cS\hat\otimes\K(\hill)$ and $C_0(\bbR)\otimes\K(\hill)$ (an isomorphism can be given in the same way as above), we can consider the functional calculus for $T$ as the functional calculus associated to an unbounded multiplier:
\begin{equation*}T'=\left(
\begin{array}{cccc}
x &0&0&\cdots\\
0& \left(\begin{array}{cc}
x&\sqrt{2\alpha}\\
\sqrt{2\alpha}&-x\\
\end{array} \right)
&0&\cdots\\
0&0&\left(\begin{array}{cc}
x&\sqrt{4\alpha}\\
\sqrt{4\alpha}&-x\\
\end{array} \right)&\cdots\\
\vdots&\vdots&\vdots&\ddots\\
\end{array}
\right) \end{equation*} on an ungraded Hilbert $C_0(\bbR)$-module $C_0(\bbR)\otimes\hill$. Note the functional calculus for $T'$ sends an even function $f$ to:
\begin{scriptsize}
\begin{equation*}f(T')=\left(
\begin{array}{cccc}
f(x) &0&0&\cdots\\
0& \left(\begin{array}{cc}
f(\sqrt{x^2+2\alpha})&0\\
0&f(\sqrt{x^2+2\alpha})\\
\end{array} \right)
&0&\cdots\\
0&0&\left(\begin{array}{cc}
f(\sqrt{x^2+4\alpha}) & 0\\
0 & f(\sqrt{x^2+4\alpha}) \\
\end{array} \right)&\cdots\\
\vdots&\vdots&\vdots&\ddots\\
\end{array}
\right). \end{equation*}  \end{scriptsize}
\end{example}

\newpage
\section{$K$-Theory and $K$-Homology of $\Calgs$}

We are going to see some definitions and basic results in $\Calg$ $K$-theory and $K$-homology. They are non-commutative generalizations of $K$-theory and $K$-homology of locally compact Hausdorff spaces. Our basic reference here is \cite{HigRoe}.

\begin{dfn} Let $A$ be a $\Calg$. Two projections $p,q$ in $M_n(A)$ (we call such projections as projections {over} $A$) are {unitary equivalent} if there exists a unitary $u\in M_n(A)$ which conjugates one to another: i.e.\ $upu^\ast=q$. We define a direct sum of two projections over $A$ by the following:
\begin{equation*}
p\oplus q=
\begin{pmatrix}
p & 0 \\
0 & q
\end{pmatrix}
\quad \text{for $p \in M_m(A)$ and $q \in M_n(A)$} 
\end{equation*}
\end{dfn}

\begin{dfn} Let $A$ be a unital $\Calg$. The $\tilde K_0$ group of $A$ is a group $\tilde K_0(A)$ generated by unitary equivalence classes of projections over $A$ subject to the relations $[p]+[q]=[p\oplus q]$ for projections $p,q$ over $A$ and $[0]=0$. This is an abelian group and a countable group if $A$ is separable.
\end{dfn}

\begin{example} An abelian group $\tilde K_0(\bbC)$ is isomorphic to the integer $\mathbb{Z}$ via a map which sends the unitary equivalent class $[p]$ of a projection $p \in M_n(\bbC)$ to the rank of $p$.
\end{example}

A unital $\ast$-homomorphism from $A$ to $B$ defines a group homomorphism from $\tilde K_0(A)$ to $\tilde K_0(B)$. In this way, we obtain a (covariant) functor $\tilde K_0$ from the category of unital $\Calgs$ and unital $\ast$-homomorphisms to the category of abelian groups.

\begin{dfn} Let $A$ be a $\Calg$. The $K_0$ group of $A$ is the kernel $K_0(A)$ of a homomorphism $\tilde K_0(\tilde A) \to \tilde K_0(\mathbb{C})$ associated to the unique unital $\ast$-homomorphism from $\tilde A$ to $\bbC$ with the kernel $A$. For a unital $\Calg$ $A$, the group $K_0(A)$ is isomorphic to the group $\tilde K_0(A)$ by the inclusion $\tilde K_0(A) \to \tilde K_0(\tilde A)$ which identifies a class defined by a projection over $A$ as a class defined by the same projection viewed as over $\tilde A$. Hence, we usually regard $K_0(A)$ as $\tilde K_0(A)$ for a unital $\Calg$ $A$. 
\end{dfn}

Any $\ast$-homomorphism from $A$ to $B$ extends uniquely to a unital $\ast$-homomorphism from $\tilde A$ to $\tilde B$; and using this, we obtain a group homomorphism from $K_0(A)$ to $K_0(B)$. In this way, $K_0$ becomes a functor from the category of $\Calgs$ and $\ast$-homomorphisms to the category of abelian groups.

A {homotopy} of $\ast$-homomorphisms from $A$ to $B$ is a $\ast$-homomorphism from $A$ to $B[0,1]$. Two $\ast$-homomorphisms from $A$ to $B$ are {homotopic} if there exists a homotopy whose evaluation at $0$ and $1$ gives the two $\ast$-homomorphisms.

A {stabilization} of a $\Calg$ $A$ is a $\ast$-homomorphism from $A$ to $A\otimes \K$ sending $a \in A$ to $a\otimes p \in \K$ where $p$ is a rank-one projection in $\K$.

Let $F$ be a functor from the category of $\Calgs$ to the category of abelian groups. A functor $F$ is called {homotopy invariant} if any two homotopic $\ast$-homomorphisms induce the same group homomorphism; {stable} if any stabilization for any $\Calg$ induces an isomorphism of groups; and {half-exact} if it sends a short exact sequence of $\Calgs$:
\[
\xymatrix{
0 \ar[r] & J \ar[r] & A \ar[r] & A/J \ar[r] & 0
}
\]
to a half exact sequence of groups:
\[
\xymatrix{
F(J) \ar[r] & F(A) \ar[r] & F(A/J)
}
\]
A functor $F$ is called {split-exact} if it sends a split exact sequence of $\Calgs$:
\[
\xymatrix{
0 \ar[r] & J \ar[r] & A \ar[r]_-{\longleftarrow} & A/J \ar[r] & 0
}
\]
to a split exact sequence of groups:
\[
\xymatrix{
0 \ar[r] & F(J) \ar[r] & F(A) \ar[r]_-{\longleftarrow} & F(A/J) \ar[r] & 0
}
\]
The above definitions for two kinds of exactness are written for a covariant functor $F$. In a contravariant case, the arrows for groups go in the reverse direction.

\begin{prop} (cf. \cite{HigRoe} Chapter 4) The functor $K_0$ is a homotopy invariant, stable and half-exact functor.
\end{prop}

Let $F$ be a functor from the category of $\Calgs$ to the category of abelian groups, which is homotopy invariant, stable and half-exact. Denote by $S$ the $\Calg$ $C_0(\mathbb{R})$ of continuous functions on the real line which vanish at infinity. For each $n \in \mathbb{N}$, we define a functor $F_n$ from the category of $\Calgs$ to the category of abelian groups by $F_n(A)=F(S^n\otimes A)$ for a $\Calg$ $A$. Then, the functors $F_n$ satisfies the same property as $F$. Cuntz showed that there is always a natural Bott Periodicity isomorphism $F_n(A)\cong F_{n+2}(A)$ (see \cite{HigRoe}). Hence we can define functors $F_n$ for each $n \in \mathbb{Z}$, by extending the previous definition with the relations $F_n(A)=F_{n+2}(A)$. The sequence of functors $(F_n)_{n \in \mathbb{Z}}$ becomes a homology (cohomology) theory on $\Calgs$, i.e.\ for any short exact sequence of $\Calgs$:
\[
\xymatrix{
0 \ar[r] & J \ar[r] & A \ar[r] & A/J \ar[r] & 0
}
\]
there is a natural long exact sequence of abelian groups and group homomorphisms:
\[
\xymatrix{ 
\, & \ar[r] & F_n(J) \ar[r] & F_n(A) \ar[r] & F_n(A/J) \ar[r] & F_{n+1}(J) \ar[r] & F_{n+1}(A) \ar[r]  & \,
}
\]
The connecting maps $F_n(A/J)\to F_{n+1}(J)$ are called the {boundary maps} of the homology (cohomology) theory $(F_n)_{n \in \mathbb{Z}}$. As a corollary of this, one sees the functors $F_n$, and so $F$, are split-exact. In view of the periodicity, the long exact sequence above is nothing but the six-term exact sequence:
\[
\xymatrix{
F_1(J) \ar[r] & F_1(A) \ar[r] & F_1(A/J) \ar[d] \\
F_{0}(A/J) \ar[u] & \ar[l] F_{0}(A) & \ar[l] F_{0}(J)
}
\]
  
\begin{dfn} The {${K}$-theory} of $\Calgs$ is the homology theory $(K_n)_{n\in\mathbb{Z}}$ on $\Calgs$ defined by the functor $K_0$ from the category of $\Calgs$ to the category of abelian groups. 
\end{dfn}

M. Atiyah proposed how to (analytically) define the $K$-homology of $\Calgs$ which is dual to the $K$-theory of $\Calgs$. We will follow the treatment given by the book \cite{HigRoe}.  

\begin{dfn} Let $A$ be a separable $\Calg$. Let $n$ be a nonnegative integer. A $n$-multigraded Fredholm module over $A$ is a triple $(\hill, \rho, F)$, where $\hill$ is a separable graded Hilbert space, $\rho$ is a graded representation of $A\hat\otimes\bbC_n$ on $\hill$ and $F$ is an odd bounded operator on $H$ satisfying the following relations:
\begin{align}\label{eq:Fred}
\rho(x)(F^2-1) \sim 0,\,\,\, \rho(x)(F-F^\ast) \sim 0,\, \,\, [\rho(x),F] \sim 0 \,\,\,\,\,\text{for $x \in A\hat\otimes\bbC_n$ }
\end{align}
Here $[\,,\,]$ denotes the graded commutator; and for $T\in \BH,$ $T \sim 0$ means $T$ is compact.
\end{dfn}

A $n$-multigraded Fredholm module $(\hill, \rho, F)$ over $A$ is {degenerate} if all the relations in \eqref{eq:Fred} are exact (i.e. if they hold with $\sim 0$ replaced by $=0$).

An\ {operator homotopy}\, of\, $n$-multigraded\ Fredholm\ modules\, over\ $A$\, is \ a \ triple $(\hill, \rho, F_t)_{t\in[0,1]}$ where for each $t \in [0,1]$, $(\hill, \rho, F_t)$ is an $n$-multigraded Fredholm module over $A$ and a map $t \mapsto F_t$ is norm-continuous. We say $n$-multigraded Fredholm modules $(\hill, \rho, F_0)$ and $(\hill, \rho, F_1)$ over $A$ are {homotopic} if there exists an operator homotopy $(\hill, \rho, F_t)_{t\in[0,1]}$.

Let $(\hill_1, \rho_1, F_1)$ and $(\hill_2, \rho_2, F_2)$ be $n$-multigraded Fredholm modules over $A$. The {direct sum} $(\hill_1\oplus\hill_2, \rho_1\oplus\rho_2, F_1\oplus F_2)$ is a $n$-multigraded Fredholm module over $A$; we denote it by $(\hill_1, \rho_1, F_1)\oplus(\hill_2, \rho_2, F_2)$.

Two $n$-multigraded Fredholm modules $(\hill_1, \rho_1, F_1)$ and $(\hill_2, \rho_2, F_2)$ are said to be {unitary equivalent} if there exists a unitary $u$ from $\hill_1$ to $\hill_2$ of degree 0 (i.e.\ it intertwines two gradings) such that $\rho_2(x)=u\rho_1(x)u^\ast$ for $x \in A\hat\otimes\mathbb{C}_n$ and $F_2=uF_1u^\ast$.

\begin{dfn} Let $A$ be a separable $\Calg$. Let $n$ be a nonnegative integer. The (analytic) $K$-homology group $K^{-n}(A)$ of $A$ of degree $-n$ is a group generated by cycles of $n$-multigraded Fredholm modules over $A$ subject to the relations $[(\hill_1, \rho_1, F_1)]+[(\hill_2, \rho_2, F_2)]=[(\hill_1, \rho_1, F_1)\oplus (\hill_2, \rho_2, F_2)]$ and $[(\hill_1, \rho_1, F_1)]=[(\hill_2, \rho_2, F_2)]$ if $(\hill_1, \rho_1, F_1)$ and $(\hill_2, \rho_2, F_2)$ are homotopic for $n$-multigraded Fredholm modules $(\hill_i, \rho_i, F_i)$ over $A$. This is an abelian group. The zero class is represented by any degenerate cycle. Unitary equivalent cycles define the same class. The additive inverse of $[(\hill, \rho, F)]$ is $[(\hill^{op}, \rho^{op}, -F^{op})]$ where $\hill^{op}$ is a graded Hilbert space $\hill$ with two eigenspaces of the grading interchanged; and $F^{op}$ is the operator on $\hill^{op}$ corresponding to an operator $F$ on $\hill$; and $\rho^{op}$ is a (graded) representation of $A\hat\otimes \bbC_n$ on $\hill^{op}$ defined by $\rho^{op}(x)=(-1)^{\partial x}\rho(x)$ for homogeneous $x \in A\hat\otimes \bbC_n$ (here, we identified $\hill^{op}$ with $\hill$ by neglecting the gradings).  
\end{dfn}

\begin{example}\label{prop:KC} The group $K^0(\bbC)$ is isomorphic to $\bbZ$. The isomorphism sends $[(\hill,1,F)]$ to the graded index $\text{Index}(F)=\text{dim}(\text{Ker($F$)}^{(0)})-\text{dim}(\text{Ker($F$)}^{(1)})$ of $F$ ($1$ denotes the unique unital representation of $\bbC$). The cycle $[(\bbC,1,0)]$ corresponding to the integer $1$ is denoted as {$1$}.
\end{example}

\begin{prop} (The Formal Periodicity) (cf. \cite{HigRoe} THEOREM 8.2.13) Let $A$ be a separable $\Calg$. Let $n$ be a nonnegative integer. Then, there is a (formal) periodicity isomorphism $K^{-n}(A)\to K^{-n-2}(A)$ which sends an element $[(\hill, \rho, F)]$ to $[(\hill\oplus\hill^{op}, \rho\hat\otimes \text{id}, F\oplus F^{op})]$. Here, $\rho\hat\otimes \text{id}$ is a representation of $(A\hat\otimes \bbC_n)\hat\otimes \bbC_2$ on $\hill\oplus\hill^{op}=\hill\hat\otimes\bbC_1$ where $\bbC_2$ is identified with $M_2(\bbC)$ acting on a graded Hilbert space $\bbC_1$.
\end{prop}

Let $\phi$ be a $\ast$-homomorphism from $A$ to $B$. We obtain a group homomorphism $K^{n}(B)\to K^{n}(A)$ which sends $[(\hill, \rho, F)]$ to $[(\hill, \rho\circ(\phi\otimes \text{id}), F)]$. In this way, $K^{n}$ becomes a (contravariant) functor from the category of separable $\Calgs$ to the category of abelian groups. Functorial properties (stability, homotopy invariance and Bott periodicity) of $K$-homology are all beautifully proved by means of the Kasparov product.

\begin{prop} (cf. \cite{HigRoe} Section 9.2) Let $A_1$ and $A_2$ be separable $\Calgs$. There is a well-defined product (the Kasparov product) on $K$-homology:
\[
\xymatrix{
K^{-n_1}(A_1) \otimes K^{-n_2}(A_2) \ar[r] & K^{-n_1-n_2}(A_1\otimes A_2) \quad \text{for $n_1,n_2\geq0$}
}
\]
The Kasparov product is bilinear, associative and functorial. It is commutative in a suitable sense. The generator ${{1}}\in K^0(\bbC)$ is the multiplicative identity of the Kasparov product. The functoriality means that the right (or left) multiplication by an element $\alpha \in K^{-n}(B)$ defines a natural transformation between functors $A\mapsto K^{-m}(A)$ and $A\mapsto K^{-m-n}(A\otimes B)$; namely, for any $\ast$-homomorphism $\phi\colon A_2\to A_1$, the following diagram commutes:
\[
\xymatrix{
K^{-m}(A_1) \ar[d]_-{\times \alpha} \ar[r]^-{\phi^\ast} &K^{-m}(A_2) \ar[d]^-{\times \alpha} \\ 
K^{-m-n}(A_1\otimes B) \ar[r]_-{(\phi\otimes \text{id})^\ast} &K^{-m-n}(A_2\otimes B) 
}
\]
See \cite{HigRoe} for details. 
\end{prop}

\begin{prop}(homotopy invariance) (cf. \cite{HigRoe} Section 9.3) The $K$-homology functors $K^n$ are homotopy invariant.
\end{prop}
\begin{proof} It can be shown that the evaluation maps $\text{ev}_0,\text{ev}_1\colon C[0,1]\to \bbC$ induce the same group homomorphism $K^0(\bbC)\to K^0(C[0,1])$, in particular $\text{ev}_0^\ast({\text{$1$}})=\text{ev}_1^\ast({\text{$1$}})$  (See also \cite{Kas1}). Using the functoriality of the Kasparov product, we have $(\text{id}_B\otimes \text{ev}_i)^\ast(\alpha)=(\text{id}_B\otimes \text{ev}_i)^\ast(\alpha\times \text{{1}})=\alpha \times \text{ev}_i^\ast(\text{{$1$}})$ for $\alpha \in K^{-n}(B)$ for $i=0,1$. This shows the desired homotopy invariance.
\end{proof}

It can be proven that the $K$-homology functors $K^{-n}$ satisfy the stability and Bott periodicity by using the Kasparov product and the homotopy invariance above. We only state the results here.

\begin{prop}(Stability) (cf. \cite{HigRoe} Section 9.4) The $K$-homology functors $K^{-n}$ are stable. In other words, a stabilization morphism $A \to A\otimes\K$ induces isomorphisms of abelian groups $K^{-n}(A\otimes\K)\to K^{-n}(A)$. The inverses are the Kasparov product by $[(\hill, \text{id}, 0)] \in K^0(\K)$ where $\hill$ is a separable infinite dimensional Hilbert space.      
\end{prop}

\begin{prop}(Bott Periodicity) (cf. \cite{HigRoe} Section 9.5) The Dirac class $d$ in $K^{-1}(C_0(-1,1))$ is defined by \[d=\left[\left(L^2[-1,1]\oplus L^2[-1,1]^{op}, \rho\hat\otimes \text{id}, \begin{pmatrix}
0 & -i(2P-I) \\
i(2P-I) & 0
\end{pmatrix} \right)\right],\] where $\rho$ is the standard representation of $C_0(-1,1)$ on $L^2[-1,1]$ sending functions to multiplication operators; $\rho\hat\otimes \text{id}$ is a representation of $C_0(-1,1)\hat\otimes\bbC_1$ on $L^2[-1,1]\oplus L^2[-1,1]^{op}$ which ``sends'' the generator (an odd selfadjoint unitary) $\epsilon \in \bbC_1$ to $\begin{pmatrix}
0 & 1 \\
1 & 0
\end{pmatrix}$ and $P$ is the projection in $L^2[-1,1]$ onto the closed subspace spanned by functions $e^{in\pi x} (n\geq 0)$. The Kasparov product by the Dirac class induces Bott periodicity isomorphisms $K^{-n}(A)\cong K^{-n-1}(A\otimes S)$. Here, we identified $S=C_0(\bbR)$ with $C_0(-1,1)$ by using an arbitrary orientation preserving homeomorphism $\bbR\cong(-1,1)$.
\end{prop}

\newpage
\section{Euivariant $KK$-Theory}

In this chapter, we will introduce Kasparov's equivariant $KK$-theory. Standard reference of Kasparov's bivariant theory is \cite{Kas1}, \cite{Kas2} and the book \cite{Bla}. 

Throughout this chapter, $G$ denotes a second countable locally compact group. 

\begin{dfn} Let $A$ and $B$ be separable graded $G$-$\Calgs$. A Kasparov $A$-$B$ module is a triple $(\E, \rho,  F)$, where $\E$ is a countably generated (as a Banach $B$-module) graded Hilbert  $G$-$B$-module; $\rho$ is a representation of a graded $G$-$\Calg$ $A$ on $\E$ (i.e.\ a $G$-equivariant graded $\ast$-homomorphism from $A$ to $B(\E)$) and $F$ is an odd adjointable $B$-linear map in $B(\E)$ satisfying the following relations:
\begin{align}\label{eq:Kas}
\rho(a)(F^2-1) \sim 0,\,\,\, \rho(a)(F-F^\ast) \sim 0,\,\,\,  [\rho(a),F] \sim 0,\\  \rho(a)(g(F)-F) \sim 0 \,\,\,\,\,\text{for $a \in A, g\in G$ } \nonumber
\end{align}
Here $[\,,\,]$ denotes the graded commutator; and for $T\in B(\E),$ $T \sim 0$ means $T$ is compact. In addition, $g\mapsto \rho(a)(g(F)-F)$ must be continuous for $a\in A, g\in G$.
\end{dfn}

If all the relations in \eqref{eq:Kas} are exact, we call a Kasparov $A$-$B$-module $(\E, \rho,  F)$ degenerate. A direct sum and unitary equivalence of Kasparov $A$-$B$-modules are defined similarly to those of Fredholm-modules. We will not distinguish between two unitary equivalent Kasparov $A$-$B$-modules.

Let $(\E, \rho,  F)$ be a Kasparov $A$-$B$-module. For a separable graded $G$-$\Calg$ $D$,\,\, using an exterior tensor product of Hilbert modules,\,\, we define a\,\, Kasparov $A\hat\otimes D$-$B\hat\otimes D$-module $\sigma_D(\E, \rho, F)$ to be $(\E\hat\otimes D, \rho\hat\otimes1, F\hat\otimes1)$. Let $\phi$ be an equivariant graded $\ast$-homomorphism from $D$ to $A$. We define a Kasparov $D$-$B$-module $\phi^\ast(\E, \rho, F)$ to be $(\E, \rho\circ\phi, F)$. \,\,If $\phi$ \,\,is\,\, an \,\,equivariant \,graded \,$\ast$-homomorphism from $B$ to $D$, we define a Kasparov $A$-$D$-module $\phi_\ast(\E, \rho, F)$ to be $(\E\hat\otimes_\phi D, \rho\hat\otimes1, F\hat\otimes1)$, here $\E\hat\otimes_\phi D$ is an interior tensor product of Hilbert modules.

A homotopy of Kasparov $A$-$B$-modules is a Kasparov $A$-$B[0,1]$-module. If there exists a Kasparov $A$-$B[0,1]$-module $(\E, \rho, F)$, we say two Kasparov $A$-$B$-modules ${\text{ev}_0}_\ast(\E, \rho, F)$ and ${\text{ev}_1}_\ast(\E, \rho, F)$ are {homotopic}. This homotopy is an equivalence relation.

\begin{dfn} Let $A$ and $B$ be separable graded $G$-$\Calgs$. The set $KK^G(A,B)$ is the set of (unitary equivalence classes of ) Kasparov $A$-$B$-modufles divided by the equivalence relation of homotopy. The set $KK^G(A,B)$ becomes a group with addition defined by direct sums of Kasparov $A$-$B$-modules. The zero class is represented by degenerate modules. The additive inverse of $[(\E, \rho, F)]$ is $[(\E^{op}, \rho^{op}, -F^{op})]$ (the latter module is defined analogously to the case of Fredholm modules). When $G=1$, we usually denote the Kasparov group by $KK(A,B)$ instead of $KK^G(A,B)$. 
\end{dfn}

We defined a map from the set of Kasparov $A$-$B$-modules to the set of Kasparov $D$-$B$-modules (resp.\ $A$-$D$-modules) for an equivariant graded $\ast$-homomorphism from $D$ to $A$ (resp.\ $B$ to $D$). One can check this defines a group homomorphism from $KK^G(A,B)$ to $KK^G(D,B)$ (resp.\ to $KK^G(A,D)$). In this way, $KK^G(\,,\,)$ becomes a bi-functor (contravariant in the first variable and covariant in the second) from the category of graded $G$-$\Calgs$ to the category of abelian groups. Similarly, we have a group homomorphism $\sigma_D$ from $KK^G(A,B)$ to $KK^G(A\hat\otimes D,B\hat\otimes D)$ which is natural in both variables.

Homotopy invariance is almost incorporated in the definition of the Kasparov groups.

\begin{prop} (cf. \cite{Bla} Proposition 17.9.1.) The bi-functor $KK^G(\,,\,)$ is homotopy invariant in both variables.
\end{prop}

When $G=1$, the Kasparov group $KK(A\hat\otimes\bbC_n,\bbC)$ is nothing but the $K$-homology group $K^{-n}(A)$. The following proposition explains that the Kasparov group generalizes both $K$-theory and $K$-homology of $\Calgs$. 

\begin{prop} (cf. \cite{Bla} Proposition 17.5.5.) Let $G=1$ and $B$ be a separable ungraded $\Calg$. The Kasparov group $KK(\bbC, B)$ is isomorphic to the $K_0$ group $K_0(B)$ of $B$.   
\end{prop}

\begin{dfn} Let $D$ and $B$ be a separable graded $G$-$\Calg$. Let $\E_1$ be a countably generated graded $G$-$D$-Hilbert module and $(\E_2, \rho, F_2)$ be a Kasparov $D$-$B$-module. Define an adjointable map $T_{e_1}\colon \E_2 \to \E_1\hat\otimes_\rho\E_2$ for $e_1$ in $\E_1$ by $T_{e_1}\colon e_2 \mapsto e_1\hat\otimes e_2$. An odd adjointable $B$-linear map $F$ in $B(\E_1\hat\otimes_\rho\E_2)$ is an $F_2$-connection if for any $e_1$ in $\E_1$, the following diagrams graded commute modulo compact operators.
\[
\xymatrix{
\E_2 \ar[d]_-{F_2} \ar[r]^-{T_{e_1}} & \E_1\hat\otimes_\rho\E_2 \ar[d]^-{F} & \E_2 \ar[d]_-{F_2^\ast} \ar[r]^-{T_{e_1}} & \E_1\hat\otimes_\rho\E_2 \ar[d]^-{F^\ast} \\
\E_2 \ar[r]_-{T_{e_1}} & \E_1\hat\otimes_\rho\E_2 & \E_2 \ar[r]_-{T_{e_1}} & \E_1\hat\otimes_\rho\E_2 
}
\]
\end{dfn} 

\begin{prop} (cf. \cite{Kas2} Theorem 2.14.) Let $A_1,A_2$, $D$ and $B_1,B_2$ be separable graded $G$-$\Calgs$. There is a bilinear pairing of the Kasparov groups:
\begin{align*}
KK^G(A_1,B_1\hat\otimes D)\times KK^G(D\hat\otimes A_2 ,B_2) \to KK^G(A_1\hat \otimes A_2,B_1\hat\otimes B_2)
\end{align*}
This pairing is associative, functorial in both variable and commutative when $D=\bbC$. We write the product of $\alpha \in KK^G(A_1,B_1\hat\otimes D)$ and $\beta \in KK^G(D\hat\otimes A_2,B_2)$ by  $\alpha\otimes_D\beta$. For any separable graded $G$-$\Calg$ $A$, $KK^G(A,A)$ is a ring with unit $1_A$ represented by a cycle $(A, \text{id}_A, 0)$. For further properties of the product, see \cite{Kas2}.
\end{prop}

When $B_1=A_2=\bbC$, at the level of cycles, a product of a Kasparov $A$-$D$-module $(\E_1, \rho_1, F_1)$ and a Kasparov $D$-$B$-module $(\E_2, \rho_2, F_2)$ is defined as a Kasparov $A$-$B$-module $(\E_1\hat\otimes_{\rho_2}\E_2, \rho_1\hat\otimes1, F)$ where the operator $F \in B(\E_1\hat\otimes_{\rho_2}\E_2)$ is an $F_2$-connection and satisfies $(\rho_1\hat\otimes1)(a)[F_1\hat\otimes1, F](\rho_1\hat\otimes1)(a)^\ast\geq0$ for $a$ in $A$. A general case is defined as above after applying $\sigma_{A_2}$ and $\sigma_{B_1}$.  

When $G$ acts trivially, there is a simple case where one has a good formula for the Kasparov product:

\begin{prop} (cf. \cite{Bla} Proposition 18.10.1.) Let $A$, $D$ and $B$ be separable graded $\Calgs$. Let $(\E_1, \rho_1, F_1)$ be a  Kasparov $A$-$D$-module with $F_1$ being selfadjoint and contractible and $(\E_2, \rho_2, F_2)$ be a Kasparov $D$-$B$-module. Let $F\in B(\E_1\hat\otimes_{\rho_2}\E_2)$ be an $F_2$-connection. Assume  $(\E_1\hat\otimes_{\rho_2}\E_2, \rho_1\hat\otimes1, F_1\hat\otimes1+((1-F_1^2)\hat\otimes1)^{\frac12}F)$ is a Kasparov $A$-$B$-module. Then, this Kasparov $A$-$B$-module defines the same class in $KK^G(A,B)$ as the Kasparov product of $(\E_1, \rho_1, F_1)$ and $(\E_2, \rho_2, F_2)$.
\end{prop} 

\begin{dfn} Separable graded $G$-$\Calgs$ $A$ and $B$ are {{$KK^G$}-equivalent} if there exist $\alpha \in KK^G(A,B)$ and $\beta \in KK^G(B,A)$ such that $\alpha \otimes_B \beta=1_A$ and $\beta \otimes_A \alpha=1_B$. In other words, $A$ and $B$ are $KK^G$-equivalent if they are isomorphic in the additive category of separable graded $G$-$\Calgs$ with morphisms $KK^G(A,B)$. We denote by $KK^G$ its full subcategory consisting of separable (trivially graded) $G$-$\Calgs$. We call this additive category $KK^G$ as Equivariant Kasparov's category.
\end{dfn}

\begin{example} Let $A$ and $B$ be separable graded $G$-$\Calgs$. An $A$-$B$ imprimitivity bimodule is a full graded Hilbert $G$-$B$-module $\E$ with a graded $G$-equivariant isomorphism $\rho_A\colon A\cong \K(\E)$. We say $A$ and $B$ are Morita-Rieffel equivalent if there exists an $A$-$B$ imprimitivity bimodule. This is an equivalence relation; if $\E$ is an $A$-$B$ imprimitivity bimodule, then $\E^\ast=\K(\E,B)$ with the left multiplication $\rho_B$ by $B$ becomes a $B$-$A$ imprimitivity bimodule. Morita-Rieffel equivalence implies $KK^G$-equivalence; an isomorphism is given by $[(\E, \rho_A, 0)]$ and $[(\E^\ast, \rho_B, 0)]$.
\end{example}

As a corollary, we see, for any separable graded $G$-Hilbert space $\hill$, $\K(\hill)$ is $KK^G$-equivalent to $\bbC$. This is more or less implying the stability of the bifunctor $KK^G(\,,\,)$. Take any projection $p\in \K(\hill)$ onto a one-dimensional even subspace of $\hill$ with trivial $G$-action. One sees the stabilization $\ast$-homomorphism $\rho$ from $\bbC$ to $\K(\hill)$ given by $p$ defines a element $\rho=[(\K(\hill), \rho, 0)] \in KK^G(\bbC, \K(\hill))$ which is a left inverse of the element $[(\hill,\text{id}_{\K(\hill)},0)] \in KK^G(\K(\hill),\bbC)$ implementing Morita-Rieffel equivalence between $\bbC$ and $\K(\hill)$; hence $\rho$ is invertible. A general stabilization $\sigma_A(\rho)$ is invertible by functoriality of the Kasparov product. 

Before going to prove Bott-periodicity in our quite general context, we state a lemma which generalizes the rotational argument used by M. Atiyah.

\begin{lemma} Let $A$ be a separable graded $G$-$\Calg$. Assume $A$ has the following property: the flip isomorphism $A\hat\otimes A\to A\hat\otimes A$ is $\pm1$ in the group $KK^G(A\hat\otimes A, A\hat\otimes A)$. Suppose one finds $\alpha \in KK^G(\bbC, A)$ and  $\beta \in KK^G(A, \bbC)$ such that $\alpha\otimes_A\beta=1_{\bbC}$. Then, $A$ is $KK^G$-equivalent to $\bbC$.
\end{lemma}
\begin{proof} This follows from the fact that the following diagram commutes:
\[
\xymatrix{
\bbC \hat\otimes A \ar[d]_{1\hat\otimes\beta} \ar[r]^{\alpha\hat\otimes1} & A\hat\otimes A \ar[d]^{1\hat\otimes \beta} \ar[r]^{\text{flip}} & A\hat\otimes A \ar[d]^{\beta\hat\otimes1} \\
\bbC \hat\otimes \bbC \ar[r]_{\alpha\hat\otimes1} \ar@/_15pt/[rr]_{1\hat\otimes \alpha} & A\hat\otimes \bbC \ar[r]_{\text{flip}} & \bbC\hat\otimes A 
}
\]
The author would like to thank Nigel Higson for showing him this diagram.
\end{proof}

\begin{prop} (Bott periodicity) (cf. \cite{Bla} Section 19.2.) A (trivially graded) separable $\Calg$ $S^2$ is $KK^G$-equivalent to $\bbC$.
\end{prop}
\begin{proof} It suffices to show that a graded $\Calg$ $S\hat\otimes\bbC_1$ is $KK^G$-equivalent to $\bbC$; and in view of the previous lemma, this follows once we have shown that the Dirac class $d \in K^{-1}(C_0(-1,1))=KK^G(C_0(-1,1)\hat\otimes\bbC_1,\bbC)$ is right invertible. Let $s=[(C_0(-1,1)\hat\otimes\bbC_1, 1, x\hat\otimes\epsilon)] \,\,\, \in KK^G(\bbC, C_0(-1,1)\hat\otimes\bbC_1)$. The product $s\otimes_{C(-1,1)\hat\otimes\bbC_1}d \in KK^G(\bbC, \bbC)=K^0(\bbC)$ is represented by a Fredholm module \[\left( L^2[-1,1]\oplus L^2[-1,1]^{op}, 1, \begin{pmatrix}
0 & x-i(1-x^2)^\frac12(2P-I) \\
x+i(1-x^2)^\frac12(2P-I) & 0
\end{pmatrix} \right).\] By Example \ref{prop:KC}, we must calculate the Fredholm index of the operator $x+i(1-x^2)^\frac12(2P-I) \in L^2[-1,1]$. Using the straight line homotopy between $x$ and $\sin\frac{\pi}{2}x$, we see that this is same as $\text{Index}(\sin\frac{\pi}{2}x+i\cos\frac{\pi}{2}x(2P-I))$ which in turn is same as $\text{Index}(P+e^{-i\pi x}(I-P))=-1$. This shows $d$ is right invertible.
\end{proof}

The proof above showed that the $\Calg$ $S$ of continuous functions on the real line is $KK^G$-equivalent to the first Clifford $\Calg$ $\bbC_1$. We define for any graded $G$-$\Calgs$ $A,B$, the even $G$-equivariant Kasparov group $KK^G_0(A,B)=KK^G(A,B)$ and the odd $G$-equivariant Kasparov group \[KK^G_1(A,B)=KK^G(A\hat\otimes\bbC_1,B)=KK^G(A,B\hat\otimes\bbC_1)=KK^G(A\hat\otimes S,B)=KK^G(A,B\hat\otimes S).\] Thanks to the Bott Periodicity, the odd group $KK^G_1(A\hat\otimes S,B)$ is naturally isomorphic to $KK^G_0(A,B)$ for any graded $G$-$\Calgs$ $A,B$.

In the following discussions, we will essentially consider the $G$-equivariant $KK$-theory of ungraded $G$-$\Calgs$. In this case, there is a different but useful description of even and odd $G$-equivariant Kasparov groups. 

\begin{dfn} Let $A$ and $B$ be separable (ungraded) $G$-$\Calgs$. An even Kasparov $A$-$B$ module is a triple $(\E, \rho,  F)$, where $\E$ is a countably generated (ungraded) Hilbert  $G$-$B$-module; $\rho$ is a representation of a $G$-$\Calg$ $A$ on $\E$ and $F$ is an adjointable $B$-linear map in $B(\E)$ satisfying the following relations:
\begin{align}\label{eq:evenKas}
\rho(a)(FF^\ast-1) \sim 0,\,\,\,  \rho(a)(F^\ast F-1) \sim 0,\,\,\,  [\rho(a),F] \sim 0,\\  \rho(a)(g(F)-F) \sim 0 \,\,\,\,\,\text{for $a \in A, g\in G$ }\nonumber
\end{align}
In addition, $g\mapsto \rho(a)(g(F)-F)$ must be continuous for $a\in A, g\in G$. Simply put, the map $F$ is essentially $G$-equivariant, essentially unitary and essentially commuting with the representation of $A$. An odd Kasparov $A$-$B$ module is a triple $(\E, \rho,  P)$, where $\E$ is a countably generated (ungraded) Hilbert  $G$-$B$-module; $\rho$ is a representation of a $G$-$\Calg$ $A$ on $\E$ and $P$ is an adjointable $B$-linear map in $B(\E)$ satisfying the following relations:
\begin{align}\label{eq:oddKas}
\rho(a)(P^\ast-P) \sim 0,\,\,\,  \rho(a)(P^2-P) \sim 0,\,\,\,  [\rho(a),P] \sim 0,\\  \rho(a)(g(P)-P) \sim 0 \,\,\,\,\,\text{for $a \in A, g\in G$ } \nonumber
\end{align}
In addition, $g\mapsto \rho(a)(g(P)-P)$ must be continuous for $a\in A, g\in G$. Simply put, the map $P$ is essentially $G$-equivariant, an essentially projection and essentially commuting with the representation of $A$. 
\end{dfn}

All the notion defined for Kasparov $A$-$B$-modules, namely addition, unitary equivalence, functoriality, homotopy e.t.c., are defined for even and odd Kasparov $A$-$B$-modules. We will not distinguish between two unitary equivalent Kasparov $A$-$B$-modules. The set of homotopy equivalence classes of even (odd) Kasparov $A$-$B$-modules is a group with the obvious addition (the direct sum) of even (odd) Kasparov $A$-$B$-modules. The following proposition gives the nice description of $G$-equivariant Kasparov groups for ungraded $G$-$\Calgs$ mentioned earlier.

\begin{prop} For any separable $G$-$\Calgs$ $A,B$, the group of homotopy equivalence classes of even Kasparov $A$-$B$-modules is naturally isomorphic to the even $G$-equivariant Kasparov group $KK^G_0(A,B)=KK^G(A,B)$. The isomorphism takes an even Kasparov $A$-$B$-module $(\E, \rho, P)$ to a Kasparov $A$-$B$-module $\left(\E\oplus\E^{op}, \rho\otimes1, \begin{pmatrix}
0 & F^\ast \\
F & 0
\end{pmatrix}\right)$. The group of homotopy equivalence classes of odd Kasparov $A$-$B$-modules is naturally isomorphic to the odd $G$-equivariant Kasparov group $KK^G_1(A,B)=KK^G(A\hat\otimes\bbC_1,B)$. The isomorphism takes an odd Kasparov $A$-$B$-module $(\E, \rho, F)$ to a Kasparov $A\hat\otimes\bbC_1$-$B$-module $\left(\E\oplus\E^{op}, \rho\hat\otimes\text{id$_{\bbC_1}$}, \begin{pmatrix}
0 & -i(2P-1) \\
i(2P-1) & 0
\end{pmatrix}\right)$.  Here, we are identifying the graded Hilbert $B$-module $\E\oplus\E^{op}$ as $\E\hat\otimes\bbC_1$.
\end{prop}

\begin{example}\label{exampleBott} Let $A,B$ be separable $G$-$\Calgs$ and $\Sigma=C_0(0, 1)\cong S$. We simply write $B\otimes\Sigma$ by $B\Sigma$ for example. Let $x=(\E, \phi, P)$ be an element of $KK^G_1(A, B)$. We would like to compute the element in $KK^G_0(A , B\Sigma)$ which corresponds to the element $x$ under the Bott Periodicity $KK^G_1(A, B)\cong KK^G_0(A, B\Sigma)$. In such computations, we frequently use the Formal Periodicity (or just Morita Equivalence) such as $KK^G(A, B)\cong KK^G(A\hat\otimes\bbC_2, B)$. Hence, it is always safer and easier to say that we compute in up to sign precision. We recall $x$ is represented as $\left(\E\oplus\E^{\op}, \phi\hat\otimes\id_{\bbC_1}, \begin{pmatrix}
0 & -i(2P-1) \\
i(2P-1) & 0
\end{pmatrix}\right)$ as the element in $KK^G(A\hat\otimes\bbC_1, B)$. The Bott Periodicity maps this element to the Kasparov product $x\otimes_\bbC\left(\Sigma\oplus\Sigma^{\op}, \id_{\bbC_1}, \begin{pmatrix}
0 & -i(2x-1) \\
i(2x-1) & 0
\end{pmatrix}\right)$ in $KK^G(A\hat\otimes\bbC_1\hat\otimes\bbC_1, B\Sigma)$ which can be computed as $\left((\E\oplus\E^{\op})\hat\otimes(\Sigma\oplus\Sigma^{\op}), \phi\hat\otimes\id_{\bbC_1}\hat\otimes\id_{\bbC_1}, T\right)$ where \begin{small}\[T=
\begin{pmatrix}
0 & -i(2P-1)  \\
i(2P-1) & 0 
\end{pmatrix} \hat\otimes\begin{pmatrix}
2(x-x^2)^{-\frac12} & 0  \\
0 & 2(x-x^2)^{-\frac12} 
\end{pmatrix} 
+1\hat\otimes\begin{pmatrix}
0 & -i(2x-1)  \\
i(2x-1) & 0
\end{pmatrix}.\]  \end{small}After the identification $KK^G(A\hat\otimes\bbC_2, B)\cong KK^G(A, B)$, this element can be represented by $(\E\otimes\Sigma, \phi\otimes1, 1\otimes(2x-1)+i(2P-1)\otimes2(x-x^2)^{-\frac12})$ in $KK^G_0(A, B\Sigma)$. Applying first the straight line homotpy between $x$ and $\sin^2\frac\pi2x$, next multiplying $-1$ and last multiplying a unitary $1\otimes e^{-i\pi x}$ (which is homotopic to 1), we see that the element can be written by the following (probably) simplest form $(\E\otimes\Sigma, \phi\otimes1, P\otimes e^{2\pi ix}+(1-P)\otimes1)$ in $KK^G_0(A, B\Sigma)$. Similarly, the element in $KK^G_1(A, B\Sigma)$ which corresponds to an element $(\E, \phi, F)$ in $KK^G_0(A, B)$ under the Bott Periodicity can be computed as $\left(\E\Sigma\oplus\E\Sigma, \phi\otimes1, \begin{pmatrix}
1\otimes x &  F^\ast\otimes(x-x^2)^{\frac12}\\
F\otimes(x-x^2)^{\frac12}   & 1\otimes(1-x) 
\end{pmatrix} \right)$. Finally, let us compute the element in $KK^G_0(A, B)$ which corresponds to an element $y$ in $KK^G_1(\Sigma A, B)$ under the Bott Periodicity. We suppose $y=(\E, \phi, P)$ where $\phi$ is a nondegenerate representation of $\Sigma A$ on $\E$; this ensures that we can write $\phi$ as $\phi_\Sigma\otimes\phi_A$ where $\phi_\Sigma$ and $\phi_A$ are commuting, nondegenerate representations of $\Sigma$ and $A$ respectively. We remark that any Kasparov module is homotopic to such one (such one is called an essential Kasparov module). Again, note that $y$ is represented as $\left(\E\oplus\E^{\op}, \phi\hat\otimes\id_{\bbC_1}, \begin{pmatrix}
0 & -i(2P-1) \\
i(2P-1) & 0
\end{pmatrix}\right)$ in $KK^G(\Sigma A\bbC_1, B)$. The Bott Periodicity maps this element to $\left(\Sigma\hat\otimes\bbC_1, 1, x\hat\otimes\epsilon 
\right)\otimes_{\Sigma\bbC_1}y$ ($\epsilon$ is the standard generator of $\bbC_1$). We compute this to get $(\E, \phi_A, (2x-1)+2i(x-x^2)^{\frac12}(2P-1))$ in $KK^G_0(A, B)$ where $x$ is an operator on $\E$ obtained by extending the nondegenerate representation $\phi_\Sigma$ of $\Sigma$ to that of $C_b(0,1)$. A similar calculation as above leads us to get probably the simplest form $(\E, \phi_A, Pe^{2i\pi x}+1-P)$. These computations will be used in Chapter 9.

\end{example}

We recall here Equivariant Kasparov's category $KK^G$. It is the additive category whose objects are separable $G$-$\Calgs$ and morphisms are the elements in the Kasparov groups $KK^G(A, B)$ for separable $G$-$\Calgs$ $A,B$.  Also, we frequently denote by $KK^G$ the bifunctor $(A, B)\mapsto KK^G(A, B)$ from the category of separable $G$ $\Calg$ to abelian groups. We mentioned that this functor is stable and homotopy invariant. It can be shown that the functor $KK^G$ is split-exact in both variables. In \cite{Mey}, R. Meyer elegantly showed the following universal property of the category $KK^G$. 

\begin{theorem}(cf. \cite{Mey} Theorem 6.6.)\label{universal} Let $F$ be any stable, homotopy invariant, split exact (covariant or contravariant) functor from the category of separable $G$-$\Calg$ to an additive category. Then, the functor $F$ uniquely factors through the category $KK^G$.
\end{theorem}

One may want to consider whether the bifunctor $KK^G(\cdot, \cdot)$ from the category of separable $G$-$\Calgs$ to the category of abelian groups is half-exact in either variable. Unfortunately, this is not true in general. However, we have a following.

\begin{prop}(cf. \cite{bolic} PROPOSITION 5.7.) Let $A$ be a proper, nuclear $G$-$\Calg$, then the functor $KK^G(A, \cdot)$ is half-exact.
\end{prop}
 
It is also true that the functor $KK^G(\cdot, A)$ is half-exact for any proper, nuclear $G$-$\Calg$. Thanks to the above proposition, we have for any proper $G$-$\Calg$ $A$ and for any $G$-extension:
\begin{align*}
\xymatrix{
0 \ar[r] & J \ar[r] & B \ar[r] & B/J \ar[r] & 0 \\
}
\end{align*} the six-term exact sequence in Equivariant $KK$-Theory:
\begin{align*}
\xymatrix{
KK^G_0(A, J) \ar[r] & KK^G_0(A, B) \ar[r] & KK^G_0(A, B/J) \ar[d] \\
KK^G_1(A,  B/J) \ar[u] & KK^G_1(A,  B) \ar[l] & KK^G_1(A, J) \ar[l] \\  
}
\end{align*} 

\newpage
\section{Asymptotic Morphisms and Equivariant $E$-Theory}

This chapter introduces asymptotic morphisms which are now been regarded as another fundamental tool for calculating the $K$-theory of $\Calgs$. The importance of this notion comes from the fact that associated to any extension of $\Calgs$, there is a canonical asymptotic morphism called a central invariant which is unique up to suitable equivalence relation (i.e. homotopy). We follows the treatment given in \cite{GHT}.

\begin{dfn}(asymptotic algebra) Let $B$ be a separable $G$-$\Calg$. The $G$-$\Calg$ $\T_0(B)=C_0([1, \infty), B)$ of continuous functions from the interval $[0, \infty)$ to $B$ which vanish at infinity sits as a $G$-invariant ideal in the $\Calg$ of $C_b([1, \infty), B)$ of bounded functions from $[0, \infty)$ to $B$ with a natural pointwise $G$-action. We denote by $\T(B)$ the subalgebra of $G$-continuous elements in $C_b([1, \infty), B)$; this is a $G$-$\Calg$ containing $\T_0(B)$ as a $G$-equivariant ideal. The asymptotic algebra $\A(B)$ of $B$ is the quotient $G$-$\Calg$ $\T(B)/\T_0(B)$.
\end{dfn}

\begin{dfn}(asymptotic morphisms)\label{dfn:asym} For separable $G$-$\Calgs$ $A,B$, an equivariant asymptotic morphism from $A$ to $B$ is an equivariant $\ast$-homomorphism from $A$ to the asymptotic algebra $\A(B)$. An equivariant asymptotic morphism $\phi$ from $A$ to $B$ is denoted by $\phi\colon A\to\to B$. A homotopy of equivariant asymptotic morphisms from $A$ to $B$ is an equivariant asymptotic morphism from $A$ to $B[0,1]$. We denote by $[[A, B]]_G$ the set of homotopy equivalence classes of equivariant asymptotic morphisms from $A$ to $B$. Any element of $[[A, B]]_G$ can be represented as a family of continuous maps $\phi_t\colon A\to B$ $(t\geq1)$ satisfying the following conditions (such a map is called an equicontinuous equivariant asymptotic morphism from $A$ to $B$ with a slight abuse of language).
\begin{itemize}
\item for any $a$ in $A$, the map $t\mapsto\phi_t(a)$ is in $\T(B)$;
\item $(g,a)\mapsto g(\phi_t((a)))$ is a continuous map from $G\times A$ to $B$ uniformly in $t \in [0, \infty)$;
\item the map $A\to \T(B) \to \A(B)$ given by composition of the map $a\mapsto (\phi_t(a))$ with the quotient map from $\T(B)$ to $\A(B)$ is an equivariant $\ast$-homomorphism.
\end{itemize}
Hence, we frequently write an element of $[[A, B]]_G$ as an equicontinuous equivariant asymptotic morphism $(\phi_t)_{t\geq1}$ without any fear of confusion. Any continuous family $(\phi_t)_{t\geq1}$ of equivariant $\ast$-homomorphisms from $A$ to $B$ defines an equivariant asymptotic morphism from $A$ to $B$ in an obvious way. More generally, a continuous family $(\phi_t)_{t\geq}$ of $\ast$-homomorphisms from $A$ to $B$ defines an  equivariant asymptotic morphism from $A$ to $B$, if the family is asymptotically equivariant (meaning for any $a$ in $A$ and for any $g$ in $G$, $\phi_t(g(a))-g(\phi_t(a))$ converges to $0$ as $t$ goes to infinity). There is a well-defined ``composition'' operation $[[A, B]]_G\times[[B, C]]_G\to[[A, C]]_G$ given by a composition after a reparametrization of asymptotic morphisms: more specifically, for given two equicontinuous equivariant asymptotic morphisms $(\phi_t)_{t\geq1}\colon A\to\to B$ and $(\psi_t)_{t\geq1}\colon B\to\to C$, there is a strictly increasing continuous function $r$ from $[1, \infty)$ onto $[1, \infty)$ such that the composition $(\psi_{s(t)}(\phi_{t}))_{t\geq1}$ defines an equivariant asymptotic morphism from $A$ to $C$ for any reparametrization in $t$ (strictly increasing continuous functions from $[1, \infty)$ to $[1, \infty)$) satisfying $s(t)\geq r(t)$ for all $t$; and this defines a well-defined operation $[[A, B]]_G\times[[B, C]]_G\to[[A, C]]_G$. We write the composition of two asymptotic  morphisms $\phi\colon A\to\to B$ and $\psi\colon B\to\to C$ by $\psi \circ \phi$ as long as it makes no confusion. The set $[[A, B\K(\hill_G)]]_G$ can be endorsed with an abelian semigroup structure by the following way. (Here, $B\K(\hill_G)=B\otimes\K(\hill_G)$.) The addition operation comes from an (equivariant) embedding of $\K(\hill_G)\oplus\K(\hill_G)\subset\K(\hill_G\oplus\hill_G)$ into $\K(\hill_G)$ induced from an (equivariant) embedding of $\hill_G\oplus\hill_G$ into $\hill_G$. We may use any embedding here, since any pair of such embeddings can be connected through a homotopy of embeddings. With these in mind, associativity and commutativity of this operation is clear. The zero element is represented by $0$ morphism. The semigroup $[[\Sigma A, B\K(\hill_G)]]_G$ becomes a group thanks to the presence of $\Sigma$: the inverse operation is defined by the composition with a $\ast$-homomorphism $h^\ast\otimes\id_A\colon\Sigma A\to \Sigma A$ where $h^\ast\colon \Sigma\to \Sigma$ is induced from an order reversing homeomorphism $h\colon s\mapsto 1-s$ on $(0,1)$.
\end{dfn}

\begin{dfn}(Equivariant $E$-Theory) Let\, $A$ \, and\, $B$\, be\, separable\, $G$-$\Calgs$.\, The\, Equivariant\, $E$-theory\, group\, $E^G(A, B)$\,\,\, is\,,\ defined\,\, to\, be\, \,the\, \,abelian\, \,group\, \,$[[\Sigma A\K(\hill_G), \Sigma B\K(\hill_G)]]_G$. The composition of asymptotic morphisms defines a bilinear map $E^G(A, B)\times E^G(B,C)\to E^G(A, C)$. We define an additive category $E^G$ to be the category which has separable $G$-$\Calgs$ as objects and an $E$-theory group $E^G(A, B)$ as the morphism group from $A$ to $B$. We call the category $E^G$ as the Equivariant $E$-Theory category.
\end{dfn}

We have a functor from category of separable $G$-$\Calg$ to $E^G$ which is identity on objects and sends an equivariant-$\ast$-homomorphism $\phi$ to the class of equivariant asymptotic morphisms $\id_\Sigma\otimes\phi\otimes\id_{\K(\hill_G)}$ in $E^G(A, B)$. For any nuclear $G$-$\Calg$ $D$, we have a tensor product functor $\sigma_D$ on the category $E^G$ coming from the operation $\sigma_D\colon [[A, B]]_G\to [[A\otimes D, B\otimes D]]_G$ which is the tensor product by $\id_D$ at the level of cycles. There is a The bifunctor $(A,B)\mapsto E^G(A,B)$ from the category of separable $G$-$\Calgs$ to the category of abelian groups are homotopy invariant and stable. Moreover, it is half-exact in both variable with respect to any $G$-extension. We are mostly interested in half-exactness in the second variable. We first define an canonical equivariant asymptotic morphism associated to a $G$-extension.

\begin{dfn} For any $G$-extension: \begin{align}\label{gext}
\xymatrix{
0 \ar[r] & J \ar[r] & B \ar[r]^-{\pi} & B/J \ar[r] & 0 \\
}
\end{align}an approximate unit for the $G$-extension \eqref{gext} is a continuous approximate unit $(u_t)_{t\geq1}$ of $J$ satisfying the following: \begin{itemize}
\item it is asymptotically equivariant: that is, for any $g$ in $G$, $g(u_t)-u_t\to0$ as $t\to0$ uniformly over compact subsets of $G$;
\item it asymptotically commutes with elements in $B$: that is, for any $b\in B$, $[a, u_t]=bu_t-u_tb\to0$ as $t\to0$.
\end{itemize} Such an approximate unit always exists, and any two can be connected through the obvious straight line path. We call an approximate unit having the second property above as quasicentral with respect to $B$; and by an approximate unit for the pair of $\Calgs$ $B_1\subset B_2$, we usually mean an approximate unit of $B_1$ which is quasicentral with respect to $B_2$.
\end{dfn}

\begin{dfn} For any $G$-extension \eqref{gext}, a central invariant for the $G$-extension \eqref{gext} is an equivariant asymptotic morphism from $\Sigma (B/J)$ to $J$ defined by $f\otimes x\mapsto f(u_t)s(x)$ for $t\geq1$ using any set-theoretic section $s$ of the quotient map $\pi\colon B\to B/J$. A central invariant for the $G$-extension \eqref{gext} defines a unique element in the set $[[\Sigma (B/J), J]]_G$ independent of choices of an approximate unit and of a section $s$ which we also call as the central invariant for the $G$-extension \eqref{gext}.
\end{dfn}

Central invariants are natural in the following sense. Suppose we have a following digram of $G$-extension:
\begin{align*}
\xymatrix{
0 \ar[r] & J \ar[d]^-{q}\ar[r] & B \ar[d]\ar[r] & B/J \ar[d]^{p}\ar[r] & 0 \\
0 \ar[r] & J' \ar[r] & B' \ar[r] & B'/J' \ar[r] & 0 \\
}
\end{align*} Then, if we denote the central invariant for the first row and for the second row by $x\in [[\Sigma (B/J), J]]_G$ and $x' \in [[\Sigma (B'/J'),  J']]_G$ respectively, we have $x'\circ\Sigma p=q\circ x$ in $[[\Sigma (B/J), J']]_G$.

\begin{example} Let $A$ be a separable $G$-$\Calg$. Consider the following $G$-extension:\begin{align}\label{gext2}
\xymatrix{
0 \ar[r] & \Sigma A \ar[r] & A(0,1] \ar[r] & A \ar[r] & 0 \\
}
\end{align} The central invariant associated to the $G$-extension \eqref{gext2} and the class defined by $\id_{\Sigma A}$ coincide in the group $[[\Sigma A, \Sigma A]]_G$.
\end{example}

We now state the important property of the bifunctor $E^G$.
\begin{theorem} (cf. \cite{GHT} THEOREM 6.20.) The bifunctor $(A,B) \mapsto E^G(A, B)$ is half-exact in both variables.
\end{theorem}

In the same way as for the $K$-theory functor, the half-exactness together with the stability and the homotopy invariance of the functor $E^G$ automatically implies Bott-Periodicity that is, an isomorphism between $\Sigma^2$ and $\bbC$ in the category $E^G$. For any $G$-extension \eqref{gext}, and for any separable $G$-$\Calg$ $A$, we have six-term exact sequences:
\begin{align*}
\xymatrix{
E^G(A, J) \ar[r] & E^G(A, B) \ar[r] & E^G(A, B/J) \ar[d] \\
E^G(A, \Sigma (B/J)) \ar[u] & E^G(A, \Sigma B) \ar[l] & E^G(A, \Sigma J) \ar[l] \\  
}
\end{align*}
and,
\begin{align*}
\xymatrix{
E^G(J, A) \ar[d] &  E^G(B, A) \ar[l] & E^G(B/J, A) \ar[l] \\
E^G(\Sigma (B/J), A) \ar[r] & E^G(\Sigma B, A) \ar[r] & E^G(\Sigma J, A) \ar[u] \\  
}
\end{align*}
In the above sequences, the boundary maps are given by the composition by the central invariant associated to the extension \eqref{gext}.

Just as in the case of the functor $KK^G$, the bifunctor $E^G$ has a following universal property which can be shown purely categorically using the property of $E^G$ listed so far.

\begin{theorem}(cf. \cite{AsympE} Theorem 1.13., \cite{GHT})\label{universalE} Let $F$ be any stable, homotopy invariant, half-exact (covariant or contravariant) functor from the category of separable $G$-$\Calg$ to an additive category. Then, the functor $F$ uniquely factors through the category $E^G$.
\end{theorem}

Now, it is time to relate the Equivariant $KK$-Theory with Equivariant $E$-Theory. Note, since $E^G$ is stable, homotopy invariant and split exact, by the universal property of $KK^G$, the canonical functor from the category of separable $G$-$\Calgs$ to the Equivariant $E$-Theory category $E^G$ uniquely factors through Kasparov's category $KK^G$. We first describe this unique functor from the category $KK^G$ to the category $E^G$. Next, we see the important fact that when $A$ is a separable $G$-$\Calg$ such that the functor $B\mapsto KK^G(A, B)$ is half exact, the abelian groups $KK^G(A, B)$ and $E^G(A,B)$ is isomorphic for any separable $G$-$\Calg$ $B$ via this unique functor from $KK^G$ to $E^G$.
 
 Let $A, B$ be any separable $G$-$\Calgs$. We define a homomorphism from $KK^G_1(A, B)$ to $E^G(\Sigma A, B)=[[\Sigma^2 A\K(\hill_G), \Sigma B\K(\hill_G)]]_G$ by the following way. Take any odd Kasparov $A$-$B$-module $x=(\E, \phi, P)$. We have a canonical pullback extension associated to $x$:\begin{align}\label{extpicture2}
\xymatrix{
0 \ar[r] & \K(\E) \ar@{=}[d]\ar[r] & E_{\phi'} \ar[d]\ar[r] & A \ar[d]^-{\phi'}\ar[r] & 0\\
0 \ar[r] & \K(\E) \ar[r] & M(B) \ar[r]  & Q(B) \ar[r] & 0\\
}
\end{align} Denote by $c$, the central invariant in the group $[[\Sigma A, \K(\E)]]_G$ associated to the $G$-extension \eqref{extpicture2}. We tensor this element $c$ with $\id_\Sigma$ and $\id_{\K(\hill_G)}$ to obtain the element $\id_{\Sigma}\otimes c \otimes\id_{\K(\hill_G)}$ in the group $[[\Sigma^2 AK(\hill_G), \Sigma \K(\E\otimes\hill_G)]]_G$. Finally, using any $G$-embedding $\E\otimes\hill_G\to B\otimes\hill_G$, we map this element to $[[\Sigma^2 AK(\hill_G), \Sigma B\K(\hill_G)]]_G=E^G(\Sigma A, B)$. This procedure obviously respects homotopy of Kasparov-modules. In this way, we obtained the homomorphism from $KK^G_1(A, B)$ to $E^G(\Sigma A, B)$. It is easy to see, for any fixed $B$, the homomorphisms $KK^G_1(A, B)$ to $E^G(\Sigma A, B)$ defines a natural transformation between two functors $KK^G_1(\cdot, B)$ and $E^G(\Sigma\cdot, B)$ from the category of separable $G$-$\Calgs$ to the category of abelian groups. Using the Bott Periodicity $KK^G(A, B)\cong KK^G_1(A, \Sigma B)$ and the above homomorphisms $KK^G_1(A, \Sigma B)\to E^G(\Sigma A, \Sigma B)\cong E^G(A, B)$, we have a natural transformation of the functors $KK^G(\cdot, B)$ to $E^G(\cdot, B)$ which sends $\id_B$ to $\id_B$. Each homomorphism $KK^G(A,  B)\to E^G(A, B)$  must coincide with the homomorphisms from $KK^G(A, B)$ to $E^G(A, B)$ of the canonical functor from $KK^G$ to $E^G$. To see this, since the homomorphisms $KK^G(A, B) \to E^G(A, B)$ commute with $\sigma_{\K(\hill_G)}$, we may assume $A= A'\K(\hill_G)$ and $B=B'\K(\hill_G)$ for some separable $G$-$\Calgs$ $A',B'$. According to \cite{Mey}, there is a $G$-$\Calgs$ $qA$ and $qB$ such that $qA$ (resp. $qB$) is isomorphic to $A$ (resp. $B$) in $KK^G$ via a canonical $\ast$-homomorphism $qA\to A$ (resp. $qB\to B$) and that any element in $KK^G(A, B)$ correspond to the element in $KK^G(qA, B)$ defined by a $\ast$-homomorphism (up to tensoring $\K$) via the $\ast$-homomorphism $qA\to A$. Since our two homomorphisms $KK^G(A, B)\to E^G(A, B)$ coincide on elements defined by $\ast$-homomorphisms and natural in first variable, they must coincide on whole elements. It follows that the homomorphisms $KK^G_1(A, B)$ to $E^G(\Sigma A, B)$ defined above are in fact, natural in both variables.

The following important fact is proven by Kasparov and Skandalis in \cite{bolic}.

\begin{prop}\label{KasSka} (cf. \cite{bolic} PROPOSITION A.3.) Let $A$ be a separable $G$-$\Calg$ such that $KK^G(A, \cdot)$ is half-exact (for example, when $A$ is proper and nuclear). Then, the canonical homomorphism from $KK^G(A, B)\to E^G(A, B)$ is an isomorphism for any separable $G$-$\Calg$ $B$.
\end{prop}
\begin{proof} Showing the surjectivity is easier.  Since the homomorphisms $KK^G(A, B) \to E^G(A, B)$ commute with $\sigma_\Sigma$ and $\sigma_{\K(\hill_G)}$, we may assume $A=\Sigma A'\K(\hill_G)$ and $B=\Sigma B'\K(\hill_G)$ for some separable $G$-$\Calgs$ $A',B'$. By our assumption, for any $x$ in $E^G(A, B)$, there is a separable $G$-$\Calg$ $C$ and a $\ast$-homomorphism $\phi$ from $A$ to $C$ and another $\ast$-homomorphism $\phi'$ from $B$ to $C$ which defines an invertible map $\phi'_\ast\colon E^G(A,B)\to E^G(A,C)$ such that $\phi'^{-1}_\ast(\phi)=x$ in $E^G(A, B)$: here, $C$ is a mapping cone of some $G$-extension. Half-exactness of $KK^G(A, \cdot)$ implies that $\phi'$ induce the isomorphism $\phi'_\ast\colon KK^G(A, B) \to KK^G(A, C)$ too. Since our homomorphisms $KK^G(A, B)\to E^G(A,B)$ are natural and send an element defined by a $\ast$-homomorphism to the element defined by the same $\ast$-homomorphism, we see that they are surjective. We now turn to the injectivity. Thanks to the universal property of $E^G$, the assumption that the functor $KK^G(A, \cdot)$ is half-exact implies there is a natural transformation from $E^G(A, \cdot)$ to $KK^G(A, \cdot)$ sending $\id_A$ to $\id_A$. Composing this with the homomorphisms $KK^G(A, B)$ to $E^G(A, B)$, we have a natural transformation from $KK^G(A, \cdot)$ to $KK^G(A, \cdot)$ which send $\id_A$ to $\id_A$. We show this natural transformation gives an isomorphism (actually the identity) $KK^G(A, B) \to KK^G(A, B)$ for any $B$. The homomorphisms $KK^G(A, B)\to KK^G(A,B)$ are natural in both  variables (as long as we consider $A$ such that $KK^G(A, \cdot)$ is half-exact) and send $\id_A$ to $\id_A$. Thus, it is the identity on elements defined by $\ast$-homomorphisms. As above, according to \cite{Mey}, any element in $KK^G(A, B)$ is a composition of $\ast$-homomorphisms and inverses of $\ast$-homomorphisms in the category $KK^G$. Thus, the homomorphisms $KK^G(A, B)\to KK^G(A, B)$ are identities on whole morphisms. We conclude the canonical homomorphism $KK^G(A, B)$ to $E^G(A, B)$ is an isomorphism for all $B$.
\end{proof}

The following is an immediate corollary of this:
\begin{cor}\label{important} Let $A$ be a separable $G$-$\Calg$ such that $KK^G(A, \cdot)$ is half-exact (for example, when $A$ is proper and nuclear). Then, the canonical homomorphism from $KK^G_1(A, B)\to E^G(\Sigma A, B)$ is an isomorphism for any separable $G$-$\Calg$ $B$.
\end{cor}

\newpage
\section{The Baum-Connes Conjecture and \\ \,the Higson-Kasparov Theorem}

Let $G$ be a second countable, locally compact group. The Baum-Connes conjecture proposes the formula for calculating $K$-theory of the reduced group algebra $\C_{\text{red}}(G)$, which is highly analytic object (it is a $\C$-completion of the convolution algebra $C_c(G)$ or $L^1(G)$), in terms of $G$-equivariant $K$-homology (with $G$-compact supports) of a universal proper $G$-space $\underline{E}G$, which is certainly more geometric in nature.  Following \cite{Val}, we will first quickly introduce the most current form of the conjecture using the equivariant $KK$-theory. After that, we will introduce the Higson-Kasparov Theorem which is one of the most general results concerning the Baum-Connes Conjecture and discuss some of the technical issues one must overcome when proving this theorem which will be explored in detail in later chapters.

\begin{dfn} A Hausdorff, paracompact topological space $X$ with a continuous $G$-action is a proper $G$-space if it is covered by $G$-invariant open subsets $U$ such that there exists a compact subgroup $H$ of $G$ and a $G$-equivariant map from $U$ to $G/H$. A proper $G$-space $X$ is universal if for any proper $G$-space $Y$, there exists a $G$-equivariant continuous map from $Y$ to $X$ unique up to $G$-homotopy.
\end{dfn}

It it known that properness of a locally compact $G$-space coincides with the usual notion of properness of $G$-actions. A universal proper $G$-space exists; and it is unique up to $G$-homotopy; we denote it by $\underline{E}G$. See \cite{BCH} and also \cite{CEM} for a detailed exposition of these notions. A proper $G$-space is called $G$-compact if it is covered by translates of a compact subset $K$ over $G$. A $G$-compact proper $G$-space is locally compact; and its quotient by $G$ is compact. Given a universal proper $G$-space $\underline{E}G$, $G$-invariant $G$-compact proper subsets of $\underline{E}G$ form an inductive system under (proper) inclusion. Hence, we obtain an inductive system of $G$-equivariant $K$-homology groups.

\begin{dfn} A $K$-homology group $RK^G_{\ast}(\underline{E}G)$ of a universal proper $G$-space $\underline{E}G$ with $G$-compact supports is defined by:
\begin{align*}
RK^G_{\ast}(\underline{E}G) = \lim_{\substack{X\subseteq \underline{E}G \\ \text{$X$:$G$-inv.\,$G$-cp.}}}K_\ast^G(X) = \lim_{\substack{X\subseteq \underline{E}G \\ \text{$X$:$G$-inv.\,$G$-cp.}}}KK_\ast^G(C_0(X), \bbC)
\end{align*}
\end{dfn}

G. Kasparov defined in \cite{Kas2} a descent homomorphism for separable $G$-$\Calgs$:
\[
\xymatrix{
KK_\ast^G(A,B) \ar[r]^-{j_G} & KK_\ast(\C_{\text{max}}(G,A), \C_{\text{max}}(G,B));
}
\]
and its reduced version:
\[
\xymatrix{
KK_\ast^G(A,B) \ar[r]^-{j_{G,\text{red}}} & KK_\ast(\C_{\text{red}}(G,A), \C_{\text{red}}(G,B)). 
}
\]
On the other hand, for any proper $G$-compact space $X$, there is a distinguished class $[\mathcal{L}_X]$ in the $K$-theory group of the reduced group $\Calg$ $K_\ast(\C_{\text{red}}(G,X))=KK_\ast(\bbC,\C_{\text{red}}(G,X))$ (see \cite{Val}). The assembly map $\mu_{G,\text{red}}^X$ for a proper $G$-compact space $X$ is defined by a descent homomorphism followed by the Kasparov product with $[\mathcal{L}_X]$:
\[
\xymatrix{
\mu_{G,\text{red}}^X\colon KK_\ast^G(C_0(X),\bbC) \ar[r]^-{j_{G,\text{red}}} & KK_\ast(\C_{\text{red}}(G,X),\C_{\text{red}}(G)) \ar[r]^-{[\mathcal{L}_X]\times} & KK_\ast(\bbC,\C_{\text{red}}(G)).
}
\]
By fixing a universal proper $G$-space $\underline{E}G$, the Baum-Connes assembly map $\mu_{G,\text{red}}$ is defined as the inductive limit of above defined assembly maps $\mu_{G,\text{red}}^X$ for $G$-invariant, $G$-compact subsets $X$ of $\underline{E}G$:
\[
\xymatrix{
\mu_{G,\text{red}}\colon RK^G_{\ast}(\underline{E}G) = \lim_{\substack{X\subseteq \underline{E}G \\ \text{$X$:$G$-inv.\,$G$-cp.}}}KK_\ast^G(C_0(X),\bbC) \ar[r]^-{\mu_{G,\text{red}}^X} & KK_\ast(\bbC,\C_{\text{red}}(G)) = K_\ast(\C_{\text{red}}(G)).
}
\]
Shintaro  Nishikawa
\begin{conj}\label{thm:BC} (Baum-Connes conjecture) The assembly map $\mu_{G,\text{red}}$ is always an isomorphism.
\end{conj}
For general groups $G$, Conjecture \ref{thm:BC} is still open. Nonetheless, it has been verified for quite large classes of groups. See \cite{Val} for a list of groups satisfying Conjecture \ref{thm:BC}.

For any separable $G$-$\Calg$ $A$, there is a more general assembly map $\mu_{G,\text{red}}^A$ with coefficient $A$ which can be defined by almost the same way as above:\begin{scriptsize}
\[
\xymatrix{
\mu_{G,\text{red}}^A\colon RKK^G_{\ast}(\underline{E}G, A)  =  \lim_{\substack{X\subseteq \underline{E}G \\ \text{$X$:$G$-inv.\,$G$-cp.}}}KK_\ast^G(C_0(X), A)  \ar[r] &KK_\ast(\bbC,\C_{\text{red}}(G,A)) = K_\ast(\C_{\text{red}}(G,A))
}
\] \end{scriptsize}
\begin{conj}\label{thm:BCC} (Baum-Connes Conjecture with coefficients) The assembly map $\mu_{G,\text{red}}^A$ is an isomorphism for any separable $G$-$\Calg$ $A$.
\end{conj}

Conjecture \ref{thm:BCC} has a virtue of being hereditary to closed subgroups. Although this conjecture is known to be false in general (See \cite{HigGuen}), it is believed that this conjecture should be valid for a reasonably large class of groups. Concerning Conjecture \ref{thm:BCC}, N. Higson and G. Kasparov proved in \cite{HigKas2} a quite general result which states that Conjecture \ref{thm:BCC} holds for any (second countable) groups satisfying a certain geometric condition.

Let $G$ be a locally compact group. An affine isometric action of $G$ on a real Hilbert space $\hill$ will be denoted by $(\pi, b)$. It means we have a continuous group homomorphism $\pi\colon G\to O(\hill)$ (the infinite orthogonal group $O(\hill)$ of $\hill$ is equipped with the strong operator topology) and a (norm) continuous map $b\colon G\to \hill$ satisfying the cocycle condition $b(gg')=\pi(g)b(g')+b(g)$ for any $g,g'$ in $G$. It is called metrically proper if $\displaystyle \lim_{g\to \infty}\|b(g)\|=\infty$.  

\begin{dfn}\label{dfn:a-t} A second countable, locally compact group is called {a-${T}$-menable} if it admits a metrically proper, affine isometric action on a Hilbert space.
\end{dfn}

A-$T$-menable groups are also called as groups with the Haagerup property. The class of a-$T$-menable groups contains all second countable amenable groups; and an a-$T$-menable group $G$ has Kazhdan's property $(T)$ if and only if $G$ is compact. The prefix a-$T$ means ``not'' having property $(T)$.

\begin{theorem}\label{thm:BCCa-T} (cf. \cite{HigKas2} Theorem 9.1.) The Baum-Connes conjecture with coefficients holds for all a-$T$-menable groups.
\end{theorem}

In the rest of this chapter, we roughly summarize the proof of Theorem \ref{thm:BCCa-T} given by N. Higson and G. Kasparov in \cite{HigKas2} and discuss some of the technical issues surrounding the proof. Our reference includes \cite{HigKas1} \cite{HigGuen} and \cite{Julg}. 

Recall for a second countable, locally compact group $G$, $KK^G$ denotes the additive category of separable $G$-$\Calgs$ whose morphism groups are the Kasparov groups $KK^G(A,B)$. A standard approach to Conjecture \ref{thm:BCC} is so-called the Dual-Dirac method which may be summarized in the following way. See \cite{Kas2} \cite{Tu} and \cite{MeNe}.

\begin{theorem} (cf. \cite{MeNe} Theorem 8.3.) Let $G$ be a second countable, locally compact group. Suppose one finds a separable proper $G$-$\Calg$ $A$, an element $D$ (a Dirac morphism) in $KK^G(A, \bbC)$ and an element $\beta$ (a dual Dirac morphism) in $KK^G(\bbC, A)$ such that $\beta\circ D=1_A$. Then, $\gamma=D\circ\beta$ (a gamma element for $G$) is an idempotent in a ring $KK^G(\bbC,\bbC)$; and the Baum-Connes assembly map $\mu_{G,\text{red}}^A$ is split-injective for any separable $G$-$\Calg$ $A$. Moreover, a gamma element $\gamma$ is unique if it exists. If $\gamma=1 \in KK^G(\bbC, \bbC)$, then the assembly map $\mu_{G,\text{red}}^A$ is also surjective for any $A$: in other words, the Baum-Connes conjecture with coefficients holds for $G$.  
\end{theorem}

There is a way of seeing a Dirac morphism as an analogue of simplicial approximation of topology; and in that sense it is known a Dirac morphism (which can be defined in a suitable way) always exists, see \cite{MeNe}. What we actually use is the following.

\begin{theorem}\label{DD} (cf. \cite{MeNe} Theorem 8.3., \cite{Tu} Theorem 2.2.) Let $G$ be a second countable, locally compact group. Suppose the identity $1_\bbC \in  KK^G(\bbC, \bbC)$ factors through a separable proper $G$-$\Calg$ $A$. Then, a gamma element $\gamma$ for $G$ exists and $\gamma=1$. Hence, the Baum-Connes conjecture with coefficients holds for $G$.
\end{theorem}

We fix a second countable, locally compact group $G$ which acts properly, and affine isometrically on a fixed separable Hilbert space $\hill$. In view of Theorem \ref{DD}, in order to prove Theorem \ref{thm:BCCa-T}, we need to find a natural candidate $A$ through which the identity $1_\bbC$ factors through. The candidate $A$ is the $\Calg$ $A(\hill)$ which is now being called the $\Calg$ of the Hilbert  space $\hill$. 

\begin{dfn} ($\Calg$ of Hilbert space) (cf. \cite{HigKas2} Section. 4.) We define a graded $\Calg$ $A(\hill)$ of the Hilbert space $\hill$ by the following way. For each finite dimensional affine subspace $V$ of $\hill$, denote by $V_0$ the linear part of $V$ which is naturally regarded as a linear subspace of $\hill$. The complexified exterior algebra  $\Lambda^\ast(V_0)\otimes\bbC$ is naturally regarded as a graded Hilbert space (see Example \ref{Clif}). We simply denote the graded $\Calg$ $B(\Lambda^\ast(V_0)\otimes\bbC)$ by $L(V)$. The graded $\Calg$ $\cC(V)$ is defined by $\cC(V)=C_0(V\times V_0, L(V))$. The graded $\Calg$ $A(V)$ is a graded tensor product $\cS\hat\otimes\cC(V)$: recall $\cS$ is the $\Calg$ $C_0(\bbR)$ graded by reflection at the origin. For finite dimensional affine subspaces  $V \subseteq V'=V\oplus W$ ($W$ is defined as $V'_0\ominus V_0$ which is a finite dimensional linear subspace of $\hill$), we have an isomorphism of graded $\Calgs$  $\cC(V^\prime)\cong\cC(V)\hat\otimes\cC(W)$. We define an inclusion $A(V)\xhookrightarrow{}A(V^\prime)$ by tensoring the inclusion $\cS\xhookrightarrow{}A(W)$ which we will define soon below, with the identity on $\cC(V)$. For a finite dimensional linear subspace $W$ of $\hill$, the Bott operator $\cB_W$ for $W$ is an odd unbounded multiplier on $\cC(W)$ defined as $(w_1,w_2)\mapsto i\overline{c}(w_1)+c(w_2)$ on the subspace of compactly supported functions ($c(w)$ and $\overline{c}(w)$ are Clifford multiplication operators defined in Example \ref{Clif}). We recall here that associated to an (odd) multiplier $T$ on $A$, there is a unique (graded) functional calculus homomorphism from $\cS$ to the multiplier algebra $M(A)$ sending resolvent functions $(x\pm i)^{-1}$ to the resolvents $(T\pm i)^{-1}$. We use an odd multiplier $X\hat\otimes1+1\hat\otimes\cB_W$ on $A(W)=\cS\hat\otimes\cC(W)$ to define the inclusion $\cS\xhookrightarrow{}A(W)$ for any finite dimensional linear subspace $W$ of $\hill$. The $\Calg$ $A(\hill)$ of the Hilbert space $\hill$ is defined as an inductive limit of the $\Calg$ $A(V)$ for all finite dimensional affine subspace $V$ of $\hill$ using the inclusions we defined above.
\end{dfn}

The $\Calg$ $A(\hill)$ naturally becomes a $G$-$\Calg$. We note here that for any increasing sequence of affine subspaces $V_n$ whose union is dense in $\hill$, an inductive limit of $A(V_n)$ is canonically isomorphic to $A(\hill)$. The next proposition says that the $\Calg$ $A(\hill)$ is a proper $G$-$\Calg$.

\begin{prop} (cf. \cite{HigKas2} Theorem 4.9.) The $\Calg$ $A(\hill)$ of the Hilbert space $\hill$ is a proper $G$-$\Calg$. 
\end{prop}
\begin{proof} The center $Z(\hill)$ of $A(\hill)$ is an inductive limit of the center $Z(V)\cong C_0([0,\infty)\times V\times V_0)$ of $A(V)$. The inclusion $A(V)\xhookrightarrow {}A(V^\prime)$ descends to an inclusion $Z(V)\xhookrightarrow {}Z(V^\prime)$; and this corresponds to `projections' \[[0,\infty)\times V^\prime\times V^\prime_0 \ni (t, v^\prime_1, v^\prime_2) \mapsto (\sqrt{\mathstrut t^2+\|w_1\|^2+\|w_2\|^2}, v_1, v_2) \in [0,\infty)\times V\times V_0\] where $v^\prime_i=v_i+w_i$ in the decomposition $V^\prime=V\oplus V^\perp$ or $V^\prime_0=V_0\oplus V^\perp_0$. Therefore, the Gelfand spectrum of $Z(\cH)$ is identified with the second countable, locally compact Hausdorff space $[0,\infty)\times \cH\times\cH$ whose topology is the weak topology defined by inclusion $[0,\infty)\times \cH\times\cH \ni (t,v_1,v_2) \mapsto (\sqrt{\mathstrut t^2+\|v_1\|^2+\|v_2\|^2}, v_1, v_2) \in [0,\infty)\times \cH\times\cH$ where two $\hill$ in the right-hand side are endowed with the weak topology of the Hilbert space. The $G$-action on $A(\cH)$ corresponds to a $G$-action on $[0,\infty)\times \cH\times\cH$ which is identity on the first factor, the affine action of $G$ on $\hill$ on the second and the linear part of the affine action of $G$ on the third. The properness of this $G$-action is easily verified. One can also check $Z(\cH)A(\cH)$ is dense in $A(\cH)$. This shows $A(\hill)$ is a proper $G$-$\Calg$.
\end{proof}

The Bott operator $\cB_W$ for each finite dimensional linear subspace $W$ of $\hill$ can be assembled together to define a single odd unbounded multiplier $\cB$ on $A(\hill)$ which we call the Bott operator for $\hill$. Using the functional calculus for $\cB$, we obtain an element $F=\cB(1+\cB^2)^{-\frac{1}2}$ in $M(A(\hill))$. This element is selfadjoint and essentially unitary and essentially equivariant (meaning $F^2-I, \,g(F)-F \in A(\hill) \,\,\text{for $g \in G$}$); thus $\frac{F+1}{2}$ is an essentially projection which is essentially equivariant. It defines an element $b$ in $KK^G_1(\bbC, A(H))$.

\begin{dfn} In the same notation as above, we call the element $b=(A(\hill), 1, \frac{F+1}{2})$ in $KK^G_1(\bbC, A(\hill))$ as the Bott element or the dual Dirac element.
\end{dfn}


To find the Dirac element which inverts the Bott element $b$, we need to find a certain ``Dirac operator'' which defines an extension of $A(\hill)$ by ``compact operators'' because a natural Dirac element should lie in the boundary of such an extension (just like Toeplitz extension inverts the classical Bott element). The approach given in \cite{HigKas2} is slightly different from this. They constructed a certain $G$-continuous field $(A_\alpha(\hill))_{\alpha\in[0,\infty)}$ of $G$-$\Calgs$ over the interval $[0,\infty)$ with $A_0(\hill)=A(\hill), A_{\alpha}(\hill)= \cS\hat\otimes\K(H_\alpha(\cH))$ for $\alpha$ in $(0,\infty)$ where $(H_\alpha(\cH))_{\alpha\in(0,\infty)}$ is a certain continuous field of $G$-Hilbert spaces (the details are given in the following chapter). The $G$-$\Calg$ $\mathcal{F}$ of continuous sections of the field $(A_\alpha(\hill))_{\alpha\in[0,\infty)}$ which vanish at infinity (by evaluating at $0$) would give us an extension of $G$-$\Calgs$:
\begin{equation}\label{ext}
\xymatrix{
0 \ar[r] & \K(\cS\hat\otimes\E) \ar[r] & \mathcal{F} \ar[r] & A(\hill) \ar[r] & 0
}
\end{equation}
where $\E$ is a continuous sections of the field $(H_\alpha(\cH))_{\alpha\in(0,\infty)}$ of $G$-Hilbert spaces which vanish at infinity and where $\cS\hat\otimes\E$ is regarded as a $G$-$S\Sigma$-Hilbert module (we identify $\Sigma=C_0(0, 1)$ as $C_0(0, \infty)$ as far as it makes no confusion). We note that the extension \eqref{ext} is isomorphic to
\begin{equation*}
\xymatrix{
0 \ar[r] & S\K\Sigma \ar[r] & \mathcal{F} \ar[r] & A(\hill) \ar[r] & 0
}
\end{equation*}
if we disregard the $G$-actions.
  
Unfortunately, we would not be able to directly associate an element in $KK^G_1(A(\hill),S\Sigma)\\\cong KK^G_1(A(\hill),\bbC)$ to the extension \eqref{ext} since it is not clear that the extension \eqref{ext} admits an $G$-equivariant completely positive section. Nonetheless, there is an element in $KK^G_1(A(\hill),S\Sigma)$ which serves as an approximation of an ``element'' associated to the extension \eqref{ext}. The precise meaning of this approximation is as follows. For separable $G$-$\Calgs$ $A$ and $B$, one can define an abelian semigroup $\{A, B\}_G$ of homotopy equivalence classes of equivariant asymptotic morphisms from $A$ to $\K(\E)$ where $\E$ is a countably generated $G$-$B$-Hilbert module. The semigroup $\{\Sigma A, B\}_G$ is an abelian group; and associated to any extension of separable $G$-$\Calgs$:
\[
\xymatrix{
0 \ar[r] & \K(\E) \ar[r] & \mathcal{F} \ar[r] & A \ar[r] & 0
}
\]
there is a uniquely determined class of the group $\{\Sigma A, B\}_G$. Now, we denote by $\alpha$ the class in $\{\Sigma A(\hill), S\Sigma\}_G$ uniquely associated to the extension \eqref{ext}.

There are naturally defined group homomorphisms $\eta$ from $KK^G_1(A,B)$ to $\{\Sigma A, B\}_G$ and from $KK^G_0(A,B)$ to $\{\Sigma^2 A, B\}_G$ (\cite{HigKas2} Definition 7.4.). In the paper \cite{HigKas2}, N. Higson and G. Kasparov constructed a canonical element $d$ in $KK^G_1(A(\hill),S\Sigma)\cong KK^G_1(A(\hill),\bbC)$ such that $\eta(d)=\alpha \in \{\Sigma A(\hill), S\Sigma\}_G$. We call the element $d$ as the Dirac element. The conclusion is as follows.

\begin{theorem}\label{conc} (cf. \cite{HigKas2} Theorem 8.5.) The Dirac element $d \in KK^G_1(A(\hill), \bbC)$ is a (right) inverse of the dual Dirac element $b \in KK^G_1(\bbC, A(\hill))$.  In other words, we have $b\otimes_{A(\hill)}d=1_\bbC$. 
\end{theorem}

As is implied in the construction of the Dirac element $d$, the proof of Theorem \ref{conc} takes a somewhat indirect approach. It is based on an $E$-theoretic argument. One first calculates the composition of asymptotic morphisms $\eta(d)$ and $\eta(b)$, and next translates this calculation to one for the Kasparov bivariant theory $KK^G$. 
This ends our brief summary of the Higson-Kasparov Theorem.  In the following chapters, we are going to give details of the proof of the Higson-Kasparov Theorem.

\newpage
\section{$E$-theoretic Part of the Higson-Kasparov Theorem}

In this chapter, we will give the detail of $E$-theoretic theoretic part of Higson-Kasparov Theorem. We will consider a second countable group $G$ which acts affine isometrically on a separable infinite dimensional (real) Hilbert space $\hill$. The goal of this chapter is to define the canonical $G$-extension \eqref{ext} of the $\Calg$ $A(\hill)$ and to compute the composition (of asymptotic morphisms) of the Bott element and the central invariant associated to this extension.

\begin{dfn} (cf. \cite{HigKas2} Definition 2.6.) For any positive real number $\alpha>0$, we will later define the canonical graded (complex!) Hilbert space $H_\alpha(\hill)$ associated to the (real) Hilbert space $\hill$ and the canonical unbounded operator on $H_\alpha(\hill)$. Fix $\alpha>0$. We first define for any finite dimensional affine subspace $V$ of $\hill$, the graded (complex Hilbert space) $H(V)=L^2(V, \Lambda^\ast(V_0)\otimes\bbC)$. Here, we use the usual Lebesgue measure on $V$. For any finite dimensional linear subspace $W$ of $\hill$, the Bott-Dirac operator on $W$ (for a fixed $\alpha$) is an odd symmetric unbounded operator\begin{equation}\label{BottDiracW}
 B_{W,\alpha}=\sum^m_{j=1}\alpha \overline{c}(w_j)\frac{\partial}{\partial x_j}+c(w_j)x_j \end{equation} defined on the subspace $s(W)$ of Schwarts functions of $H(W)$ where $m=\dim W$, $x_j$ are coordinate functions for some fixed orthonormal system for $W$ and $w_j$ are its dual basis. One can check the Bott-Dirac operator $B_{W,\alpha}$ is defined independently of choices of an orthonormal system (basis) for $W$. When $W=W_1\oplus W_2$, $H(W)=H(W_1)\hat\otimes H(W_2)$ naturally; and we have $B_{W_\alpha}=B_{W_1,\alpha}\hat\otimes1 + 1\hat\otimes B_{W_2,\alpha}$.\end{dfn}

\begin{prop} (cf. \cite{HigKas2} Definition 2.6.) The Bott Dirac operator $B_{W,\alpha}$ is an essentially selfadjoint odd unbounded operator having compact resolvent with one-dimensional kernel.
\end{prop}
\begin{proof} If $W$ is one-dimensional, the Bott Dirac operator $B_{W,\alpha}$ on $W$ is nothing but the one which we described in Example \ref{BottDirac}; and we know it is an essentially selfadjoint odd unbounded operator having compact resolvent with one-dimensional kernel. In general case, decompose $W$ into one-dimensional subspace. Then, $B_{W, \alpha}^2$ may be written as $B_{W_1,\alpha}^2\hat\otimes1\hat\otimes\cdots\hat\otimes1 + 1\hat\otimes B_{W_2,\alpha}^2\hat\otimes\cdots\hat\otimes1  + \cdots + 1\hat\otimes1\hat\otimes\cdots\hat\otimes B_{W_m,\alpha}^2$ with $m=\dim(W)$ and $W_i$ are mutually orthogonal one-dimensional subspace of $W$ for $i=1,2,\dots,m$. It is now clear that $B_{W,\alpha}^2$ is an essentially selfadjoint operator having compact resolvent with one dimensional kernel; and so is $B_{W,\alpha}$ by Lemma \ref{lemselfad}. We note here that the kernel of $B_{W,\alpha}$ (hence of $B_{W,\alpha}^2$) is spanned by a normalized vector $\xi_{W,\alpha}(x)=(\alpha\pi)^{-\frac{m}{4}}\exp(-\frac{||x||^2}{2\alpha}).$
\end{proof}

\begin{dfn}(The Hilbert space $H_\alpha(\hill)$ and the Bott-Dirac operator $B_\alpha$ on $\hill$) (cf. \cite{HigKas2} Definition 2.8.) We still implicitly fix $\alpha>0$. The graded Hilbert space $H_\alpha(\hill)$ is defined as an inductive limit of graded Hilbert spaces $H(V)$ where $V$ runs through all finite dimensional affine subspaces of $\hill$: given finite dimensional subspaces $V$ and $ V'=V\oplus W$, we define an inclusion $H(V) \xhookrightarrow{} H(V')=H(V)\otimes H(W)$ by $\xi\mapsto\xi\otimes\xi_{W,\alpha}$. When a group $G$ acts on $\hill$ by affine isometries, it naturally acts on $H_\alpha(\hill)$.  For finite dimensional linear subspaces $W\subseteq W'$ of $\hill$, we have the following commutative diagram.
\begin{align}\label{diag:bottdirac}
\xymatrix{
s(W) \ar[d]^-{B_{W,\alpha}} \ar[r] & s(W') \ar[d]^-{B_{W',\alpha}} \\
s(W) \ar[r] & s(W')\\
}
\end{align}
The Bott-Dirac operator on $\hill$ is an odd symmetric unbounded operator $B_\alpha$ defined on a subspace $\displaystyle s_\alpha(\hill)=\lim_{\substack{W\subseteq\hill \\ \text{$W$:f.n.dim.\,linear}}} s(W)$ of $\displaystyle H_\alpha(\hill)\cong\lim_{\substack{W\subseteq\hill \\ \text{$W$:f.n.dim.\,linear}}} H(W)$ which is defined as an inductive limit of the Bott-Dirac operator $B_{W,\alpha}$. We note here that we have a continuous field $(H_\alpha(\hill))_{\alpha\in(0, \infty)}$ of (graded) Hilbert spaces over the interval $(0, \infty)$: basic sections are defined by vectors in $H(V)$ for any finite dimensional affine subspace $V$ of $\hill$.  When $G$ acts on $\hill$ by affine isometries, this becomes a continuous field of (graded) $G$-Hilbert spaces naturally. 
\end{dfn}

\begin{prop} (cf. \cite{HigKas2} Definition 2.8.) The Bott-Dirac operator $B_{\alpha}$ is an essentially selfadjoint odd unbounded operator having one-dimensional kernel. When a group $G$ acts on $\hill$ by linear isometries, it is $G$-equivariant.
\end{prop}
\begin{proof} The first part is similar to the finite dimensional case. Taking any decomposition of $\hill=\displaystyle\bigoplus_{i=1}^\infty W_i$ with finite dimensional subspaces $W_i$ for $i=1,2,\dots$, we may write $B_\alpha$ as an infinite sum $B_{W_1, \alpha}\hat\otimes1\hat\otimes1\hat\otimes\cdots+1\hat\otimes B_{W_2, \alpha}\hat\otimes1\hat\otimes\cdots+\cdots$. Then, we have $B_\alpha^2=B_{W_1, \alpha}^2\hat\otimes1\hat\otimes1\hat\otimes\cdots+1\hat\otimes B_{W_2, \alpha}^2\hat\otimes1\hat\otimes\cdots+\cdots$. It is clear that $B_\alpha^2$ is an essentially selfadjoint (diagonalizable) operator having one-dimensional kernel, and so is $B_\alpha$ by Lemma \ref{lemselfad}. That it is $G$-equivariant follows from that for finite dimensional subspace $W$, $B_{W, \alpha}$ is well-defined by the expression \eqref{BottDiracW} independently of choices of a basis for $W$ and that the diagram \eqref{diag:bottdirac} commutes for any $W$ and $W'$.
\end{proof}

Unfortunately, the Bott-Dirac operator $B_\alpha$ does not have compact resolvent. In \cite{HigKas2}, N. Higson and G. Kasparov introduced a non-commutative functional calculus for the operator $B_\alpha$ in order to perturb $B_\alpha$ to make it having compact resolvent in a very tractable way. \\

For any (not necessarily bounded, but densely defined) operator $h$ on (real) Hilbert space $\hill$, we define an (unbounded) operator $h(B_\alpha)$ defined on a subspace $\displaystyle s_\alpha(\hill_{h})=\lim_{\substack{W\subseteq\hill_h \\ \text{$W$:f.n.dim.\,linear}}} s(W)$ of $s_\alpha(\hill)$ where we denote the domain of $h$ by $\hill_h$.  For any finite dimensional linear subspaces $W$ of $\hill_h$ and $V\supseteq W+hW$, we denote by $s(W,V)$  the space of Schwarts functions from $W$ to $\Lambda^\ast(V)\otimes\bbC$ naturally regarded as a subspace of  $L^2(W,\Lambda^\ast(V)\otimes\bbC)\subseteq H_\alpha(\hill)$ (Note $s(W)=s(W,W)$); and we define an (unbounded) operator $h(B_{W,\alpha})$ from $s(W)$ to $s(W,V)\subseteq s(\hill)$ by the following formula: 
\begin{align}
h(B_{W,\alpha}) &= \sum^m_{j=1}\alpha \overline{c}(h(w_j))\frac{\partial}{\partial x_j}+c(h(w_j))x_j
\label{def:cal}
\end{align}
here, $m=\dim{W}$ and $x_j$ and $w_j$ are the same as before. This is again defined independently of a choice of a basis for $W$.\\
Now, we consider whether for any finite dimensional linear subspaces $W\subseteq W'$ and $V\supseteq W'+hW'$, the following diagram is commutative or not.
\begin{align}
\xymatrix{
s(W) \ar[d]^-{h(B_{W,\alpha})} \ar[r] & s(W') \ar[d]^-{h(B_{W',\alpha})}\\
s(W,V) \ar[r] & s(W',V)\\
}
\label{diagram:cal}
\end{align}
\begin{prop}\label{prop:cal1} Let $W''=W'\ominus W$. The diagram \eqref{diagram:cal} commutes if and only if $hW''$ is orthogonal to $W$. In particular, when $h$ is symmetric, the diagram \eqref{diagram:cal} commutes if and only if $hW$ is orthogonal to $W''$.
\end{prop}
\begin{proof} We fix orthonormal bases for $W$ and for $W''$ and denote the corresponding coordinate functions and dual bases by $x_j, w_j \, (j=1,\dots,m)$ and by $x''_k, w''_k \, (k=1,\dots,l)$.  We first note that vectors of the form $\xi\otimes (w_{j_1}\wedge\dots\wedge w_{j_s})$ spans $s(W)$ where $\xi$ is a (complex valued) Schwarts function on $W$. Therefore, the diagram \eqref{diagram:cal} is commutative if and only if 
\begin{align}
h(B_{W,\alpha})(\xi\otimes (w_{j_1}\wedge\dots\wedge w_{j_s}))\otimes\xi_{W'',\alpha}=h(B_{W',\alpha})(\xi\otimes\xi_{W'',\alpha}\otimes (w_{j_1}\wedge\dots\wedge w_{j_s}))
\label{cal1}
\end{align}
holds for any $\xi$ and $j_1,\dots,j_s$. By considering a natural decomposition $h(B_{W',\alpha})=h(B_{W,\alpha})+h(B_{W'',\alpha})$, we see the equation \eqref{cal1} holds if and only if
\begin{align}
h(B_{W'',\alpha})(\xi\otimes\xi_{W'',\alpha}\otimes (w_{j_1}\wedge\dots\wedge w_{j_s}))=0
\label{cal2}
\end{align}
By a further decomposition 
\begin{align*}
h(B_{W'',\alpha})=\sum^l_{k=1}\exte(h(w''_k))(\alpha\frac{\partial}{\partial x''_k}+x''_k)+\sum^l_{k=1}\inte(h(w''_k))(-\alpha\frac{\partial}{\partial x''_k}+x''_k)
\end{align*}
we see the equation \eqref{cal2} holds if and only if
\begin{align}
\sum^l_{k=1}\inte(h(w''_k))(-\alpha\frac{\partial}{\partial x''_k}+x''_k)(\xi\otimes\xi_{W'',\alpha}\otimes (w_{j_1}\wedge\dots\wedge w_{j_s}))=0
\label{cal3}
\end{align}
since $\xi_{W'',\alpha}=(\alpha\pi)^{-\frac{l}{4}}\exp(-\frac{||x''||^2}{2\alpha})$ is in the kernels of differential operators $\alpha\frac{\partial}{\partial x''_k}+x''_k$. In sum, by further calculating the equation \eqref{cal3}, the diagram \eqref{diagram:cal} is commutative if and only if for any Schwarts function $\xi$ on $W$ and $j_1,\dots,j_s \in \{1,\dots,m\}$ the following equation holds.
\begin{align}
\sum^l_{k=1}\xi\otimes2x''_k\xi_{W'',\alpha}\otimes\inte(h(w''_k))(w_{j_1}\wedge\dots\wedge w_{j_s})=0
\label{cal4}
\end{align}
Now, note that $2x''_k\xi_{W'',\alpha}$ are mutually orthogonal vectors in $L^2(W)$, hence we conclude the diagram \eqref{diagram:cal} is commutative if and only if
\begin{align*}
&\inte(h(w''_k))(w_{j_1}\wedge\dots\wedge w_{j_s})=0 \,\,\, \text{for any $k$ and $j_1,\dots,j_s$}\\
\iff& \text{$hW''$ is orthogonal to W} 
\end{align*} 
\end{proof}

For any finite dimensional linear subspaces $W\subseteq W'\subseteq W''$ and $V\supseteq W''+hW''$, let us now consider the following slightly different diagram from \eqref{diagram:cal}:
\begin{align}
\xymatrix{
s(W) \ar[r] & s(W') \ar[d]^-{h(B_{W',\alpha})} \ar[r] & s(W'') \ar[d]^-{h(B_{W'',\alpha})}\\
& s(W',V) \ar[r] & s(W'',V)\\
}
\label{diagram:cal2}
\end{align}
Let us say that the diagram \eqref{diagram:cal2} eventually commutes if there exists a finite dimensional subspace $W'\supseteq W$ of $\hill_h$ such that for any finite dimensional subspace $W''\supseteq W'$ of $\hill_h$, the diagram \eqref{diagram:cal2} commutes. The following is an immediate corollary.
\begin{cor}\color{black} Let $W'^\perp=\hill_h\ominus W'$. The diagram \eqref{diagram:cal2} eventually commutes if and only if there exists a finite dimensional subspace $W'$ of $\hill_h$ such that $hW'^\perp$ is orthogonal to $W$. In particular, when $h$ is symmetric, the diagram \eqref{diagram:cal} eventually commutes if and only if $hW\subseteq \hill_h$.
\end{cor}
\color{black}
The above corollary says that when trying to define an inductive limit of $h(B_{W,\alpha})$, one needs to be careful more than merely observing whether the diagram \eqref{diagram:cal2} eventually commutes. This is the point which is not mentioned in the paper \cite{HigKas2}. Fortunately, as the next proposition and its corollary says, even though for any finite dimensional subspace $W$ of $\hill_h$, the diagram \eqref{diagram:cal2} may not eventually commute, if $h$ has its adjoint defined on $\hill_h$, it always asymptotically commutes in the following sense: we say that the diagram\eqref{diagram:cal2} asymptotically commutes if for any vector in $s(W)$ and for any $\epsilon>0$, there exists a finite dimensional subspace $W'\supseteq W$ of $\hill_h$ such that for any finite dimensional subspace $W''\supseteq W'$ of $\hill_h$, the difference between two vectors gained by two ways in the diagram \eqref{diagram:cal2} is within $\epsilon$ in the norm of $s_\alpha(\hill)\subseteq H_\alpha(\hill)$.

\begin{prop}\color{black} The diagram \eqref{diagram:cal2} asymptotically commutes if and only if $h$ has its adjoint defined on $W$.
\end{prop} 
\color{black}
\begin{proof} Let us still denote a fixed coordinates and basis $w_1,\dots,w_m$ for $W$. We do a similar calculation as in Proposition \ref{prop:cal1}. We see that the diagram \eqref{diagram:cal2} asymptotically commutes if and only if for any $j_1,\dots, j_s$ and for any $\epsilon>0$, there exists a finite dimensional subspace $W'\supseteq W$ of $\hill_h$, such that for any finite dimensional subspace $W''\supseteq W'$ of $\hill_h$, 
\begin{align*}
\sum^l_{k=1}||\inte(h(w''_k))(w_{j_1}\wedge\dots\wedge w_{j_s})||^2<\epsilon
\end{align*}
where $w''_1,\dots,w''_l$ is some (arbitrary) basis for $W''\ominus W'$ and the norm is computed in $L^2(\Lambda^\ast(W)\otimes\bbC)$. Considering each one-dimensional subspace of $L^2(\Lambda^\ast(W)\otimes\bbC)$ we see that the diagram \eqref{diagram:cal2} asymptotically commutes if and only if for any $\epsilon>0$ there exists $W'\supseteq W$ such that for any $W''\supseteq W'$,
\begin{align*}
\sum^l_{k=1}|\langle w_j, hw''_k \rangle|^2<\epsilon
\end{align*}
for any $j=1,\dots,m$. It is now clear this is equivalent to that the adjoint of $h$ is defined on $W$.
\end{proof}
\begin{cor}\label{cor:cal}\color{black} The diagram \eqref{diagram:cal2} asymptotically commutes for any finite dimensional subspace $W$ of $\hill_h$ if and only if $h$ has its adjoint defined on $\hill_h$. In particular, when $h$ is symmetric, the diagram \eqref{diagram:cal} always asymptotically commutes.
\end{cor}
\color{black}
\begin{dfn}\label{dfn:cal}\color{black}(a fixed non-commutative functional calculus) (cf. \cite{HigKas2} Definition 3.5.) Let $h$ be a densely defined operator on $\hill$ whose adjoint is defined on the domain $\hill_h$ of $h$, we define a densely defined operator $h(B_\alpha)$ on $H_\alpha(\hill)$ defined on its subspace $\displaystyle s_\alpha(\hill_h)=\lim_{\substack{W\subseteq\hill_h \\ \text{$W$:f.n.dim.\,linear}}} s(W)$ by the following:
\begin{align*}
h(B_\alpha)(\xi):= \displaystyle\lim_{\substack{W\subseteq W' \subseteq \hill_h \\ \text{$W'$: f.n. dim. \,linear}}} h(B_{W',\alpha})(\xi\otimes\xi_{W'\ominus W,\alpha})
\end{align*}
for a finite dimensional subspace $W$ of $\hill_h$, $\xi$ in $s(W)$. The limit is taken in the Hilbert space $H_\alpha(\hill)$. This is well-defined thanks to Corollary  \ref{cor:cal}.
\end{dfn}
\color{black}
We note for diagonalizable operators $h$, our fixed non-commutative functional calculus is essentially the same as defined in the paper \cite{HigKas2}. Since, the arguments following the definition of a non-commutative functional calculus in \cite{HigKas2} are, fundamentally, about the diagonalizable operators, they are still valid without any change. Hence, we give here the important properties of a non-commutative functional calculus without any proof as is proven in \cite{HigKas2}.

\begin{prop}\label{property} A non-commutative functional calculus \ref{dfn:cal} has the following properties. (cf. \cite{HigKas2} Section 3.)
\begin{itemize}
\item For any $h$ satisfying the assumption of Definition \ref{dfn:cal}, $h(B_\alpha)$ is a symmetric operator defined on $s_\alpha(\hill_h)$;
\item The assignment $h\mapsto h(B_\alpha)$ is ``$\bbR$-linear'' (on the domain where the sum makes sense);
\item if $h$ is diagonalizable and $h=\sum^\infty_{k=1}\lambda_kP_{W_i}$, $h(B_\alpha)=\sum^\infty_{k=1}\lambda_kB_{W_k,\alpha}$; hence $h(B_\alpha)$ is diagonalizable and in particular, essentially selfadjoint; if $h$ has compact resolvent, so is $h(B_\alpha)$;
\item if $h$ is diagonalizable and $h^2\geq1$, $||h(B_\alpha)\xi||\geq||B_\alpha\xi||$ for any $\xi$ in $s_\alpha(\hill_h)$; hence the selfadjoint domain of $h(B_\alpha)$ is contained in that of $B_\alpha$, and this inequality extends to the selfadjoint domain of $h(B_\alpha)$;
\item if $h$ is an bounded operator, $||h(B_\alpha)\xi||\leq||h||||B_\alpha\xi||$ for any $\xi$ in $s_\alpha(\hill)=s_\alpha(\hill_h)$;
\item if $h_1, h_2$ are positive, diagonalizable operators which differ by a bounded operator (hence have their common domain $\hill_h$), and if $h_1^2,h_2^2\geq1$, $||h_1(B_\alpha)\xi-h_2(B_\alpha)\xi||\leq||h_1-h_2||||B_\alpha\xi||$ for any $\xi$ in $s_\alpha(\hill_h)$; and this inequality extend to the selfadjoint domain of $h_1(B_\alpha)$ or of $h_2(B_\alpha)$;
\item For two positive, diagonalizable operators $h_1,h_2$ having compact resolvent which differ by a bounded operator, if we set $B_{\alpha,1,t}=(1+th_1)(B_\alpha),  B_{\alpha, 2, t}=(1+th_2)(B_\alpha)$ for $t>0$, we have for any $f$ in $C_0(\bbR)$,
\begin{align}
\displaystyle \lim_{t\to0}\sup_{s>0,\alpha>0}||f(sB_{\alpha,1,t})-f(sB_{\alpha,2,t})||=0
\label{eq:perturb}
\end{align}
\item When a group $G$ acts on $\hill$ by linear isometries and if $h$ is a positive, diagonalizable operator having compact resolvent whose domain is $G$-invariant 
and if $g(h)-h$ is bounded for any $g$ in $G$, we set as above $B_{\alpha,t}=(1+th)(B_\alpha)$. Then we have for any $f$ in $C_0(\bbR)$ and for any $g$ in $G$,
\begin{align}
\displaystyle \lim_{t\to0}\sup_{s>0,\alpha>0}||f(sB_{\alpha,t})-g(f(sB_{\alpha,t}))||=0
\label{eq:perturb2}
\end{align}
\end{itemize}
\end{prop} 
We remark here that the equation \eqref{eq:perturb} describes the ``asymptotic behaviors'' of the perturbations $(1+th_1)(B_\alpha)$ and $(1+th_2)(B_\alpha)$ of $B_\alpha$ (which has compact resolvent) are ``close'' in some strong sense when $h_1$ and $h_2$ differ by a bounded operator. Also, the equation \eqref{eq:perturb2} says that such perturbations can be made to be ``asymptotically $G$-equivariant'' in some strong sense when one finds a good operator $h$  and uses it for the perturbation. As is proven in \cite{HigKas2} this is alway possible.

\begin{lemma}\label{adapt}(cf. \cite{HigKas2} Lemma 5.7.) Let $G$ be a second countable, locally compact group. Suppose $G$ acts on a real separable Hilbert space $\hill$ by affine isometries. Write the action of $G$ by $(\pi, b)$. Then, there exists a positive, diagonalizable operator $h$ on $\hill$ having compact resolvent whose domain is $G$-invariant and $\pi(g)h-h\pi(g)$ is bounded for any $g$ in $G$. We say such an operator $h$ is adapted to the action of $G$.
\end{lemma}
\begin{proof} We follow the argument as in \cite{HigKas2}. It actually proves the following: let $X$ and $Y$ be $\sigma$-compact subsets of $O(\hill)$ and $\hill$ respectively. Then there exists a positive, diagonalizable operator $h$ on $\hill$ having compact resolvent whose domain contains $Y$ and is $X$-invariant and $xh-hx$ is bounded for any $x\in X$. It is clear that this implies our stated claim. Now, we prove this. Write $X$ and $Y$ as increasing unions of compact sets $X_n$ and $Y_n$ ($n\geq1$) respectively. Take an increasing sequence of finite rank projections $(P_n)_{n\geq1}$ such that $\|(1-P_n)y\|\leq2^{-n}$ for $y$ in $Y_n$ and $\|(1-P_{n+1})xP_{n}\|\leq2^{-n}$ for $x\in X_n$. Set $P_0=P_{-1}=0$ and $Q_n=P_n-P_{n-1}$ for $n\geq0$. Define, a positive, diagonalizable operator $h$ by $h=\sum_{n\geq1}^{\infty}nQ_n$. It is clear that $h$ has compact resolvent and that the (selfadjoint) domain of $h$ contains $Y$. Take $x\in X$; we claim $xh-hx=\sum_{n\geq1}^{\infty}n(xQ_n-Q_nx)$ is a bounded operator. Write $xh-hx$ as the sum of $\sum_{n\geq1}^{\infty}n((1-P_{n+1})xQ_n-Q_nx(1-P_{n+1}))$ and $\sum_{n\geq1}^{\infty}n(P_{n-2}xQ_n-Q_nxP_{n-2})$ and $\sum_{n\geq1}^{\infty}n(Q_{n+1}xQ_n+Q_{n-1}xQ_n-Q_nxQ_{n+1}-Q_nxQ_{n-1})$. It is now clear each of them are bounded.
\end{proof} 

We now go to the definition of a continuous field of $\Calgs$ which is the key component of the construction of $G$-extension of the $\Calg$ $A(\hill)$ of Hilbert space. We will consider a bit more general situation than affine isometric actions of $G$.

\begin{dfn} Let $G$ be a second countable, locally compact group, $Y$ be a second countable, locally compact $G$-space, $\hill$ be a separable real Hilbert space. A continuous field of affine isometric actions of $G$ on $\hill$ (parametrized) over $Y$ is a pair $(\pi, (b_y)_{y\in Y})$ where $\pi\colon G\to O(\hill)$ is a continuous group homomorphism from $G$ to $O(\hill)$ and $(b_y)_{y\in Y}$ is a continuous map $(b_y) \colon G\times Y\to \hill$ satisfying a (twisted) cocycle condition $b_y(gg')=b_y(g)+\pi(g)b_{g^{-1}y}(g')$ for any $g,g'$ in $G$ and $y$ in $Y$. Given such a field $(\pi, (b_y)_{y\in Y})$, the $\Calg$ $A(\hill)(Y)=C_0(Y, A(\hill))$ becomes a $G$-$\Calg$ naturally: we set for any $g\in G$ and for $f\colon Y\to A(\hill)$, $g(f)(y)=(\pi(g), b_y(g))_\ast f(g^{-1}y)$ for $y\in Y$ where $(\pi(g), b_y(g))_\ast$ is an action on $A(\hill)$ induced by an affine isometric action $(\pi(g), b(g))$.  Also, we have a continuous field $(C_0(Y, H_\alpha(\hill)))_{\alpha\in (0, \infty)}$ of (graded) $G$-$C_0(Y)$-Hilbert modules.
\end{dfn}

Note, we may allow the linear part $\pi$ also vary along $Y$, but we will stick to the above simple case. For example, for any affine isometric action $(\pi, b)$ of $G$ on $\hill$, taking $Y$ as a (trivial) $G$-space $[0, 1]$, we have a continuous field $(\pi, (b_y)_{y\in[0,1]})$ of affine isometric actions of $G$ on $\hill$ over $[0, 1]$ with $b_y(g)=yb(g)$ which gives us a homotopy between the affine isometric action $(\pi, b)$ and the liner isometric action $(\pi, 0)$. More generally, for any continuous filed $(\pi, (b_y)_{y\in Y})$ of affine isometric actions of $G$ on $\hill$ over $Y$, we have a homotopy between $(\pi, (b_y)_{y\in Y})$ and $(\pi, (0)_{y\in Y})$.

In the following discussion of this chapter, we fix one continuous field $(\pi, (b_y)_{y\in Y})$ of a second countable, locally compact group $G$ on a separable real Hilbert space $\hill$ over a second countable, locally compact $G$-space $Y$.

\begin{dfn}(continuous field $(A_\alpha(\hill)(Y))_{\alpha\in[0,\infty)}$) (cf. \cite{HigKas2} Section 5.) We define a continuous field $(A_\alpha(\hill)(Y))_{\alpha\in[0,\infty)}$ of $G$-$\Calgs$ \,with fibers $A_0(\hill)(Y)=A(\hill)(Y),\,\, A_{\alpha}(\hill)(Y)= \cS\hat\otimes\K(H_\alpha(\cH))(Y)$ for $\alpha$ in $(0,\infty)$. For any finite dimensional affine subspace $V$ of $\hill$, a continuous field $(C_\alpha(V))_{\alpha\in[0,\infty)}$ of graded $\Calgs$ with fibers $C_0(V)=C(V)=C_0(V\times V_0)\hat\otimes L(V), C_{\alpha}(V)=\K(H(V))=\K(L^2(V))\hat\otimes L(V)$ for $\alpha$ in $(0,\infty)$ is defined by a (graded) tensor product of a (trivially graded) continuous field $(C^\ast_\alpha(V_0,C_0(V)))_{\alpha\in[0,\infty)}$ by a graded $\Calg$ $L(V)$: the continuous field $(C^\ast_\alpha(V_0,C_0(V)))_{\alpha\in[0,\infty)}$ is obtained as a  (reduced) crossed product of a ``constant'' field $(C_0(V))_{\alpha\in[0,\infty)}$ by an additive group $V_0$ whose action on the fiber $C_0(V)$ at $\alpha$ is induced from the translation action of $V_0$ on $V$ defined by $v_0\cdot v=v+\alpha v_0$. A continuous field $(A_\alpha(V)(Y))_{\alpha\in[0,\infty)}$ of $\Calgs$ with fibers $A_\alpha(V)(Y)=\cS\hat\otimes C_\alpha(V)(Y)$ for $\alpha$ in $[0,\infty)$ is obtained by a graded tensor product of the above continuous field $(C_\alpha(V))_{\alpha\in[0,\infty)}$ by a graded $\Calg$ $\cS$ and by a (ungraded) $\Calg$ $C_0(Y)$. With these in mind, continuous sections of $(A_\alpha(\hill)(Y))_{\alpha\in[0,\infty)}$ are defined as follows. Fix a positive, selfadjoint compact operator $h$ on $\hill$ which has compact resolvent and adapted to the ``actions'' $(\pi, b(y))$ for all $y$ in $Y$: this is possible; see the proof of Lemma \ref{adapt}. Denote as before the domain of $h$ by $\hill_h$. For any finite dimensional affine subspace $V$ of $\hill_h$, any continuous sections $(T_\alpha)_{\alpha\in[0,\infty)}$ of the continuous field $(C_\alpha(V)(Y))_{\alpha\in[0,\infty)}$ and for any $f$ in $\cS$, a basic section of $(A_\alpha(\hill)(Y))_{\alpha\in[0,\infty)}=(A_\alpha(V^\perp)\hat\otimes C_\alpha(V)(Y))_{\alpha\in[0,\infty)}$ associated to $(T_\alpha)_{\alpha\in[0,\infty)}$ and $f$ is defined as $f(X\hat\otimes1+1\hat\otimes \cB_{V^\perp})\hat\otimes T_0$ at $\alpha=0$ and $f(X\hat\otimes1+1\hat\otimes (1+\alpha h_V)(B_{V^\perp,\alpha}))\hat\otimes T_\alpha$ at $\alpha>0$. Here $V^\perp=\hill\ominus V_0$; $\cB_{V^\perp}$ and $B_{V^\perp,\alpha}$ are the Bott operator on $A(V^{\perp})$ and the Bott-Dirac operator on $V^\perp$ respectively; and $h_V$ is the compression of $h$ to $V^\perp$. A section of $(A_\alpha(\hill)(Y))_{\alpha\in[0,\infty)}$ is defined to be continuous if it is a uniform limit over compact subsets of basic sections. We denote by $\mathcal{F}_h$ the $\Calg$ of the continuous sections of $(A_\alpha(\hill)(Y))_{\alpha\in[0,\infty)}$ which vanish at infinity. On the one hand, the evaluation of the section algebra $\mathcal{F}_h$ at $\alpha=0$ gives a surjective homomorphism from $\mathcal{F}_h$ onto $A(\hill)(Y)$. On the other hand, the $\Calg$ of the continuous sections of $(A_\alpha(\hill)(Y))_{\alpha\in[0,\infty)}$ which vanish at $0$ and at infinity, i.e. the kernel of the evaluation of $\mathcal{F}_h$ at $\alpha=0$, is naturally isomorphic to the $\Calg$ $\cS\hat\otimes\K(\E)$ where $\E$ is a graded Hilbert $\Sigma(Y)$-module of the continuous sections of $(H_\alpha(\hill)(Y))_{\alpha\in(0,\infty)}$ which vanish at infinity. Hence, we have a following extension of $\Calgs$:
\begin{equation}\label{Gext}
\xymatrix{
0 \ar[r] & \cS\hat\otimes\K(\E) \ar[r] & \mathcal{F}_h \ar[r] & A(\hill)(Y) \ar[r] & 0
}
\end{equation}
As is proven in \cite{HigKas2}, the $\Calg$ $\mathcal{F}_h$ becomes a $G$-$\Calg$ naturally; though we are in a bit general situation, the proof goes verbatim. Hence, the extension \eqref{Gext} becomes a $G$-extension of $\Calgs$. We have a natural isomorphism $\cS\hat\otimes\K(\E)\cong S\otimes\K(\E)$; and this is even an isomorphism of $G$-$\Calgs$. Hence, we have actually a following extension ($G$-extension):
\begin{equation}\label{Gext2}
\xymatrix{
0 \ar[r] & S\otimes\K(\E) \ar[r] & \mathcal{F}_h \ar[r] & A(\hill)(Y) \ar[r] & 0
}
\end{equation}
\end{dfn}

We now come to the definition of two important asymptotic morphisms. As in \cite{HigKas2} (Definition 6.4.), we define for separable $G$-$\Calgs$ $A,B$, the abelian semigroup $\{A, B\}_G$ as a set of homotopy equivalence classes of asymptotic morphisms from $A$ to $\K(\E)$ for a countably generated Hilbert $G$-$B$-module $\E$. Here, the homotopy means the asymptotic morphism from $A$ to $\K(\E')$ for a countably generated Hilbert $G$-$B[0,1]$-module. Addition law for $\{A, B\}_G$ is induced from the direct sum operation for Hilbert $G$-$B$-module. As in Definition \ref{dfn:asym}, the semigroup $\{\Sigma A, B\}_G$ is an abelian group thanks to the presence of $\Sigma$.

\begin{dfn}(cf. \cite{HigKas2} Definition 6.6.) The dual Dirac element $\beta$ is the class in the group $\{S(Y), A(\hill)(Y)\}_G$ of the $G$-equivariant asymptotic morphism $(\phi_t)\colon S(Y)\to\to A(\hill)(Y)$ defined by $\phi_t(f\otimes f'):=f(t^{-1}\cB)\otimes f'$ for $t$ in $[1,\infty)$, for $f$ in $S$ and $f'$ in $C_0(Y)$.
\end{dfn}

\begin{dfn}(cf. \cite{HigKas2} Definition 6.7.) The Dirac element $\alpha$ is the class in the group $\{\Sigma A(\hill)(Y),S\Sigma(Y)\}_G$ defined by a central invariant of the extension \eqref{Gext2} (recall this asymptotic morphism is defined using some asymptotically equivariant continuous approximate unit of $S\otimes\K(\E)$ but its class is independent of choices).  
\end{dfn}

We also call the asymptotic morphisms defining the dual Dirac element $\beta$ and the Dirac element $\alpha$ as the dual Dirac element and the Dirac element respectively and even write them as $\alpha$ or $\beta$.

\begin{theorem}\label{theorem:comp}(cf. \cite{HigKas2} Theorem 6.10.) The composition of $G$-equivariant asymptotic morphisms $\Sigma\beta\colon\Sigma S(Y)\rightarrow\rightarrow \Sigma A(\hill)(Y)$ and $\alpha\colon\Sigma A(\hill)(Y)\rightarrow\rightarrow S\K(\E)$ represents the same class in the group $\{\Sigma S(Y), S\Sigma(Y)\}_G$ as the flip isomorphism $\Sigma S\to S\Sigma$ tensored with the identity $\id_{C_0(Y)}$.
\end{theorem}
\begin{proof} A homotopy of continuous fields of affine isometric actions between $(\pi, (b_y)_{y\in Y})$ and $(\pi, (0)_{y \in Y})$ evidently produce homotopy between the compositions of the dual Dirac elements and the Dirac elements corresponding to the two continuous fields of affine isometric actions. Hence, we can assume the affine part $(b_y)_{y\in Y}$ is $0$. In this case, $\phi_1\colon S(Y)\to A(\hill)(Y)$ is a $G$-equivariant homomorphism which, viewed as an equivariant asymptotic morphism, is homotopic to $(\phi_t)$. Therefore, by the naturality of central invariants, $\alpha\circ\Sigma\beta$ in $\{\Sigma S(Y), S\Sigma(Y\}_G$ is represented by the central invariant of the following pullback $G$-extension:
 \begin{align}\label{pulled}
 \xymatrix{
 0 \ar[r] & S\K(\E) \ar@{=}[d] \ar[r] & \mathcal{F}_{h,S} \ar[d] \ar[r] & S(Y) \ar[d]^{\phi_1} \ar[r] & 0 \\
  0 \ar[r] & S\K(\E) \ar[r] & \mathcal{F}_{h} \ar[r] & A(\hill)(Y) \ar[r] & 0  
 }
 \end{align}
 Here, $\mathcal{F}_{h, S}$ is the $G$-$\C$-subalgebra of $\mathcal{F}$ consisting of continuous sections $(a_\alpha)_{\alpha \in [0, \infty)}$ of the continuous field $(A_\alpha(\hill)(Y))_{\alpha \in [0, \infty)}$ vanishing at infinity taking values in $S(Y)\subset A(\hill)(Y)$ at $\alpha=0$. Modulo null sections, that is, the elements in $S\K(\E)$, this algebra is generated by basic sections associated to continuous sections $(T_\alpha)_{\alpha \in [0, \infty)}$ of the constant field $(C_0(Y))_{\alpha \in [0, \infty)}$ vanishing at infinity and $f\in S$. The functional calculus $f\mapsto f(X\hat\otimes1+1\hat\otimes (1+\alpha h)(B_{\alpha}))$ decomposes into the identity $f\mapsto f$ and the other part similarly to the one explained in Chapter 2. Therefore, the central invariant associated to the extension \eqref{pulled} is the sum of central invariants associated to the following two $G$-extensions of $S(Y)$:
 \begin{align}\label{flipext}
 \xymatrix{
 0 \ar[r] & S(0,\infty)(Y) \ar[r] &S[0, \infty)(Y) \ar[r] & S(Y) \ar[r] & 0 
 }
 \end{align}
and
 \begin{align}\label{nullext}
 \xymatrix{
 0 \ar[r] & PS\K(\E)P  \ar[r] & P\mathcal{F}_{h,S}P \ar[r] & S(Y)  \ar[r] & 0 
 }
 \end{align}
 Here, $P=(P_\alpha)$ denotes the (pointwise) orthogonal projection of the Hilbert space $H_\alpha(\hill)$ onto the subspace orthogonal to the one dimensional kernel of the Bott-Dirac operator $B_\alpha$ of $\hill$. Therefore, it suffices to show the central invariant associated to the extension \eqref{flipext} is $0$. As in \cite{HigKas2}, we define a $G$-$\Calg$ $\mathcal{D}$. To produce this, we consider a continuous field $(D_\alpha)_{\alpha\in [0, \infty)}$ with fibers $D_\alpha=P_\alpha A_\alpha(\hill)(Y)P_\alpha(0, 1]$ for $\alpha$ in $(0, \infty)$ and $D_0=S(Y)$. Continuous sections are generated by continuous sections of $(D_\alpha)_{\alpha \in (0 \infty)}$ vanishing at $0$ and infinity and by basic sections associated to continuous sections $(T_\alpha)_{\alpha \in [0, \infty)}$ of the constant field $(C_0(Y))_{\alpha \in [0, \infty)}$ and $f$ in $S$ which in tern defined as a section which is $f\otimes T_0$ at $\alpha=0$ and is a function $s\mapsto P_\alpha f(X\hat\otimes1+1\hat\otimes (1+\alpha h)(s^{-1}B_{\alpha}))P_\alpha\otimes T_\alpha$ at $\alpha>0$. Thanks to the last property listed in Proposition \ref{property}, the $\Calg$ $\mathcal{D}$ of continuous sections of $(D_\alpha)_{\alpha\in[0, \infty)}$ naturally becomes a $G$-$\Calg$. Moreover, we have a following diagram of $G$-extension:
 \begin{align*}
\xymatrix{ 
  0 \ar[r] & PS\K(\E)P(0, 1] \ar[d] \ar[r] & \mathcal{D} \ar[d] \ar[r] & S(Y) \ar@{=}[d] \ar[r] & 0 \\
   0 \ar[r] & PS\K(\E)P  \ar[r] & P\mathcal{F}_{h,S}P \ar[r] & S(Y)  \ar[r] & 0 
 }
 \end{align*} Here, the vertical arrows are the (fiberwise) evaluation at $s=1$ of $C_0(0, 1]$. By the naturality of central invariants, we see the central invariant of the $G$-extension \eqref{nullext} is $0$. 
 \end{proof}

As in \cite{HigKas2}, \,\,we \,want to\, compute\,\, ``the composition of asymptotic morphisms'' $\Sigma\beta\colon\Sigma S(Y)\to\to \Sigma A(\hill)(Y)$ and $\alpha\colon \Sigma A(\hill)(Y)\to\to SK(\E)$ in the other order to conclude $A(\hill)(Y)$ and $S(Y)$ are isomorphic in the equivariant $E$-Theory category $E^G$. We consider another continuous field over $[0, \infty)$ with fibers $A(\hill)\hat\otimes A_\alpha(\hill)(Y)$ for $\alpha$ in $(0, \infty)$ and $A(\hill\times\hill)(Y)$ at $\alpha=0$. The continuous sections of this field are generated by continuous sections of the field $(A(\hill)\hat\otimes A_\alpha(\hill)(Y))_{\alpha \in (0, \infty)}$ which vanish at $0$ and infinity and by basic sections associated to $f$ in $S$, $T$ in $C(V)$ and a continuous section $(T_\alpha)_{\alpha \in [0, \infty)}$ of the field $(C_\alpha(V)(Y))_{\alpha \in [0, \infty)}$ which vanish at infinity for a finite dimensional subspace $V$ of $\hill$ which are defined analogously as before. Denote by $\mathcal{F}$ the $G$-$\Calg$ of continuous sections of this field. Evaluation at $\alpha=0$ produces the following $G$-extension:
\begin{align}\label{hillhill}
\xymatrix{
0 \ar[r] & \K(\E') \ar[r] & \mathcal{F} \ar[r] & A(\hill\times\hill)(Y) \ar[r] & 0 
}
\end{align} Here, $\E'$ is a Hilbert $G$-$A(\hill)\Sigma(Y)$-module of continuous sections of $(A(\hill)\hat\otimes H_\alpha(\hill)(Y))_{\alpha\in(0,\infty)}$ which vanish at infinity. We denote a central invariant of the extension \eqref{hillhill} by $\zeta$. If we consider the equivariant asymptotic morphism $(\phi'_t)\colon A(\hill)(Y)\to\to A(\hill\times\hill)(Y)$ associated to the embedding of $\hill$ into the second factor of $\hill\times\hill$, a similar argument as before shows the composition $\zeta\circ\Sigma(\phi'_t)\colon \Sigma A(\hill)(Y)\to\to \K(\E')$ in the group $\{\Sigma A(\hill)(Y), A(\hill)\Sigma(Y)\}_G$ is the same as the one defined by the flip tensored with the identity $\id_{C_0(Y)}$: as explained in \cite{HigKas2}, one uses Atiyah's rotation trick flipping the Hilbert space $\hill\times\hill$. Now, we return to the asymptotic morphism $\alpha\colon \Sigma A(\hill)(Y)\to\to S\hat\otimes\K(\E)$. We can ``compose'' this with the equivariant asymptotic morphism $\beta_{p}\colon S\to\to A(\hill)$ (the dual Dirac element for $Y=$ point) tensored with $\id_{\K(\E)}$ to get an asymptotic morphism $\beta_{p}\hat\otimes\id_{\K(\E)}\circ\alpha\colon \Sigma A(\hill)(Y)\to\to\K(\E')$: note, here, we are using an isomorphism $\K(\E')\cong A(\hill)\hat\otimes\K(\E)$ of $\Calgs$ not of $G$-$\Calg$. It is easy to see the two asymptotic morphisms $\zeta\circ\Sigma(\phi'_t)$ and $\beta_{p}\hat\otimes\id_{\K(\E)}\circ\alpha$ are homotopic. Hence, we have a following result. 

\begin{theorem}(cf. \cite{HigKas2} Theorem 6.11.) The\ dual\, Dirac\, element $\alpha\colon S(Y) \to A(\hill)(Y)$\,\,\, defines\,\, an\, invertible\, morphism\,\, \[\id_\Sigma\otimes\alpha\otimes\id_{\K(\hill_G)}\colon \Sigma S(Y)\K(\hill_G)\to\to\Sigma A(\hill)(Y)\K(\hill_G)\] in $E^G(S(Y), A(\hill)(Y))$. Its inverse is $\beta\otimes\id_{\K(\hill_G)}\colon\Sigma A(\hill)(Y)\K(\hill_G)\to\to \K(\E)\K(\hill_G)\cong \Sigma S(Y)\K(\hill_G)$ defined by the the Dirac element $\beta\colon \Sigma A(\hill)(Y)\to\to \K(\E)$. In particular, $S(Y)$ and $A(\hill)$ are isomorphic in the Equivariant $E$-Theory category $E^G$.
\end{theorem}

\newpage
\section{Technical Part of the Higson-Kasparov Theorem}

In this chapter, we are going to discuss the technical part of the Higson-Kasparov Theorem. Throughout this chapter, we assume an additional assumption that the action of the group $G$ on the Hilbert space $\hill$ is metrically proper. Hence, the $\Calg$ $A(\hill)$ of the Hilbert space is a proper $G$-$\Calg$. What we will be concerned is how to lift the Dirac element $\alpha$ in the group $\{\Sigma A(\hill), S\Sigma\}_G$ to the group $KK^G_1(A(\hill), S\Sigma)$. In view of Proposition \ref{KasSka}, there is one obvious candidate in the group $KK^G_1(A(\hill), S\Sigma)_G$. In this chapter, $\K$ denotes the $G$-$\Calg$ $\K(\hill_G)$ of compact operators on the standard $G$-Hilbert space $\hill_G=L^2(G)\otimes l^2$.

N. Higson and G. Kasparov defined very natural group homomorphisms\,\, $\eta$ from $KK^G_0(A, B)$ to $\{\Sigma^2A, B\}_G$ and $KK^G_1(A, B)$ to $\{\Sigma A, B\}_G$ for any separable $G$-$\Calgs$ $A$ and $B$ (we use the same notation $\eta$ for these two homomorphisms). Also, they defined left inverses $\rho$ of the homomorphisms $\eta$ when $A=\bbC$.  We first recall the definition of the homomorphisms $\eta$ and $\rho$.

\begin{dfn}\label{eta} (cf. \cite{HigKas2} Definition 7.2.) We define the homomorphism $\eta$ from $KK^G_0(A, B)$ to $\{\Sigma^2 A, B\}_G$ as follows. Let $x$ be an element in the group $KK^G_0(A, B)$. Suppose $x$ is represented by a cycle $(\E, \phi, F)$; recall that $\E$ is a countably generated Hilbert $G$-$B$-module, $\phi$ is an equivariant $\ast$-homomorphism from $A$ to $B(\E)$, and $F$ is an operator in $B(\E)$ which is essentially unitary, essentially equivariant and essentially commuting with elements in $A$. We obtain an equivariant $\ast$-homomorphism $\phi'$ from $\Sigma A$ to $Q(\E)$ defined by $\phi'\colon f\otimes a \mapsto f(F)\phi(a)$ for $f$ in $\Sigma\cong C_0(S^1-\{1\})$ and $a$ in $A$ (We omit to write the quotient map from $B(\E)$ to $Q(\E)$). We define $\eta(x)$ to be an element in the group $\{\Sigma^2A, B\}_G$ represented by a central invariant for the following pullback extension of $\Sigma A$ by $\K(\E)$ defined by $\phi'$:
\begin{align*}
\xymatrix{
0 \ar[r] & \K(\E) \ar[d]\ar[r] & E_{\phi'} \ar[d]\ar[r] & \Sigma A \ar[d]^-{\phi'}\ar[r] & 0\\
0 \ar[r] & \K(\E) \ar[r] & B(\E) \ar[r]   & Q(\E) \ar[r] & 0\\
}
\end{align*}
The definition of the homomorphism $\eta$ from $KK^G_1(A, B)$ to $\{\Sigma A, B\}_G$ is similar but simpler.  Let $x$ be an element in the group $KK^G_1(A, B)$ which is represented by a cycle $(\E, \phi, P)$; recall this time, $P$ is an operator in $B(\E)$ which is essentially an projection, essentially equivariant and essentially commuting with elements in $A$. We obtain an equivariant $\ast$-homomorphism $\phi'$ from $A$ to $Q(\E)$ defined by $\phi'\colon a\mapsto \phi(a)P$ for $a$ in $A$. We define $\eta(x)$ to be an element in the group $\{\Sigma A, B\}_G$ represented by a central invariant for the following pullback extension of $A$ by $\K(\E)$ defined by $\phi'$:
\begin{align*}
\xymatrix{
0 \ar[r] & \K(\E) \ar[d]\ar[r] & E_{\phi'} \ar[d]\ar[r] & A \ar[d]^-{\phi'}\ar[r] & 0\\
0 \ar[r] & \K(\E) \ar[r] & B(\E) \ar[r]   & Q(\E) \ar[r] & 0\\
}
\end{align*}
\end{dfn}

The defined homomorphisms $\eta$ behave well with the Bott-Periodicity, tensor products with the identity morphisms and Stabilization.

\begin{lemma} (cf. \cite{HigKas2} Lemma 7.3.) For any separable $G$-$\Calgs$ $A,B$, the following diagram commutes up to sign.
\begin{align*}
\xymatrix{
KK^G_1(\Sigma A, B) \ar[d]^{\eta} \ar@{=}[r] & KK^G_0(A, B) \ar[d]^{\eta} \\
\{\Sigma^2A, B\}_G \ar@{=}[r] & \{\Sigma^2A, B\}_G
}
\end{align*}
Here, the top horizontal equality means the natural isomorphism by the Bott Periodicity which is unique up to sign.
\end{lemma}
\begin{proof} Lex $x=(\E, \phi, P)$ be an element of $KK^G_1(\Sigma A, B)$ with $\phi=\phi_\Sigma\otimes\phi_A$ a nondegenerate representation of $\Sigma A$ on $\E$. As is explained in Example \ref{exampleBott}, the Bott periodicity maps this element to $y=(\E, \phi_A, e^{2i\pi x}P+1-P)$ in $KK^G_0(A, B)$ where $x$ is $\phi_\Sigma(x)$ (recall that we extend the nondegenerate representation $\phi_\Sigma$ of $\Sigma$ to that of $C_b(0,1)$). That the homomorphisms $\eta$ send two elements to the same class in $\{\Sigma^2A, B\}_G$ can be seen as follows. The $\ast$-homomorphism from $\Sigma A$ to $Q(\E)$ defining $\eta(x)$ is $f\otimes a\mapsto \phi_\Sigma(f)\phi_A(a)P$. On the other hand, $\ast$-homomorphism defining $\eta(y)$ is $f\otimes a\mapsto f(e^{2i\pi x}P+1-P)\phi_A(a)$ where $f$ is in $C_0(S^1-1)\cong\Sigma$. For any $f$ in $C_0(S^1-1)$ and $a$ in $A$, we have $f(e^{2i\pi x}P+1-P)\phi_A(a)=f(e^{2i\pi x})P$ in the Calkin algebra $Q(\E)$.  We can now see that the two $\ast$-homomorphisms are actually the same via the identification $\Sigma\cong C_0(S^1-1)$ given by a homeomorphism $x\mapsto e^{2i\pi x}$ from $(0,1)$ to $S^1-1$.
\end{proof}

\begin{lemma} For any separable $G$-$\Calgs$ $A,B$ and $C$, the following diagram commutes for $\ast=0,1$.
 \begin{align*}
\xymatrix{
KK^G_\ast(A, B) \ar[d]^{\eta} \ar[r]^-{\sigma_C} & KK^G_\ast(A\otimes C, B\otimes C) \ar[d]^{\eta} \\
\{\Sigma^{2-\ast}A, B\}_G \ar[r]^-{\sigma_C} & \{\Sigma^{2-\ast}A\otimes C, B\otimes C\}_G
}
\end{align*}
\end{lemma}
\begin{proof} We consider the case $\ast=1$. Take any element $x=(\E, \phi, P)$ in $KK^G_1(A, B)$. Then, the element $\sigma_C(x)$ in $KK^G_1(A\otimes C, B\otimes C)$ is by definition, represented as $(\E\otimes C, \phi\otimes \id_C, P\otimes1)$. On the one hand, the element $\sigma_C(\eta(x))$ in the $\{\Sigma A\otimes C, B\otimes C\}_G$ is represented by an asymptotic morphism $\phi_t\colon f\otimes a\otimes c\mapsto (f(u_t)\otimes1)(\phi(a)P\otimes c) \in \K(\E)\otimes C$ where $(u_t)_{t\geq1}$ is an approximate unit for the pair $\K(\E)\subset E_{\phi'}$. On the other hand, $\eta(\sigma_C(x))$ is represented by an asymptotic morphism $\psi_t\colon f\otimes a\otimes c\mapsto f(v_t)(\phi(a)P\otimes c) \in \K(\E)\otimes C$ where $(v_t)_{t\geq1}$ is an approximate unit for the pair $\K(\E)\otimes C\subset E_{\phi'\otimes\id_C}$. They are homotopic via the straight line homotopy between $(u_t)$ and $(v_t)$. The case $\ast=0$ can be handled completely analogously.
\end{proof}

\begin{lemma}\label{lemStab} For any separable $G$-$\Calgs$ $A,B$, the following diagram commutes for $\ast=0,1$.
 \begin{align*}
\xymatrix{
KK^G_\ast(A, B) \ar[d]^{\eta} \ar@{=}[r] & KK^G_\ast(A, B\K) \ar[d]^{\eta} \\
\{\Sigma^{2-\ast}A, B\}_G \ar@{=}[r] & \{\Sigma^{2-\ast}A, B\K\}_G
}
\end{align*} Here, the top inequality is Stabilization in Equivariant $KK$-Theory: that is the Kasparov product with $(\K(\hill_G, \bbC), 1, 0)$ in $KK^G(\bbC, \K)$ or with its inverse $(\hill_G, \id_\K, 0)$ in $KK^G(\K, \bbC)$. The meaning of the bottom inequality is similar: in rightward direction, for any Hilbert $G$-$B$-module $\E$, we identify $\K(\E)$ with $\K(\E\otimes\K(\hill_G, \bbC))$; and for any Hilbert Hilbert $G$-$B\K$-module $\E'$, we identify $\K(\E')$ with $\K(\E\otimes_\K\hill_G)$.
\end{lemma}
\begin{proof} This is more or less obvious. We just need to see for any Hilbert $G$-$B$-module $\E$, we have an isomorphism from $\K(\E)$ to $K(\E\otimes\K(\hill_G, \bbC))$ sending $T$ to $T\otimes1$.
\end{proof}

\begin{dfn}\label{dfnrho}(cf. \cite{HigKas2} Definition 7.4.) We define the homomorphisms $\rho$ from $\{\Sigma^2, B\}_G$ to $KK^G_1(\bbC, B(0, \infty))\cong KK^G_0(\bbC, B)$ as follows. Given any asymptotic morphism $\phi=(\phi_t)_{t\geq1}$ from $\Sigma^2$ to $\K(\E)$, we view $\phi$ as a map from $\Sigma^2$ to a $\Calg$ $K(\E)(0,\infty)$ by deeming $\phi_t=t\phi_1$ for $t<1$. By extending $\phi$ to a unital map from $\tilde\Sigma^2$ to $B(\E)(0,\infty)$ and naturally extending it to a map $\phi'$ from $M_2(\tilde\Sigma^2)$ to $B(\E\oplus\E)(0,\infty)$, we now define an operator $P$ in $B(\E\oplus\E)(0,\infty)$ to be an image $\phi'(p)$ of a projection $p=\frac{1}{1+|z|^2}\begin{pmatrix}
1 & z \\
\bar{z} & |z|^2
\end{pmatrix}$ in $M_2(\tilde\Sigma^2)\cong M_2(\widetilde{C_0(\bbC)})$. An operator $P$ is an essentially projection, which is essentially equivariant. We set $\eta(\phi)$ in $KK^G_1(\bbC B)$ to be the odd Kasparov module $((\E\oplus\E)(0, \infty), 1, P)$.\\
We next define $\rho$ from $\{\Sigma, B\}_G$ to $KK^G_0(\bbC, B(0, \infty))\cong KK^G_1(\bbC, B)$. Given any asymptotic morphism $\phi=(\phi_t)_{t\geq1}$ from $\Sigma$ to $\K(\E)$, we view $\phi$ as a map from $\Sigma$ to a $\Calg$ $K(\E)(0,\infty)$ by deeming $\phi_t=t\phi_1$ for $t<1$. By extending $\phi$ to a unital map $\phi'$ from $\tilde\Sigma$ to $B(\E)(0,\infty)$, we define an operator $T$ in $B(\E)(0,\infty)$ to be an image $\phi'(z)$ of a unitary $z$ in $\tilde\Sigma\cong C(S^1)$. An operator $T$ is an essentially unitary which is essentially equivariant. We set $\eta(\phi)$ in $KK^G_0(\bbC, B(0, \infty))$ to be the even Kasparov module $(\E(0, \infty), 1, T)$.
\end{dfn}

The defined homomorphisms $\rho$ are right inverses of $\eta$.

\begin{lemma}\label{rhoinverse} (cf. \cite{HigKas2} Lemma 7.5.) For any separable $G$-$\Calg$ $B$, the following composition for $\ast=0,1$
\begin{align*}
\xymatrix{
KK^G_\ast(\bbC, B) \ar[r]^-{\eta} &  \{\Sigma^{2-\ast}, B\}_G \ar[r]^-{\rho} & KK^G_{1-\ast}(\bbC, B(0, \infty)) 
}
\end{align*} coincides with the Bott-Periodicity map (up to sign). Hence, $\rho$ is a right inverse of $\eta$.
\end{lemma}
\begin{proof} We consider the case $\ast=1$. Take any element $x=(\E, 1, P)$ in $KK^G_1(\bbC, B)$. The asymptotic morphism $\phi_t\colon f\mapsto f(u_t)P$ represents $\eta(x)$ in $\{\Sigma, B\}_G$. Here, $(u_t)$ is an asymptotically equivariant, approximate unit asymptotically commuting with $P$. Then, the map $\phi'\colon\tilde\Sigma\to B(\E)(0, \infty)$ as in Definition \ref{dfnrho} sends $z$ in $\tilde\Sigma\cong C(S^1)$ to an essentially unitary $T_t=e^{i2\pi u_t}P+I-P$ in $B(\E)(0, \infty)$ with $u_t=tu_1$ for $t\leq1$ (we are identifying $\Sigma$ with $C_0(S^1-1)$ using a homeomorphism $x\mapsto e^{i2\pi x}$ from $(0,\infty)$ to $S^1-1$). The straight line homotopy between $u_t$ and $\min(t, 1)$ shows, the even Kasparov module $(\E(0, \infty), 1, T_t)$ defines the class in $KK^G_0(\bbC, B(0, \infty))$ which corresponds to $x$ under the Bott Periodicity. The case $\ast=0$ is similar but a bit complicated. Take any element $y=(\E, 1, F)$ in $KK^G_0(\bbC, B)$. The asymptotic morphism $\phi_t\colon f\otimes f'\mapsto f(u_t)f'(F)$ represents $\eta(y)$ in $\{\Sigma^2, B\}_G=\{\Sigma C_0(S^1-1), B\}_G$. Here, $(u_t)$ is an asymptotically equivariant, approximate unit asymptotically commuting with $F$. Then, the map $\phi'\colon M_2(\tilde\Sigma^2)\to B(\E\oplus\E)(0,\infty)$ as in Definition \ref{dfnrho} sends $p=\frac{1}{1+|z|^2}\begin{pmatrix}
1 & z \\
\bar{z} & |z|^2
\end{pmatrix}$ in $M_2(\tilde\Sigma^2)\cong M_2(\widetilde{C_0(\bbC)})$ to an essentially projection which is homotopic to $P_t=\begin{pmatrix}
1-u_t &  (u_t-{u_t}^2)^{\frac12}F^\ast\\
(u_t-{u_t}^2)^{\frac12}F   & u_t
\end{pmatrix}$ in $M_2(B(\E)(0, \infty))$ with $u_t=tu_1$ for $t\leq1$). The straight line homotopy between $u_t$ and $\min(t, 1)$ shows, the odd Kasparov module $(\E(0, \infty)\oplus\E(0, \infty), 1, P_t)$ defines the class in $KK^G_1(\bbC, B(0, \infty))$ which corresponds to $y$ under the Bott Periodicity. \end{proof}

The following lemmas show the homomorphisms $\rho$ also behave well with the Bott-Periodicity and Stabilization.

\begin{lemma} (cf. \cite{HigKas2} Lemma 7.6.) For any separable $G$-$\Calg$ $B$, the following diagram commutes.
\begin{align*}
\xymatrix{
\{\Sigma, B\}_G \ar[d]^{\rho} \ar[r]^-{\sigma_\Sigma} & \{\Sigma^2, \Sigma B\}_G \ar[d]^\rho \\
KK^G_0(\bbC, B(0, \infty)) \ar@{=}[r] & KK^G_1(\bbC, \Sigma B(0, \infty))
}
\end{align*} The bottom inequality is of course, the Bott Periodicity. 
\end{lemma}
\begin{proof} Take any $x$ in $\{\Sigma, B\}_G$ represented by an asymptotic morphism $\phi_t\colon C_0(S^1-1)\to \K(\E)$. We may write an essential unitary on $\E(0, \infty)$ defining $\rho(x)$ in $KK^G_0(\bbC, B(0, \infty))$ as $T_t=\phi_t(z)$. Now, the homomorphism $\sigma_\Sigma$ sends $\phi_t$ to an asymptotic morphism $\id_\Sigma\otimes\phi_t\colon f\otimes f'\mapsto f\otimes \phi_t(f')$. The homomorphism $\rho$ sends this asymptotic morphism to an essential projection homotopic to $P_t=\begin{pmatrix}
1-x &  (x-{x}^2)^{\frac12}\phi_t(\overline{z})\\
(x-{x}^2)^{\frac12}\phi_t(z)   & x
\end{pmatrix}$ on $(\Sigma\otimes\E\oplus\Sigma\otimes\E)(0,\infty))$ which defines clearly the element in $KK^G_1(\bbC, \Sigma B(0,\infty))$ corresponding to $\rho(x)$. 
\end{proof}

\begin{lemma} For any separable $G$-$\Calgs$ $B$, the following diagram commutes for $\ast=0,1$.
 \begin{align*}
\xymatrix{
\{\Sigma^{2-\ast}, B\}_G \ar[d]^{\rho} \ar@{=}[r] & \{\Sigma^{2-\ast}, B\K\}_G \ar[d]^\rho \\
KK^G_{1-\ast}(\bbC, B(0, \infty))  \ar@{=}[r] & KK^G_{1-\ast}(\bbC, B\K(0, \infty))
}
\end{align*} Here, the top and bottom equalities are analogous to those of Lemma \ref{lemStab}.
\end{lemma}
\begin{proof} This is again, more or less obvious. It is clear that for any $x$ in $\{\Sigma^{2-\ast}, B\}_G$, if we denote by $x'$ the corresponding element in $\{\Sigma^{2-\ast}, B\K\}_G$, $\rho(x')$ is just the Kasparov product of $\rho(x)$ with $(K(\hill_G, \bbC), 1, 0)$.
\end{proof}

\begin{prop} (cf. \cite{HigKas2} Lemma 8.4.) Consider the Bott element $b$ in $KK^G_1(\bbC, A(\hill))$. We have $-\eta(b)=\beta$ in the group $\{\Sigma, A(\hill)\}_G$. Here, we consider the dual Dirac element $\beta$ in $\{S, A(\hill)\}_G$ as an element in $\{\Sigma, A(\hill)\}_G$ using an (order preserving) homeomorphism between the real line $\bbR$ and the interval $(0, 1)$.
\end{prop}
\begin{proof} We use a homeomorphism $x\mapsto \frac{x(x^2+1)^{-\frac12}+1}{2}$ from $\bbR$ to $(0, 1)$. Recall that the Bott-element $b$ is represented by an essentially projection $P=\frac{\cB(1+\cB^2)^{-\frac12}+1}{2}$ in $A(\hill)$. We set $P_t=\frac{t^{-1}\cB(1+t^{-2}\cB^2)^{-\frac12}+1}{2}$. Then, the dual Dirac element $\beta$ in $\{\Sigma, A(\hill)\}_G$ is represented by an asymptotic morphism $f\mapsto f(P_t)$. On the other hand, $-\eta(b)$ is represented by $f\mapsto f(1-u_t)P_1$ where $(u_t)$ is an approximate unit in $A(\hill)$ which is asymptotically equivariant and quasi-central with respect to $P$. This asymptotic homomorphism is homotopic to an asymptotic morphism $f\mapsto f((1-u_t)^{\frac12}P_1(1-u_t)^{\frac12})$.  The latter is homotopic to an asymptotic morphism $f\mapsto f(1-u_t)^{\frac12}P_{s(t)}(1-u_t)^{\frac12}$ with suitably slowly increasing function $s$ on $[1,\infty)$ onto $[1, \infty)$. The straight line homotopy between $1$ and $u_t$ followed by a reparametrization connects this asymptotic morphism to the one $f\mapsto f(P_t)$.
\end{proof}

In order to get the Dirac element in Equivariant $KK$-Theory, we must lift the Dirac element $\alpha$ in the group $\{\Sigma A(\hill), S\Sigma\}_G$ to $KK^G_1(A(\hill), S\Sigma)$. The following ensures that this is possible.

\begin{theorem}\label{Result} (cf. \cite{GHT} Chapter 9.) Let $A,B$ be separable $G$-$\Calgs$. Suppose $A$ is a proper $G$-$\Calg$. Then, the abelian group $\{\Sigma A, B\}_G$ is naturally isomorphic to the abelian group $[[\Sigma A, B\K]]_G$. Suppose further that $A$ is nuclear and that $B$ is isomorphic to $\Sigma B'$ for some separable $G$-$\Calg$ $B'$. Then, the homomorphism $\eta\colon KK^G_1(A, B)\to \{\Sigma A, B\}_G$ is an isomorphism of abelian groups.
\end{theorem}
\begin{proof} We first prove a natural group homomorphism $\iota\colon[[\Sigma A,B\K]]_G \to \{\Sigma A, B\}_G$ (a map gained by regarding $B\K$ as $\K(B\otimes \hill_G)$) is an isomorphism when $A$ is a proper $G$-$\Calg$. In fact, we prove the natural map $\iota\colon[[A,B\K]]_G \to \{A, B\}_G$ is a bijection of sets (or of semigroups) when $A$ is proper. Let $\sigma_{\K}$ be a map from $\{A, B\}_G$ to $[[A\K, B\K]]_G$ which sends the class represented by an asymptotic morphism $\phi\colon A\to\to\K(\E)$ to the class represented by an asymptotic morphism $\phi\otimes\id_\K\colon A\K\to\to\K(\E)\K\to B\K$ (the last map is induced by any adjointable isometry $\E\otimes\hill_G\to B\otimes\hill_G$ of $G$-$B$-Hilbert modules) and $\kappa_A$ be a map from $[[A\K, B\K]]_G$ to $[[A, B\K]]_G$ given by the composition with a stabilization homomorphism $\Ad_{V}\colon A\to A\K$ induced by some adjointable isometry $V\colon A\to A\otimes\hill_G$ which exists since $A$ is proper (see Proposition \ref{prop:stab}). Note that the map $\sigma_\K$ is defined independently of choices of an isometry $\E\otimes\hill_G\to B\otimes\hill_G$ since any such isometry are (equivariantly) homotopic to each other (in the $\ast$-strong topology); similarly, $\sigma_A$ is defined independently of choices of an adjointable isometry $V\colon A\to A\otimes\hill_G$. We claim that the map $\kappa_A\circ\sigma_{\K}\colon\{A, B\}_G\to[[A\K, B\K]]_G\to[[A, B\K]]_G$ is the inverse of $\iota$. For later use, we prove three maps $\iota, \kappa_A, \sigma_{\K}$ are all bijective. That the maps $\kappa_A\circ\sigma_{\K}\circ\iota$ and $\sigma_{\K}\circ\iota\circ\kappa_A$ are the identities on $[[A, A\K]]_G$ and on $[[A\K, B\K]]_G$ respectively is essentially, the Stabilization in Equivariant $E$-Theory which is explained in \cite{GHT}.

$\iota\circ\kappa_A\circ\sigma_{\K}$ is the identity on $\{A, B\}_G$: Take any ($G$-equivariant) asymptotic morphism $\phi\colon A\to\to\K(\E)$ where $\E$ is a countably generated $G$-$B$-Hilbert module. The homomorphism $\iota\circ\kappa_A\circ\sigma_{\K}$ sends the class represented by an asymptotic morphism $\phi$ to the class represented by the asymptotic morphism $\phi\otimes\id_\K\circ\kappa_A\colon A\to A\K \to\to \K(\E)\K$. Let $\Ad_{V_s}\colon A\to A\K(\hill_G\oplus \bbC)$ ($s\in[0, 1]$) be a homotopy between the stabilization $\Ad_{V_0}=\Ad_{V}\colon A\to A\K$ and the identity $\Ad_{V_1}=\id_A\colon A\to A$ induced by a homotopy $V_s\colon A\to A\otimes(\hill_G\oplus\bbC)$ of (adjointable) isometries $V_0=V\colon A\to A\otimes\hill_G$ and $V_1=\id_A\colon A\to A\otimes\bbC$ of Hilbert $G$-$A$-modules which can be defined by $V_s=(1-s)^{\frac12}V_0\oplus s^{\frac12}V_1$ for example. The homotopy of asymptotic morphisms $\phi\otimes\id_{\K(\hill_G\oplus\bbC)}\circ\Ad_{V_s}\colon A\to A\K(\hill_G\oplus\bbC)\to\to\K(\E)\K(\hill_G\oplus\bbC)$ ($s\in[0, 1]$) connects the two asymptotic morphisms $\phi$ $(s=1)$ and $\phi\otimes\id_\K\circ\kappa_A$ $(s=0)$. This shows $\iota\circ\kappa_A\circ\sigma_{\K}$ is the identity on $\{A,B\}_G$.

$\kappa_A\circ\sigma_{\K}\circ\iota$ is the identity on $[[A, B\K]]_G$: The proof is almost identical as above. Take any asymptotic morphism $\phi\colon A\to\to B\K$. The map  $\kappa_A\circ\sigma_{\K}\circ\iota$ sends the class represented by $\phi$ to the one represented by $\phi\otimes\id_\K\circ\Ad_V\colon A\to A\K \to\to B\K\K$. We use the homotopy of  $\Ad_{V_s}\colon A\to A\K(\hill_G\oplus\bbC)$ defined above. The homotopy of asymptotic morphisms $\phi\otimes\id_{\K(\hill_G\oplus\bbC)}\circ\Ad_{V_s}\colon A\to A\K(\hill_G\oplus\bbC)\to\to B\K\K(\hill_G\oplus\bbC)$ ($s\in[0, 1]$) connects the two asymptotic morphisms $\phi$ $(s=1)$ and $\phi\otimes\id_\K\circ\Ad_V$ $(s=0)$. This shows $\kappa_A\circ\sigma_{\K}\circ\iota$ is the identity on $[[A, B\K]]_G$.

$\sigma_{\K}\circ\iota\circ\kappa_A$ is the identity on $[[A\K, B\K]]_G$: Take any asymptotic morphism $\phi\colon A\K\to\to B\K$. The map $\sigma_{\K}\circ\iota\circ\kappa_A$ sends the class represented by $\phi$ to the one represented by $(\phi\circ\Ad_V)\otimes\id_\K=\phi\otimes\id_\K\circ\Ad_V\otimes\id_\K\colon A\K \to A\K\K \to B\K\K$. We use the homotopy of  $\Ad_{V_s}\colon A\to A\K(\hill_G\oplus\bbC)$ again. The homotopy of asymptotic morphisms $\phi\otimes\id_{\K(\hill_G\oplus\bbC)}\circ\Ad_{V_s}\otimes\id_\K\colon A\K\to A\K(\hill_G\oplus\bbC)\K=A\K\K(\hill_G\oplus\bbC)\to\to B\K\K(\hill_G\oplus\bbC)$ ($s\in[0, 1]$) connects the two asymptotic morphisms $\phi\otimes\id_\K$ $(s=1)$ and $\phi\otimes\id_\K\circ\Ad_V\otimes\id_\K$ $(s=0)$, but the asymptotic morphisms $\phi$ and $\phi\otimes\id_\K$ are homotopic via the homotopy $\phi\otimes\id_{\K(\hill_G\oplus\bbC)}\circ\id_A\otimes\Ad_{W_{s'}}\colon A\K\to A\K\K(\hill_G\oplus\bbC)\to\to B\K\K(\hill_G\oplus\bbC)$ $(s'\in [0, 1])$ where $W_{s'}\colon \hill_G\to\hill_G\otimes(\hill_G\oplus\bbC)$ is any homotopy of isometries of $G$-Hibert spaces between $W_0\colon\hill_G\cong\hill_G\otimes\hill_G\xhookrightarrow{}\hill_G\otimes(\hill_G\oplus\bbC)$ and $W_1\colon\hill_G\cong\hill_G\otimes\bbC\xhookrightarrow{}\hill_G\otimes(\hill_G\oplus\bbC)$. This shows $\sigma_{\K}\circ\iota\circ\kappa_A$ is the identity on $[[A\K, B\K]]_G$.

Now, suppose further that $A$ is nuclear and that $B$ is isomorphic to $\Sigma B'$. We have the following commutative diagram of abelian groups.
\begin{align}
\xymatrix{
KK^G_1(A, B) \ar[d]^-{\eta} \ar[r]^-{\sigma_\Sigma} & KK^G_1(\Sigma A, \Sigma B) \ar[d]^-{\eta} \ar[r]^-{\sigma_\K} & KK^G_1(\Sigma A\K, \Sigma B\K) \ar[d]^-{\eta} \\
\{\Sigma A, B\}_G \ar[d]^-{\sigma_\K} \ar[r]^-{\sigma_\Sigma} & \{\Sigma^2 A, \Sigma B\}_G \ar[d]^-{\sigma_\K} \ar[r]^-{\sigma_\K} & \{\Sigma^2A\K, \Sigma B\K\}_G \ar[d]^-{\sigma_\K} \\
[[\Sigma A\K, B\K]]_G \ar@{=}[d]^-{} \ar[r]^-{\sigma_\Sigma} & [[\Sigma^2 A\K, \Sigma B\K]]_G \ar@{=}[d]^-{} \ar[r]^-{\sigma_\K} & [[\Sigma^2A\K, \Sigma B\K]]_G \ar@{=}[d]^-{} \\
E^G(A, B') \ar[r]^-{\sigma_\Sigma} & E^G(\Sigma A, B) \ar@{=}[r]^-{} & E^G(\Sigma A, B)
}
\label{diagram:proof of theorem}
\end{align}
In the diagram \eqref{diagram:proof of theorem}, we know all the indicated arrows are isomorphisms except those indicated by $\eta$ and $\sigma_\Sigma\colon\{\Sigma A, B\}_G\to\{\Sigma^2A, \Sigma B\}_G$. However, the composition $\sigma_\K\circ\eta\circ\sigma_\K\circ\sigma_\Sigma\colon KK^G_1(A, B) \to E^G(\Sigma A, B)$ is the natural isomorphism according to Corollary \ref{important}. It follows that all $\eta$ in the diagram \eqref{diagram:proof of theorem} are isomorphisms. In particular, the homomorphism $\eta\colon KK^G_1(A, B)\to \{\Sigma A, B\}_G$ is an isomorphism. Note, it follows $\sigma_\Sigma\colon\{\Sigma A, B\}_G\to\{\Sigma^2A, \Sigma B\}_G$ is also an isomorphism.
\end{proof}

\begin{dfn} (Compare with \cite{HigKas2} Definition 8.2.) We  define the Dirac element $d$ in $KK^G_1(A(\hill), S\Sigma)$ to be the unique element which corresponds to the Dirac element $\alpha$ in $\{\Sigma A(\hill), S\Sigma\}_G$ via the isomorphism $\eta\colon KK^G_1(A(\hill), S\Sigma)\to \{\Sigma A(\hill), S\Sigma\}_G$.
\end{dfn}

The following theorem is the heart of the Higson-Kasparov Theorem.

\begin{theorem}\label{HigKasTech} (cf. \cite{HigKas2} Theorem 7.8.) Let $A$ be a separable proper $G$-$\Calg$ and $B$ be a separable $G$-$\Calg$. Then, the following diagram commutes up to sign. \begin{align}
\xymatrix{
KK^G_1(\bbC, A) \ar[d]^-{\eta}_{\eta}  \times   KK^G_1(A, B)  \ar[r] & KK^G_0(\bbC, B)\\
[[\Sigma, A\K]]_G \ar[d]^-{\sigma_\K}_{\sigma_{\Sigma}}  \times  \{\Sigma A, B\}_G \ar[d] & \{\Sigma^2, B\}_G \ar[u]_-{\rho}\\
[[\Sigma^2, \Sigma A\K]]_G   \times  [[\Sigma A\K, B\K]]_G \ar[r] & [[\Sigma^2, B\K]]_G \ar[u]\\
}
\label{diagram:tech}
\end{align} Here, the homomorphism $\eta\colon KK^G_1(\bbC, A)\to \{\Sigma, A\}_G$ is naturally considered as the map from $KK^G_1(\bbC, A)$ to $[[\Sigma, A\K]]_G\cong\{\Sigma, A\}_G$ and the top (or the bottom) horizontal arrow is the Kasparov product (or the composition of asymptotic morphisms).
\end{theorem}  
\begin{proof} Since $\eta$ and $\rho$ are compatible with stabilization, it suffices to show when $A\cong A\K$ and $B\cong B\K$ (i.e. when $A,B$ are stable). This ensures that we need to only consider elements of the form $(A, 1, P)$ in $KK^G_1(\bbC, A)$. Also in this case, $\sigma_\K$ is the identity on $\{\Sigma A, B\}_G$. Hence, it suffices to show $\eta(y)\circ\sigma_\Sigma(\eta(x))=\eta(x\otimes_A y)$ in $\{\Sigma^2, B\}_G$ for any $x=(A, 1, P)$ in $KK^G_1(\bbC, A)$ and $y=(\E, \phi, Q)$ in $KK^G_1(A, B)$. Here, $x\otimes_Ay$ denotes the Kasparov product of $x$ and $y$ in $KK^G_0(\bbC, B)$.  The rest of the proof would be identical to the one given in \cite{HigKas2}. 
\end{proof}

Theorem\ \ref{HigKasTech} \,enables\ us\ to\ compute\ the\ composition\ of\ the\ Bott\ element\ $b$\, in $KK^G_1(\bbC, A(\hill))$ and the Dirac element $d$ in $KK^G_1(A(\hill), S\Sigma)$.

\begin{theorem}(cf. \cite{HigKas2} Theorem 8.5.)\label{KKcomp} The composition $b\otimes_{A(\hill)}d$ in $KK^G(\bbC, S\Sigma)$ coincides with the identity in $KK^G(\bbC, \bbC)$ up to sign under the Bott Periodicity $KK^G(\bbC, \bbC)\cong KK^G(\bbC, S\Sigma)$. 
\end{theorem} 
\begin{proof}
We only need to check the composition $\sigma_\K(\eta(d))\circ\sigma_\Sigma(\eta(b))$ coincides with $\alpha\circ\Sigma\beta$ up to sign in $\{\Sigma^2, S\Sigma\}_G \cong \{\Sigma S, S\Sigma \}_G$, but at the level of asymptotic morphisms, the first element is $-\Sigma\beta\colon \Sigma S\to\to \Sigma A(\hill)$ composed with the composition of the stabilization $\Sigma A(\hill)\to \Sigma A(\hill)\K$ and $\sigma_\K(\alpha)$ which is homotopic to $\alpha$.
\end{proof}

The Dual-Dirac method (Theorem \ref{DD}) says the Baum-Connes conjecture with coefficients holds for $G$, if the identity in $KK^G(\bbC, \bbC)$ factors through a proper algebra. Thus, we finally finish the proof of the Higson-Kasparov Theorem.

\begin{theorem} (cf. \cite{HigKas2} Theorem 9.1.) The Baum-Connes conjecture with coefficients holds for all a-$T$-menable groups.
\end{theorem}

\newpage
\section{Non-Isometric Actions}

In this last chapter, we consider an affine action of a second countable, locally compact group $G$ on a separable (infinite-dimensional) real Hilbert space $\hill$ whose linear part is not necessarily isometric. It has been suggested that it is important to consider such an action since some groups like $\text{sp}(n,1)$ which cannot admit metrically proper, affine isometric action on Hilbert space (due to the Kazhdan's property-($T$)) admits a metrically proper affine action whose linear part is not isometry but uniformly bounded. However, in order to carry out some analogy of the argument of the Higson-Kasparov Theorem to this case, it is necessary to go beyond the framework of $\Calgs$. For example, the $\Calg$ $A(\hill)$ of Hilbert space would not become a $G$-$\Calg$ in an obvious way. We will see, however, if we consider an affine action of $G$ whose linear part is of the form an isometry times a scalar, then there indeed exists a natural action of $G$ on the $\Calg$ $A(\hill)$ which makes it a $G$-$\Calg$.\\

In this chapter, by an affine action of a group $G$ on a Hilbert space $\hill$, we mean an affine action $(\pi\times r, b)$ of $G$ on $\hill$ whose linear part $\pi\times r$ is of the form an isometry times a scalar. Namely, $\pi$ and $r$ are continuous group homomorphisms from $G$ to $O(\hill)$ and to $\bbR_+$ respectively; and $b$ is a continuous map from $G$ to $\hill$ satisfying the cocycle condition $b(gg')=\pi(g)r(g)b(g')+b(g)$ for any $g,g'$ in $G$. We denote by $g$ the affine transformation given by $g$; i.e. the homeomorphism $v\mapsto \pi(g)r(g)v+b(g)$ of $\hill$.

Now, let $(\pi\times r, b)$ be an affine action of a group $G$ on a Hilbert space $\hill$. Then, we have a natural action of $G$ on a $\Calg$ $A(\hill)$ of Hilbert space which makes it a $G$-$\Calg$. The $G$-action is defined as follows. For $g$ in $G$ and for a finite dimensional affine subspace $V$ of $\hill$, we have a $\ast$-isomorphism $g$ from $A(V)=\mathcal{S}\hat\otimes C_0(V\times V_0, \mathcal{L}(V))$ to $A(gV)$ which decomposes as an action $r(g)_\ast$ on $\cS$ defined by $r(g)$,  isomorphisms $g_{\ast}\colon C_0(V)\to C_0(gV)$ and $(\pi(g)r(g))_\ast\colon C_0(V_0)\to  C_0(gV_0)=C_0(\pi(g)r(g)V_0)$ defined by $g$ and by $\pi(g)r(g)$ respectively and an isomorphism from $\pi(g)_\ast\colon\mathcal{L}(V)\to \cL(gV)=\cL(\pi(g)V)$ defined by $\pi(g)$. The next lemma says that this defines a $G$-action on $A(\hill)$.

\begin{lemma}\label{commute} For any element $g$ in $G$ and for any finite dimensional affine subspaces $V\subseteq V'=V\oplus W$, the following diagram commutes:
\begin{align*}
\xymatrix{
A(V) \ar[d]^-{g} \ar[r] & A(V') \ar[d]^-{g} \\
A(gV)  \ar[r] & A(gV') \\
}
\end{align*} Here, the horizontal maps are the natural inclusion; and the vertical maps are the maps defined above.
\end{lemma}
\begin{proof} The proof is identical to the isometric case. It suffices to show the following diagram commutes:
\begin{align*}
\xymatrix{
\cS \ar[d]^-{r(g)_\ast} \ar[r] & A(W) \ar[d]^-{(\pi(g)r(g))_\ast} \\
\cS  \ar[r] & A(\pi(g)r(g)W) \\
}
\end{align*} Rewrite $\pi(g)r(g)W$ as $W'$. We need to check the commutativity of the diagram only for $\exp(x^2)$ and for $x\exp(x^2)$ in $\cS$. Both routes send $\exp(x^2)$ to \begin{equation*}
\exp(r(g)^{-2}x^2)\hat\otimes\exp(r(g)^{-2}\|(w_1', w_2')\|^2)
\end{equation*}
in $A(W')=\cS\hat\otimes C_0(W'\times W', \cL(W'))$, and similarly send $x\exp(x^2)$ to \begin{align*}
r(g)^{-1}x\exp(r(g)^{-2}x^2)\hat\otimes\exp(r(g)^{-2}\|(w_1', w_2')\|^2)\\+\exp(r(g)^{-2}x^2)\hat\otimes r(g)^{-1}\cB_{W'}\exp(r(g)^{-2}\|(w_1', w_2')\|^2)
\end{align*}
in $A(W')$. Here, $\cB_{W'}$ is the Bott operator for $W'$. 
\end{proof}

\begin{dfn} For any affine action $(\pi\times r, b)$ on $\hill$, we define the $G$-action on $A(\hill)$ which is guaranteed by Lemma \eqref{commute}. This makes $A(\hill)$ a $G$-$\Calg$. 
\end{dfn}

Now, denote by $S_G$, the (ungraded) $G$-$\Calg$ $S$ with the $G$-action coming from the homomorphism $r\colon G\to\bbR_+$. Then, the natural inclusion $S_G \to A(\hill)$ is $G$-equivariant when the affine part $b$ of the $G$-action on $\hill$ is zero. In general, analogously to the case of isometric actions, we have an equivariant asymptotic morphism $(\phi_t)\colon S_G\to\to A(\hill)$ given by $\phi_t(f):=f(t^{-1}\cB)$ for $t$ in $[1,\infty)$. We will prove the following (though very little) generalization of infinite dimensional Bott Periodicity by N. Higson, G. Kasparov and J. Trout (see \cite{HKT}).

\begin{theorem} An equivariant asymptotic morphism $(\phi_t)\colon S_G\to\to A(\hill)$ defines an invertible morphism in $E^G(S_G, A(\hill))$.
\end{theorem}
\begin{proof} There might be a direct proof of this, but we will soon see that this result follows from an already established result: the infinite dimensional Bott Periodicity for a continuous field of affine isometric actions on real Hilbert spaces. We use Fell's absorption technique. Denote by $S_T$ the $G$-$\Calg$ $S$ equipped with the $G$-action induced by the translation action on $\bbR$ defined by $g\colon y\to y+\log(r(g))$ for $g$ in $G$. Since $S_T^2$ is isomorphic to $\bbC$ in Equivariant $E$-Theory, our claim follows if we show an equivariant asymptotic morphism $(\phi_t)\otimes\id_{S_T}\colon S_GS_T\to\to A(\hill)S_T$ defines an invertible morphism in $E^G(S_GS_T, A(\hill)S_T)$. Now, the $G$-$\Calg$ $S_GS_T$ is isomorphic to $SS_T$: write $S_GS_T$ as $C_0(\bbR_T, S_G)$ and $SS_T$ as $C_0(\bbR_T, S)$ when $\bbR_T$ is equipped with the translation action defined above. The isomorphism sends a function $f\colon\bbR_T\to S_G$ to a function $\bbR_T\ni y\mapsto (\exp(-y))_\ast(f(y))\in S$. Similarly, write $A(\hill)S_T$ as $C_0(\bbR_T, A(\hill))$. Use exactly the same formula; namely, send a function $f\colon\bbR_T\to A(\hill)$ to a function $\bbR_T\ni y\mapsto (\exp(-y))_\ast(f(y))\in A(\hill)$ where $(\exp(-y))_\ast$ denotes now an action on $A(\hill)$ defined by $\exp(-y)$ in $\bbR_+$. Then, this defines an isomorphism from $G$-$\Calg$ $A(\hill)S_T$ to the $G$-$\Calg$ which we write by $A(\hill)(\bbR_T)$ by the slight abuse of notation. The $G$-action on the latter algebra $A(\hill)(\bbR_T)$ is defined as follows. For $g$ in $G$ and for $f\colon\bbR_T\to A(\hill)$, $g(f)(y)=(\pi(g), \exp(-y)b(g))_\ast f (y-\log r(g))$ where $(\pi, b)_\ast$ denotes the action on $A(\hill)$ induced from an affine isometric action $(\pi, b)$ on $\hill$. Rewrite $SS_T$ as $S(\bbR_T)$. With these identifications the asymptotic morphism $(\phi_t)\otimes\id_{S_T}\colon S(\bbR_T)\to\to A(\hill)(\bbR_T)$ is nothing but the asymptotic morphism associated to the continuous filed of affine isometric actions $(\pi, (b_y)_{y\in \bbR_T})$ over $\bbR_T$ with $b_y=\exp(-v)b(g)$. In Chapter 7, we already proved that this defines invertible morphism in the category $E^G$; hence, we are done.
\end{proof}
 
One might want to consider whether we can do some analogy of the Higson-Kasparov Theorem in this situation. Namely, one may want to consider an affine action of a group $G$ on a Hilbert space $\hill$ which is metrically proper which makes $A(\hill)$ a proper $G$-$\Calg$ and see whether there is a Bott element and a Dirac element in Equivariant Kasparov's category. However, it is not enlightening to do so. (When such an action is metrically proper, it is more or less an isometric affine action.)  In stead, one should consider the action such that the $G$-$\Calg$ $A(\hill)$ becomes a proper $G$-$\Calg$ after tensoring $S_T$ as above. In such a situation, it is highly likely that one can do the exact analogy of the Higson-Kasparov Theorem to deduce $G$ satisfies BCC. Whether or not, there is a non a-$T$-menable group $G$  which admits such a ``proper'' action is not clear to the author's knowledge, but it would be just an extension of some subgroup of $\bbR_+$ by an a-$T$-menable group. As we remarked at the outset of this chapter, it is definitely an interesting and important problem to find the analogy of the Higson-Kasparov Theorem for more general affine actions of a group on a Hilbert space. The author considers in attacking this interesting problem, our Theorem \ref{Result} or the idea behind its proof could play some important role. 
 
\newpage

\bibliography{bib1}
\end{document}